\documentclass[11pt]{amsart}
\usepackage{amssymb}
\usepackage{amsmath}
\usepackage{amsfonts}
\usepackage{graphicx}

\usepackage[total={17cm,22cm},top=2.5cm, left=2.3cm]{geometry}
\usepackage{hyperref}
    \usepackage{aeguill}
    \usepackage{type1cm}

\theoremstyle{plain}
\newtheorem{thm}{Theorem}[section]

\newtheorem{claim}[thm]{Claim}

\newtheorem{corollary}[thm]{Corollary}

\newtheorem{example}[thm]{Example}

\newtheorem{lemma}[thm]{Lemma}
\newtheorem{lem}[thm]{Lemma}

\newtheorem{proposition}[thm]{Proposition}
\newtheorem{remark}[thm]{Remark}

\newtheorem{theorem}[thm]{Theorem}
\newtheorem{fact}[thm]{Fact}

\newtheorem{rem}[thm]{Remark}

\newtheorem{defn}[thm]{Definition}
\numberwithin{equation}{section}
\newcommand{\N}{\mathbb{N}}

\newcommand{\R}{\mathbb{R}}

\usepackage[usenames,dvipsnames]{color}

\usepackage[dvipsnames]{xcolor}

\begin{document}

\title[Smooth extractions and approximations without critical points]{Smooth approximations without critical points of continuous mappings between Banach spaces, and diffeomorphic extractions of sets}

\author{Daniel Azagra}
\address{ICMAT (CSIC-UAM-UC3-UCM), Departamento de An{\'a}lisis Matem{\'a}tico y Matem\'atica Aplicada,
Facultad Ciencias Matem{\'a}ticas, Universidad Complutense, 28040, Madrid, Spain.  {\bf DISCLAIMER:} The
first-named author is affiliated with Universidad Complutense de Madrid, but this does not mean this institution
has offered him all the support he expected. On the contrary, the Biblioteca Complutense has hampered his research
by restricting his access to books. } \email{azagra@mat.ucm.es}

\author{Tadeusz Dobrowolski}
\address{Department of Mathematics, Pittsburgh State Uuniversity,
1701 South Broadway Street, Pittsburg, KS 66762, United States of America} \email{tdobrowolski@pittstate.edu}

\author{Miguel Garc\'{\i}a-Bravo}
\address{ICMAT (CSIC-UAM-UC3-UCM), Calle Nicol\'as Cabrera 13-15.
28049 Madrid, Spain} \email{miguel.garcia@icmat.es}

\date{November 2018}

\keywords{Banach space, Morse-Sard theorem, approximation, critical point, diffeomorphic extraction}

\thanks{D. Azagra and M. Garc\'ia-Bravo were partially supported by grant
MTM2015-65825-P}

\subjclass[2010]{46B20, 46E50, 46B25, 46B28, 58B10}

\begin{abstract}
Let $E$, $F$ be separable Hilbert spaces, and assume that $E$ is infinite-dimensional. We show that for every
continuous mapping $f:E\to F$ and every continuous function $\varepsilon: E\to (0, \infty)$ there exists a
$C^{\infty}$ mapping $g:E\to F$ such that $\|f(x)-g(x)\|\leq\varepsilon(x)$ and $Dg(x):E\to F$ is a surjective
linear operator for every $x\in E$. We also provide a version of this result where $E$ can be replaced with a
Banach space from a large class (including all the classical spaces with smooth norms, such as $c_0$, $\ell_p$ or
$L^{p}$, $1<p<\infty$), and $F$ can be taken to be any Banach space such that there exists a bounded linear
operator from $E$ onto $F$. In particular, for such $E, F$, every continuous mapping $f:E\to F$ can be uniformly
approximated by smooth open mappings. Part of the proof provides results of independent interest that improve some
known theorems about diffeomorphic extractions of closed sets from Banach spaces or Hilbert manifolds.
\end{abstract}

\maketitle

\section{Introduction and main results}

The main purpose of this paper is to show the following two results.

\begin{thm}\label{Main result for separable Hilbert}
Let $E$, $F$ be separable Hilbert spaces, and assume that $E$ is infinite-dimensional. Then, for every continuous
mapping $f:E\to F$ and every continuous function $\varepsilon: E\to (0, \infty)$ there exists a $C^{\infty}$
mapping $g:E\to F$ such that $\|f(x)-g(x)\|\leq\varepsilon(x)$ and $Dg(x):E\to F$ is a surjective linear operator
for every $x\in E$.
\end{thm}

\begin{thm}\label{Main result for classical Banach}
Let $E$ be one of the classical Banach spaces $c_0$, $\ell_p$ or $L^{p}$, $1<p<\infty$. Let $F$ be a Banach space,
and assume that there exists a bounded linear operator from $E$ onto $F$. Then, for every continuous mapping
$f:E\to F$ and every continuous function $\varepsilon: E\to (0, \infty)$ there exists a $C^{k}$ mapping $g:E\to F$
such that $\|f(x)-g(x)\|\leq\varepsilon(x)$ and $Dg(x):E\to F$ is a surjective linear operator for every $x\in E$.
\end{thm}
Here $k$ denotes the order of smoothness of the space $E$, defined as follows: $k=\infty$ if $E\in
\{c_0\}\cup\{\ell_{2n} : n\in\N\}\cup\{L^{2n} : n\in\N\}$; $k=2n+1$ if $E\in\{\ell_{2n+1} : n\in\N\}\cup\{L^{2n+1}
: n\in\N\}$, and $k$ is equal to the integer part of $p$ if $E\in \{\ell_p\}\cup\{L^{p}\}$ and $p\notin\N$. The
Sobolev spaces $W^{k,p}(\R^n)$ with $1<p<\infty$ are also included in Theorem \ref{Main result for classical
Banach} since they are isomorphic to $L^p(\R^n)$ (see \cite[Theorem 11]{PelWoj}).

Notice that the assumption that there exists a bounded linear operator from $E$ onto $F$ is necessary, as otherwise
all points of $E$ are critical for all functions $g\in C^{1}(E, F)$.

Of course Theorem \ref{Main result for separable Hilbert} is a particular case of Theorem \ref{Main result for
classical Banach} (also note that if $E$ is a separable Hilbert space, $F$ is a Banach space, and there exists a
continuous linear surjection $T:E\to F$, then $F$ must be isomorphic to $\R^n$ or to $E$). In general, note that a
continuous linear surjection $T:E\to F$ between Banach spaces exists if and only if $F$ is isomorphic to a quotient
space of $E$.

We will also establish more technical results (see Theorems \ref{Main theorem for reflexive spaces} and \ref{Main
theorem for spaces with unconditional bases} below) that generalize the preceding theorems to much larger classes
of Banach spaces (especially in the case that $E$ is reflexive).

Part of the motivation for this kind of results is in their connection with the Morse-Sard theorem, a fundamental
result in Differential Geometry and Analysis. Throughout this paper, for a $C^k$ smooth mapping
$f:\mathbb{R}^{n}\longrightarrow \mathbb{R}^{m}$, $C_f$ stands for the set of critical points of $f$ (that is, the
points $x\in\R^n$ at which the differential $Df(x)$ is not surjective), and $f(C_{f})$ is thus the set of critical
values of $f$; the same terminology applies to smooth mappings between manifolds, both finite and
infinite-dimensional. The Morse-Sard theorem \cite{Morse, Sard} states that if $k\geq \max\{n-m+1, 1\}$ then
$f(C_{f})$ is of Lebesgue measure zero in $\mathbb{R}^{m}$. This result also holds true for $C^k$ smooth mappings
$f:N\longrightarrow M$ between two smooth manifolds of dimensions $n$ and $m$ respectively.

Given the crucial applications of the Morse-Sard theorem in several branches of mathematics, it is natural both to
try to extend this result for other classes of mappings, and also to ask what happens in the case that $M$ and $N$
are infinite-dimensional manifolds. Regarding the first issue, many refinements of the Morse-Sard theorem for other
classes of mappings (notably H\"older, Sobolev, and BV mappings) have appeared in the literature; see for instance
\cite{Whitney1, Yomdin, Norton, Bates, BatesMoreira, Moreira, DePascale, Figalli, BKK1, BKK2, KK1, HaZi, Hkk,
AzagraFerreraGomez, AzagraGarciabravo} and the references therein.

As for the second issue, which in this paper is of our concern, let us mention the results of several authors who
have studied the question as to what extent one can obtain results similar to the Morse-Sard theorem for mappings
between infinite-dimensional Banach spaces or manifolds modeled on such spaces.

S. Smale \cite{Smale} proved that if $X$ and $Y$ are separable connected smooth manifolds modeled on Banach spaces
and $f:X\longrightarrow Y$ is a $C^r$ Fredholm mapping (that is, every differential $Df(x)$ is a Fredholm operator
between the corresponding tangent spaces) then $f(C_{f})$ is meager, and in particular $f(C_{f})$ has no interior
points, provided that $r>\max\{\textrm{index}(Df(x)), 0\}$ for all $x\in X$; here index($Df(x)$) stands for the
index of the Fredholm operator $Df(x)$, that is, the difference between the dimension of the kernel of $Df(x)$ and
the codimension of the image of $Df(x)$, both of which are finite. Of course, these assumptions are very
restrictive as, for instance, if $X$ is infinite-dimensional then no function $f:X\longrightarrow\mathbb{R}$ is
Fredholm.

In general, every attempt to adapt the Morse-Sard theorem to infinite dimensions will have to impose vast
restrictions because, as shown by Kupka's counterexample \cite{Kupka}, there are $C^\infty$ smooth functions
$f:\ell_2\longrightarrow\mathbb{R}$ so that their sets of critical values $f(C_{f})$ contain intervals.
Furthermore, as shown by Bates and Moreira in \cite{BatesMoreira, Moreira}, one can take $f$ to be a polynomial of
degree $3$.

Nevertheless, for many applications of the Morse-Sard theorem, it is often enough to know that any given continuous
mapping can be uniformly approximated by a mapping whose set of critical values is small in some sense; therefore
it is natural to ask what mappings between infinite-dimensional manifolds will at least have such an approximation
property. Going in this direction, Eells and McAlpin established the following theorem \cite{EellsMcAlpin}: If $E$
is a separable Hilbert space, then every continuous function from $E$ into $\mathbb{R}$ can be uniformly
approximated by a smooth function $f$ whose set of critical values $f(C_{f})$ is of measure zero. This allowed them
to deduce a version of this theorem for mappings between smooth manifolds $M$ and $N$ modeled on $E$ and a Banach
space $F$ respectively, which they called an {\em approximate Morse-Sard theorem}: Every continuous mapping from
$M$ into $N$ can be uniformly approximated by a smooth mapping $f:M\longrightarrow N$ so that $f(C_{f})$ has empty
interior. However, as observed in \cite[Remark 3A]{EellsMcAlpin}, we have $C_{f}=M$ in the case that $F$ is
infinite-dimensional (so, even though the set of critical values of $f$ is relatively small, the set of critical
points of $f$ is huge, which is somewhat disappointing).

In \cite{AzagraCepedello}, a much stronger result was obtained by M. Cepedello-Boiso and the first-named author: if
$M$ is a $C^\infty$ smooth manifold modeled on a separable infinite-dimensional Hilbert space $X$, then every
continuous mapping from $M$ into $\mathbb{R}^{m}$ can be uniformly approximated by smooth mappings {\em with no
critical points}. P. H\'{a}jek and M. Johanis \cite{HJ} established a similar result for $m=1$ in the case that $X$
is a separable Banach space which contains $c_0$ and admits a $C^p$-smooth bump function. Finally, in the case that
$m=1$, these results were extended by M. Jim\'enez-Sevilla and the first-named author \cite{AzagraJimenez} for continuous
functions $f:X\to\R$, where $X$ is a separable Banach space admitting an equivalent smooth and locally uniformly
rotund norm.

In this paper, we will improve these results by showing that the pairs $(\ell_2, \R^m)$ or $(X, \R)$ can be
replaced with pairs of the form $(E, F)$, where $E$ is a Banach space from a large class (including all the
classical spaces with smooth norms such as $c_0$, $\ell_p$ or $L^{p}$, $1<p<\infty$), and $F$ can be taken to be
any quotient space of $E$. So we may say that even though an exact Morse-Sard theorem for mappings between
classical Banach spaces is false, a stronger approximate version of the Morse-Sard theorem is nonetheless true.

The general plan of the proof of Theorem \ref{Main result for classical Banach} consists in following these steps:
\begin{itemize}
\item Step 1:  We construct a smooth mapping $\varphi:E\to F$ such that
$\|\varphi(x)-f(x)\|\leq\varepsilon(x)/2$ and $C_{\varphi}$, the critical set of $\varphi$, is locally contained in
the graph of a continuous mapping defined on a complemented subspace of infinite codimension in $E$ and taking
values in its linear complement.
\item Step 2: We find a diffeomorphism $h:E\to E\setminus C_{\varphi}$ such
that $h$ is sufficiently close to the identity, in the sense that $\left\{ \{x, h(x) \} : x\in E \right\}$ refines
$\mathcal{G}$ (in other words, $h$ is {\em limited by} $\mathcal{G}$), where $\mathcal{G}$ is an open cover of $E$
by open balls $B(z, \delta_z)$ chosen in such a way that if $x, y\in B(z, \delta_z)$ then
$$
\|\varphi(y)-\varphi(x)\|\leq \frac{\varepsilon(z)}{4}\leq\frac{\varepsilon(x)}{2}.
$$
The existence of such a diffeomorphism $h$ follows by the results of Section 2.
\item Step 3: Then, the mapping $g(x):=\varphi(h(x))$ has no critical point and satisfies
$\|f(x)-g(x)\|\leq\varepsilon(x)$ for all $x\in E$.
\end{itemize}

The results of Section 2 are of independent interest, as they generalize important theorems on diffeomorphic {\em
extractions} of some kind of sets. Although it is well known (see \cite{BurKui, Moulis, EE, Elworthy} and the
references therein) that every two separable, homotopy equivalent, infinite-dimensional Hilbert manifolds $M$, $N$
are in fact diffeomorphic, a diffeomorphism $h:M\to N$ provided by this deep result has not been (and, in general,
cannot be) shown to be limited by an arbitrary open cover $\mathcal{G}$ of $M$, a property that is essential in
Step 3 above. The finest result we know of which provides a diffeomorphism $h:E\to E\setminus X$ limited by a given
open cover $\mathcal{G}$ of $E$, where $E$ is a separable infinite-dimensional Hilbert space $E$ and $X$ is a
closed subset of $E$, is a theorem of J.E. West \cite{West} in which $X$ is assumed to be locally compact. However,
in the proof of Theorem \ref{Main result for classical Banach},  we do not work necessarily with Hilbert spaces and
we need to diffeomorphically extract a closed set $X$ which is not necessarily locally compact but merely locally
contained in the graph of a continuous mapping defined on a complemented subspace of infinite codimension in $E$
and taking values in its linear complement (for a precise explanation of this terminology, see the statement of
Theorem \ref{final extractibility theorem rough version general form} below). In Section 2, we construct diffeomorphisms $h$ which extract such closed sets $X$.

The main result of Section 2 is the following.
\begin{thm}\label{final extractibility theorem rough version}
Let $E$ be an infinite-dimensional Hilbert space, $X$ a closed subset of $E$ which is locally contained in the
graph of a continuous function defined on a subspace of infinite codimension in $E$ and taking values in its
orthogonal complement, $\mathcal{G}$ an open cover of $E$, and $U$ an open subset of $E$. Then, there exists a
$C^{\infty}$ diffeomorphism $h$ of $E\setminus X$ onto $E\setminus (X\setminus U)$ which is the identity on
$(E\setminus U)\setminus X$ and is limited by $\mathcal{G}$.
\end{thm}
Recall that $h$ is said to be limited by $\mathcal{G}$ provided that the set $\left\{ \{x, h(x)\} : x\in E\setminus
X \right\}$ refines $\mathcal{G}$; that is, for every $x\in E\setminus X$, we may find a $G_x\in\mathcal{G}$ such
that both $x$ and $h(x)$ are in $G_x$.

Theorem \ref{final extractibility theorem rough version} is a straightforward consequence of the following much
more general result, which is true for many Banach spaces not necessarily Hilbertian.

\begin{thm}\label{final extractibility theorem rough version general form}
Let $E$ be a Banach space, $p\in\N\cup\{\infty\}$, and $X\subset E$ be a closed set with the property that, for
each $x\in X$, there exist a neighborhood $U_x$ of $x$ in $E$, Banach spaces $E_{(1,x)}$ and $E_{(2,x)}$, and a
continuous  mapping $f_{x}:C_x\to E_{(2, x)}$, where $C_x$ is a closed subset of $E_{(1, x)}$, such that:
\begin{enumerate}
\item $E=E_{(1,x)}\oplus E_{(2,x)}$;
\item $E_{(1, x)}$ has $C^p$ smooth partitions of unity;
\item $E_{(2, x)}$ is infinite-dimensional and has a (not necessarily equivalent)
norm of class $C^{p}$;
\item $X\cap U_x\subset G(f_x)$, where
$$
G(f_x)=\{ y=(y_1, y_2)\in E_{(1, x)}\oplus E_{(2, x)} \, : \, y_2= f_{x}(y_1), y_{1}\in C_x\}.
$$
\end{enumerate}
Then, for every open cover $\mathcal{G}$ of $E$  and every open subset $U$ of $E$, there exists a $C^{p}$
diffeomorphism $h$ from $E\setminus X$ onto $E\setminus (X\setminus U)$ which is the identity on $(E\setminus
U)\setminus X$ and is limited by $\mathcal{G}$. Moreover, the same conclusion is true if we replace $E$ with an
open subset of $E$.
\end{thm}

The proof of Theorem \ref{final extractibility theorem rough version general form} combines ideas and techniques
from Peter Renz's Ph.D. thesis \cite{Renz}, James West's paper \cite{West}, and some previous work of the first and
second-named authors \cite{Azagra, AzagraDobrowolski}; see Section 2 for more information. It should be noted that (even in the case that $E$ is separable)
Theorem \ref{final extractibility theorem rough version} generalizes West's theorem \cite{West}, because a closed
locally compact subset of an infinite-dimensional Hilbert space $E$, locally, can be regarded as the graph of a
continuous mapping defined on a closed subset of an infinite-codimensional subspace of $E$; see, for instance,
\cite{Renz}. Furthermore, note in the above results we do not assume separability of the
Banach space $E$.

The proof of Theorem \ref{Main result for separable Hilbert} will show that $E$ and $F$ can be replaced with open
subsets $U$ and $V$ of $E$ and $F$ respectively. Then, by combining such an equivalent statement of Theorem
\ref{Main result for separable Hilbert} with the well known result \cite{EE, Kuiper} stating that every separable
infinite-dimensional Hilbert manifold is diffeomorphic to an open subset of $\ell_2$, one may easily deduce the
following.
\begin{thm}\label{Main result for separable Hilbert manifolds}
Let $M$, $N$ be separable infinite-dimensional Hilbert manifolds. For every continuous mapping $f:M\to N$ and every
open cover $\mathcal{U}$ of $N$, there exists a $C^{\infty}$ mapping $g:M\to N$ such that $g$ has no
critical point and $\left\{\{ f(x), g(x)\} : x\in M\right\}$ refines $\mathcal{U}$.
\end{thm}
Alternatively, one can also adjust the proof of Theorem \ref{Main result for separable Hilbert} to obtain a direct
proof of Theorem \ref{Main result for separable Hilbert manifolds}.

\medskip

It is worth noting that Theorems \ref{Main result for separable Hilbert} and \ref{Main result for classical Banach} are immediate consequences of the following more general (but also more technical) results. For spaces $E$ which are reflexive and
have a certain ``composite'' structure, we have the following.

\begin{thm}\label{Main theorem for reflexive spaces}
Let $E$ be a separable reflexive Banach space of infinite dimension, and $F$ be a Banach space. In the case that
$F$ is infinite-dimensional, let us assume furthermore that:
\begin{enumerate}
\item $E$ is isomorphic to $E\oplus E$.
\item There exists a linear
bounded operator from $E$ onto $F$ (equivalently, $F$ is a quotient space of $E$).
\end{enumerate}
Then, for every continuous mapping  $f:E\to F$ and every continuous function $\varepsilon: E\to (0, \infty)$ there
exists a $C^{1}$ mapping $g:E\to F$ such that $\|f(x)-g(x)\|\leq\varepsilon(x)$ and $Dg(x):E\to F$ is a surjective
linear operator for every $x\in E$.
\end{thm}

Note that there exists separable, reflexive Banach spaces $E$ such that $E$ is not isomorphic to $E\oplus E$. The
first example of such a space was given by Figiel in 1972 \cite{Fie}.

See also Theorem \ref{result for recursively reflexively composite spaces} and Theorem
\ref{result for reflexive complemented spaces} below for more general variants of this result.

For spaces which are not necessarily reflexive but have an appropriate Schauder basis we have the following.

\begin{thm}\label{Main theorem for spaces with unconditional bases}
Let $E$ be an infinite-dimensional Banach space, and $F$ be a Banach space such that:
\begin{enumerate}
\item $E$ has an equivalent locally uniformly convex norm $\|\cdot\|$ which is $C^1$ smooth.
\item $E=(E,\|\cdot\|)$ has a (normalized) Schauder basis $\{e_{n}\}_{n\in\N}$ such that
for every $x=\sum_{j=1}^{\infty}x_j e_j$ and every $j_0\in\N$ we have that
$$
\left\|\sum_{j\in\N, \, j\neq j_0} x_j e_j\right\|\leq \left\|\sum_{j\in\N}x_j e_j\right\|.
$$
\item In the case that $F$ is infinite-dimensional, there exists
a subset $\mathbb{P}$ of $\N$ such that both $\mathbb{P}$ and $\N\setminus\mathbb{P}$ are infinite and, for every
infinite subset $J$ of $\mathbb{P}$, there exists a linear bounded operator from $\overline{\textrm{span}}\{e_j :
j\in J\}$ onto $F$ (equivalently, $F$ is a quotient space of $\overline{\textrm{span}}\{e_j : j\in J\}$).
\end{enumerate}
Then, for every continuous mapping  $f:E\to F$ and every continuous function $\varepsilon: E\to (0, \infty)$ there
exists a $C^{1}$ mapping $g:E\to F$ such that $\|f(x)-g(x)\|\leq\varepsilon(x)$ and $Dg(x):E\to F$ is a surjective
linear operator for every $x\in E$.
\end{thm}
Recall that a norm $\|\cdot\|$ in a Banach space $E$ is said to be locally uniformly convex (LUC) (or locally
uniformly rotund (LUR)) provided that, for every sequence $(x_n)\subset E$ and every point $x_0$ in $E$, we have
that
$$
\lim_{n\to\infty}2\left( \|x_0\|^2+\|x_n\|^2\right)-\|x_0+x_n\|^2=0 \implies \lim_{n\to\infty}\|x_n -x_0\|=0.
$$
Condition $(2)$ is equivalent to the fact that, for every (equivalently, finite) set $A\subset \N$, $\|P_A\|\le1$,
where $P_A$ stands for the projection $P_A(x)=\sum_{j\in A} x_je_j$. This, in particular, implies that
$\{e_n\}_{n\in N}$ is an unconditional basis; for more details see \cite[p. 53]{AlbKal} or \cite{AlbAns}.
\medskip

The proofs of these theorems will be provided in Sections 3 and 4. These results combine to yield Theorem \ref{Main
result for classical Banach} for $k=1$ (see also Remark \ref{the reason why c0 is OK} in Section 5 for an
explanation of why the space $c_0$ satisfies the assumptions of Theorem \ref{Main theorem for spaces with
unconditional bases}). In order to deduce Theorem \ref{Main result for classical Banach} in the cases of higher
order smoothness, we just have to use Nicole Moulis's results on $C^1$ fine approximation in Banach spaces
\cite{MoulisApproximation} or the more general results of \cite[Corollary 7.96]{HajJoh}, together with the
following fact.

\begin{proposition}\label{C1 approximation is enough}
Assume that the Banach spaces $E$, $F$ satisfy the following properties:
\begin{enumerate}
\item For every continuous mapping $f:E\to F$ and every continuous function $\delta:E\to (0, \infty)$ there exists a $C^1$ smooth mapping $\varphi:E\to F$ such that $\|f(x)-\varphi(x)\|
\leq\delta(x)$ and $D\varphi(x):E\to F$ is surjective for all $x\in E$.
\item For every $C^1$ mapping $\varphi:E\to F$ and every continuous function
$\eta:E\to (0, \infty)$ there exists a $C^k$ mapping $g:E\to F$ such that $\|f(x)-\varphi(x)\| \leq\eta(x)$ and
$\|D\varphi(x)-Dg(x)\|\leq\eta(x)$ for all $x\in E$.
\end{enumerate}
Then, for every continuous mapping $f:E\to F$ and every continuous function $\varepsilon:E\to (0, \infty)$ there
exists a $C^k$ smooth mapping $g:E\to F$ such that $\|f(x)-g(x)\| \leq\varepsilon(x)$ and $Dg(x):E\to F$ is
surjective for every $x\in E$.
\end{proposition}

Nevertheless, it should be noted that our proof of Theorem \ref{Main theorem for reflexive spaces} directly
provides $C^{\infty}$ approximations without critical points in the case that $E$ is a separable Hilbert space; see
Remark \ref{In the Hilbert case we directly get Cinfinity} in Section 5 below.  An easy proof of Proposition
\ref{C1 approximation is enough}, together with some examples, remarks and more technical variants of our results,
is given in Section 5.

Finally, let us mention that as a straightforward application of Theorem \ref{Main result for classical Banach}, we
obtain that, for all Banach spaces $E$ and $F$ appearing in Theorem \ref{Main result for classical Banach}, every
continuous mapping $f:E\to F$ can be uniformly approximated by {\em open} mappings of class $C^k$.  For a more
general statement, see Remark \ref{approximation by open mappings}. Obviously, the latter result is false in the
case that $E$ is finite-dimensional.


\medskip

\section{Extracting closed sets which are locally contained in graphs of infinite codimension}

In this section we will combine ideas and tools of \cite{Renz, West, AzagraDobrowolski} in order to prove Theorem
\ref{final extractibility theorem rough version general form}. We will split the proof into four subsections.
First, in Section 2.1, we will see that each piece of $X$ contained in the graph $G(f_x)$ as provided by condition
$(4)$ of the statement can be flattened by means of homeomorphisms $h_x,\varphi_x:E\to E$ which are sufficiently
close to each other, and whose restrictions to $E\setminus G(f_x)$ and $E\setminus (G(f_x)\setminus U)$ are
diffeomorphisms, respectively. Next, in Sections 2.2 and 2.3, we will show that there exists a diffeomorphism
$g_{x}: E\setminus (C_x\times\{0\}) \to E\setminus ((C_x\times\{0\})\setminus h_x(U))$ which is the identity on
$(E\setminus h_x(U))\setminus (C_x\times\{0\})$  and moves no point more than a fixed small number $\varepsilon_x$.
Then, the composition $\varphi_x^{-1}\circ g_x\circ h_x$ will extract the local chunk of graph $U_x\cap G(f_x)$ and
will move no point too much. Finally, in Section 2.4, we will see how one can patch a collection of diffeomorphisms
extracting pieces of $X$ into a diffeomorphism $h$ which extracts $X$ and is limited by $\mathcal{G}$.

In Section 2.1, we will closely follow Peter Renz's results from \cite{Renz, Renz1972}. In Sections 2.2 and 2.3, we
will combine ideas and techniques from \cite{Renz, Azagra, AzagraDobrowolski}. Finally, in Section 2.4, we will
borrow a technique of James West's \cite[p. 288-290]{West}.

\medskip

\subsection{Flattening graphs}

Here we will prove the following.

\begin{thm}\label{Renz's theorem 1 in our case, healed version}
Let $E_1$ be a Banach space with $C^p$ smooth partitions of unity and $E_2$ be a Banach space which admits a (not
necessarily equivalent) $C^p$ norm. Let $(E=E_1\times E_2,\|\cdot\|)$ and $\pi_1:E\to E_1$ be the natural
projection, i.e., $\pi_{1}(x_1, x_2)=x_1$, $(x_1, x_2)\in E$. Let $X_1\subset E_1$ be a closed set, $f:X_1\to E_2$
a continuous mapping, $U\subset E$ an open set, and $\varepsilon>0$. Write $G(f)=\{(x_1, x_2)\in E \, : \,
x_2=f(x_1),\ x_1\in X_1\}$. Then there exist a couple of homeomorphisms $h, \varphi: E\to E$ such that:
\begin{enumerate}
\item $h(G(f))\subset E_1\times\{0\}$ and $\varphi(G(f)\setminus U)\subset
(E_1\times\{0\})\setminus h(U)$;
\item $h=\varphi$ off of $U$;
\item $\pi_1\circ h=\pi_1=\pi_1\circ\varphi$;
\item $h$ restricted to $E\setminus G(f)$ is a $C^p$ diffeomorphism of
$E\setminus G(f)$ onto $E\setminus (X_1\times\{0\})$;
\item $\varphi$ restricted to $E\setminus (G(f)\setminus U)$ is a $C^p$ diffeomorphism of
$E\setminus (G(f)\setminus U)$ onto $E\setminus \left( (X_1\times\{0\})\setminus h(U)\right)$.
\item $\|h^{-1}(x)-\varphi^{-1}(x)\|\leq\varepsilon$ for every $x\in E$.
\item $h^{-1}(x_1,x_2)$ is uniformly continuous with respect to the second coordinate $x_2$ (meaning that for every $\varepsilon>0$ there exists $\delta>0$ such that if $\|x_{2}- x_{2}'\|<\delta$ then
$\|h^{-1}(x_1, x_2)- h^{-1}(x_1, x_{2}')\|<\varepsilon$ for all $x_1$).
\end{enumerate}
\end{thm}
We will assume without loss of generality that $\varepsilon\leq 1$.

In what follows, slightly abusing notation, we will indistinctly use the symbol $\|\cdot\|$ to denote the norms
$\|\cdot\|_{E_1}$, $\|\cdot\|_{E_2}$, and $\|\cdot\|$  with which the Banach spaces $E_1$, $E_2$ or
$E_{1}\times E_2$ are endowed. We may and do assume that $\|x_1\|_{E_1}=\|(x_1, 0)\|$ and
$\|x_2\|_{E_2}=\|(0, x_2)\|$ for all $(x_1, x_2)\in E_{1}\times E_2$.

Now, we state and prove a sequence of lemmas that will be employed in proving the above theorem. The most important
are Lemmas \ref{Variation-Renz's lemma 1} and \ref{Modified Renz's lemma 2}. Basically, we follow the ideas of
Renz's paper \cite{Renz1972} and Ph.D. thesis \cite{Renz}, with some minor but very important changes.

The proof of our first lemma is a consequence of the existence of $C^p$ smooth partitions of unity on $E_1$.
\begin{lemma}\label{smooth extension}
The function $f:X_1\to E_2$ extends to a continuous function $\bar f:E_1\to E_2$ such that $\bar f|E_1\setminus
X_1$ is $C^p$ smooth.
\end{lemma}

\begin{lemma}\label{Variation-Renz's lemma 1}
Let $E_1$ be a Banach space with $C^p$ smooth partitions of unity, $E_2$ be a Banach space removed text, $X_1$ be a closed subset of $E_1$, and $f:X_1\to E_2$ be a continuous mapping.
For every $n\in\N$, write
$$
W_n=\left\lbrace x_1\in E_1:\,\operatorname{dist} (x_1,X_1)\le {1\over n}\right\rbrace.
$$
Assume $\bar f:E_1\rightarrow E_2$ is a continuous extension of $f$ such that $\bar f|E_1\setminus X_1$ is $C^p$ smooth.
Then, there is a continuous mapping
$$
F:\R\times E_1\rightarrow E_2
$$
such that
\begin{enumerate}
\item $F(r,x_1)=\bar f(x_1)$ for all $(r,x_1)\in(r_n,\infty)\times E_1\setminus W_n$ and some
$0<r_n<1$; in particular, $F(r,x_1)=\bar f(x_1)$ for all $(r,x_1)$ in some neighborhood of the set
$\{1\}\times(E_1\setminus X_1)$ in $\mathbb{R}\times(E_1\setminus X_1)$;
\item $F|\mathbb{R}\times (E_1\setminus X_1)\cup(-\infty,1)\times E_1$ is $C^p$ smooth;
\item $F(r,x_1)=f(x_1)$ for $r\geq 1$ and $x_1\in X_1$;
\item $\| D_1F(r,x_1)\|\leq {1\over 2}$ for all $r\in\R$, $x_1\in E_1$.
\end{enumerate}
\end{lemma}
\begin{proof}
For every $n\in \N$ we can find a sequence of $C^p$ functions $\bar f_n:E_1\rightarrow E_2$ such that
$$
||\bar f(x_1)-\bar f_n(x_1)||\leq 2^{-2n-4}
$$
for every $x_1\in E_1$. The existence of such a sequence is again guaranteed by the existence of $C^p$ partitions
of unity in $E_1$ (see, for instance, \cite[Theorem VII.3.2]{DGZ}). We will now improve the sequence $\left\lbrace \bar
f_n\right\rbrace _{n\geq 1}$ to $\left\lbrace f_n\right\rbrace _{n\geq 1}$ so that the sequence $\{f_n|E_1\setminus
X_1\}$ locally stabilizes with respect to $n$. To achieve this, we use the existence of $C^p$ partitions of unity
to find a $C^p$ function $\lambda_n: E_1\rightarrow [0,1]$ which is $1$ on $E\setminus W_n$ and $0$ on $W_{n+1}$.
Define
$$f_n(x_1)=\lambda_n (x_1)\bar{f}(x_1)+(1-\lambda_n(x_1))\bar f_n(x_1)$$
for all $x_1\in E_1$. It follows that $\|f_n(x_1)-f_{n+1}(x_1)\|\le 2^{-2n-3}$ for all $x_1\in E_1$.

For every $n\in\mathbb{N}$, pick a nondecreasing $C^\infty$ function $h_n:\mathbb{R}\to[0,1]$ such that $h_n(r)=0$ for
$r\le 1-2^{1-n}$, $h_n(r)=1$ for $r\ge1-2^{-n}$, and $h'_n(r)\le 2^{n+1}$. One can check that
$$
F(r,x_1)=f_1(x_1)+\sum_{n=1}^\infty h_{n+1}(r)(f_{n+1}(x_1)-f_n(x_1))
$$
defines a required mapping.
\end{proof}
Observe that in fact $F(r,x_1)$ is Lipschitz with constant $1$ with respect to the first variable $r\in\R$. That
is,
\begin{align*}
||F(r,x_1)-F(r',x_1)||&\leq\sum^{\infty}_{n=1}|h_{n+1}(r)-h_{n+1}(r')|||f_{n+1}(x_1)-f_n(x_1)||\leq\\
&\leq \sum^{\infty}_{n=1}2^{n+1}|r-r'|2^{-2n-3}\leq |r-r'|
\end{align*}
for every $x_1\in E_1$.\\
We will write
$$U_0=\pi_1(G(\bar f)\cap U),
$$
which is an open set in $E_1$, and also
$$Y_1=X_1\setminus U_0=X_1\setminus \pi_1(G(\bar
f)\cap U)=\pi_1(G(f)\setminus U),
$$
which is a closed subset of $E_1$. By replacing $U$ with $U\cap\pi_1^{-1}(U_0)$, we can assume that
$$
U_0=\pi_1(U).
$$

\begin{lemma}\label{step-like open subset of $U$}
With the above notation, take a decreasing sequence of positive numbers $\{ \delta_n\}_{n\geq 1}$ converging to
zero. Then there exists an increasing sequence of open subsets in $E_1$
$$V_1\subseteq \overline{V_1}\subseteq V_2\subseteq \overline{V_2}\subseteq \cdots\subseteq \overline{V_n}\subseteq V_{n+1}\subseteq\cdots\subseteq U_0$$
such that $\bigcup^{\infty}_{n=1}V_n=U_0$ and the sets
$$U_n :=\{ (x_1,x_2)\in E:\,\|x_2-\bar f(x_1)\|<\delta_n,\, x_1\in V_n\}$$
are contained in $U$.
\end{lemma}
\begin{proof}
To be able to get the required inclusions between the sets $V_n$, we first take an auxiliary sequence of open sets
$W_n$ in $U_0$ such that $\overline{W_n}\subseteq W_{n+1}$ for every $n\in N$ and $\bigcup^{\infty}_{n=1}W_n=U_0$.

Then we define
$$V'_n=\left\lbrace x_1\in U_0:\, \{(x_1,x_2) \in E: \|x_2-\bar f(x_1)\|< \delta_n \}\subseteq U\right\rbrace $$
for every $n\in\N$. Observe that we have $V'_n\subseteq V'_{n+1}$ and $\bigcup^{\infty}_{n=1}V'_n=U_0$, but we
cannot assure that $\overline{V'_n}\subseteq V'_{n+1}$ for every $n\in N$. So now we mix these sets with the
previous $W_n$, that is, we let $V_n=W_n\cap V'_n$. Obviously, by definition, for every $n\in \N$ the set $U_n=\{
(x_1,x_2)\in E:\,\|x_2-\bar f(x_1)\|< \delta_n,\, x_1\in V_n\}$ is contained in $U$. Now, we have that
$\overline{V_n}\subseteq V_{n+1}$ for every $n\in N$; also $\bigcup^{\infty}_{n=1}V_n=U_0$.
\end{proof}

The following lemma resembles \cite[Lemma 2.2]{Renz1972} and \cite[Lemma 2]{Renz} (in which only one function
$\phi$ is considered). However, Theorem \ref{Renz's theorem 1 in our case, healed version} requires constructing
two homeomorphisms $h$ and $\varphi$ which are identical outside $U$. The building block in constructing those
homeomorphisms are two functions $\phi$ and $\tilde\phi$ whose existence is claimed in the lemma below. The
existence of $\tilde\phi$ is crucial. Incidentally, let us note that Renz's proof of \cite[Theorem 4]{Renz} is
flawed (and this is the reason why we must deal with two functions $\phi$ and $\tilde\phi$ instead of just the
function $\phi$), but can be corrected by using Theorem \ref{Renz's theorem 1 in our case, healed version}.

\begin{lemma}\label{Modified Renz's lemma 2}
Let $\bar f:E_1\rightarrow E_2$ be the uniform limit of $C^p$ functions, where $E_1$ has $C^p$ partitions of unity and $E_2$ has a (not necessarily equivalent) $C^p$ smooth norm. Then there are two continuous functions
$\phi$,  $ \tilde\phi:E\rightarrow[0,1]$ such that
\begin{enumerate}
\item $\phi^{-1}(1)=G(\bar f)$ and ${\tilde \phi}^{-1}(1)=G(\bar f)\setminus
U$;
\item $\phi|E\setminus G(\bar f)$ and $\tilde\phi|E\setminus(G(\bar f)\setminus U)$ are $C^p$
smooth;
\item $\| D_2\phi(x_1,x_2)\|\leq {1\over 2}$ for all $(x_1,x_2)\in E\setminus G(\bar f)$,
and $\|D_2\tilde \phi(x_1,x_2)\|\leq {1\over 2}$ for all $(x_1,x_2)\in E\setminus (G(\bar f)\setminus U)$;
\item $\phi=\tilde\phi$ outside $U$.
\end{enumerate}
\end{lemma}

\begin{proof}
To construct $\phi$ we will follow \cite[Lemma 2.2]{Renz1972}. A similar argument will be used to construct
$\tilde\phi$; however, we have to make sure that $\tilde\phi|G(\bar f)\cap U<1$.

For $n\in \N$, let $a_n,b_n,c_n, d_n, \varepsilon_n$ be positive numbers with the following properties:
\begin{enumerate}
\item they tend to zero as $n$ tends to infinity;
\item $a_n<b_n$ for all $n$;
\item $\varepsilon_{n+1}+b_{n+1}<a_n-\varepsilon_n$ for all $n$;
\item $\sum^{\infty}_{n=1} c_n\leq {\epsilon\over 2}\leq 1$;
\item $\sum^{\infty}_{n=1}d_n\leq
{1\over 2}$
\end{enumerate}
(for instance, let us set $a_n=\epsilon 2^{-2n}$, $b_n=\epsilon 2a_n$, $c_n=\epsilon 2^{-4n}$, $d_n=\epsilon
2^{-2n}$ and $\varepsilon_n=\epsilon 2^{-4(n+1)}$). Let $h_n$ be a nonincreasing $C^p$ function from $\R$ to $\R$
satisfying
\begin{align*}
&c_n= h_n(r)=h_n(0)>0 &\text{whenever}\; r\leq a_n,\\
&h_n(r)=0 &\text{whenever}\; r\geq b_n,\\
&|h'_n(r)|\leq d_n &\text{for all}\; r \;\text{in}\;\R,
\end{align*}
and $g_n:E_1\to E_2$ be a $C^p$ mapping such that $\|g_n(x_1)-\bar f(x_1)\|\leq \varepsilon_n$ for every $x_1\in
E_1$. Then
$$\psi_n(x_1,x_2)=h_n\left( \|x_2-g_n(x_1)\|\right) $$
defines a nonnegative $C^p$ function on $E_1\times E_2=E$ satisfying
\begin{align*}
&c_n=\psi_n(x_1,x_2)=h_n(0)>0 & \textrm{if } \,\,\, \|x_2-\bar f(x_1)\|\leq a_n-\varepsilon_n ,\\
& \psi_n(x_1,x_2)=0 & \textrm{if } \,\,\, \|x_2-\bar f(x_1)\|\geq b_n+\varepsilon_n,\\
& \|D_2\psi_n(x_1,x_2)\|\leq d_n & \textrm{for all } \,\,\, (x_1,x_2)\in E_1\times E_2.
\end{align*}

The nonnegativity and first two properties of $\psi_n$ are evident, and it is easy to see that $\psi_n$ is $C^p$ on
$E$. The bound on the norm of the derivative $D_2\psi_n$ is established by using the chain rule and the fact that
the operator norm of the derivative of the norm of any Banach space is less than or equal to one.

Define
\begin{equation}
\psi(x_1,x_2)=\sum^{\infty}_{n=1}\psi_n(x_1,x_2)
\end{equation}
for all $(x_1,x_2)\in E$. \\
Similarly, define
\begin{equation}
\tilde\psi(x_1,x_2)=\sum^{\infty}_{n=1}\lambda_n(x_1)\psi_n(x_1,x_2)
\end{equation}
for all $(x_1,x_2)\in E$, where $\lambda_n:E_1\rightarrow [0,1]$ is a $C^p$ smooth function such that
$\lambda_n(x_1)=1$ if $x_1\notin V_n$ and $\lambda_n(x_1)=0$ if $x\in V_{n-1}$. Here, the sets $V_n$ are provided
by Lemma \ref{step-like open subset of $U$} for the sequence $\delta_n:=\varepsilon_n+b_n$ (let $V_0=\emptyset$ and
assume $V_1\not=\emptyset$). In particular, observe that since $\bigcup^{\infty}_{n=1}V_n=U_0$ then
$\lambda_n(x_1)=1$ for every $x_1\notin U_0$; hence, $\psi(x_1,x_2)=\tilde\psi(x_1,x_2)$ for $x_1\notin U_0$.

Since the functions $\psi$ and $\tilde\psi$ are defined via absolutely and uniformly convergent series of
continuous functions, they are continuous.

If $(x_1,x_2)\in E\setminus G(\bar f)$, then $\|x_2-\bar f(x_1)\|>b_n+\varepsilon_n$ for some $n\in \N$. By
continuity, the inequality holds in a neighborhood of $(x_1,x_2)$ so $\psi_k$ vanishes for $k\geq n$. Hence, $\psi$
is locally a finite sum of $C^p$ functions, and in particular is $C^p$ on $E\setminus G(\bar f)$.

Also, if $(x_1,x_2)\in G(\bar f)\cap U$, then $x_1\in U_0$ and $x_1\in V_n$ for some $n\in\N$. So
$\lambda_k(x_1)=0$ for all $k\geq n+1$. This means that $\tilde\psi$ is locally a finite sum of $C^p$ functions
and, thus, is of class $C^p$ on $E\setminus (G(\bar f)\setminus U)$.

The derived series for $D_2\psi$ and $D_2\tilde\psi$ are absolutely and uniformly convergent in view of the bounds
on $\|D_2\psi_n\|$ and the fact that $D_2\lambda_n=0$. Then differentiation term by term is justified and
$\|D_2\psi(x_1,x_2)\|\leq{1\over 2}$ for all $(x_1,x_2)\in E\setminus G(\bar f)$ and
$\|D_2\tilde\psi(x_1,x_2)\|\leq{1\over 2}$ for all $(x_1,x_2)\in E\setminus (G(\bar f)\setminus U)$.

Each point $(x_1,x_2)\in G(\bar f)$ satisfies $0=\|x_2-\bar f(x_1)\|<a_n-\varepsilon_n$ for all $n\in\N$,
consequently $\psi$ equals the constant
$$d^*=\sum^{\infty}_{n=1}\psi_n(x_1,\bar f(x_1))=\sum^{\infty}_{n=1}h_n(0)=\sum^{\infty}_{n=1}c_n\leq 1.$$
On the other hand, if $(x_1,x_2)\in G(\bar f)\setminus U$, then $0=\|x_2-\bar f(x_1)\|<a_n+\varepsilon_n$ and
$\lambda_{n}(x_1)=1$ for all $n\in\N$, so $\tilde\psi$ equals again the constant $d^*$. In fact $d^*$ is the
supremum of $\psi$ and of $\tilde\psi$, and is easily seen to be attained in $G(\bar f)$ and $G(\bar f)\setminus
U$, respectively.

To show that $\psi$ and $\tilde\psi$ are equal outside $U$ take $(x_1,x_2)\in E$. By a remark after the definition
of $\psi$ and $\tilde\psi$, we can assume $x_1\in U_0$.

\begin{claim}
For every $n\in\N$, if $x_1\in V_n\setminus V_{n-1}$ and if $x_2\in E_2$ is such that $\|x_2-\bar f(x_1)\|\geq
b_n+\varepsilon_n$, then $\psi(x_1,x_2)=\tilde\psi(x_1,x_2)$.
\end{claim}

\begin{proof}[Proof of Claim]

If $\|x_2-\bar f(x_1)\|\geq b_n+\varepsilon_n$ we have that $\psi_k(x_1,x_2)=0$ for all $k\geq n$. So we have to
see that $ \lambda_k(x_1)=1$ for $k=1,\dots,n-1$. But this is clear since $x\notin V_{n-1}$ and hence $x\notin V_k$
for any $k=1,\dots,n-1$.
\end{proof}

Now, we can conclude that for each $n\in\N$, $\psi=\tilde\psi$ on the set
$$((V_n\setminus V_{n-1})\times E_2)\setminus U_n\supseteq((V_n\setminus V_{n-1})\times E_2)\setminus U.$$
Since $\bigcup^{\infty}_{n=1}V_n\setminus V_{n-1}= U_0$ and $\bigcup_{n=1}^{\infty}U_n\subseteq U$, it follows that
$\psi$ is equal to $\tilde\psi$ outside $U$.

Finally, to obtain functions $\phi$ and $\tilde\phi$ with the desired properties it is sufficient to set
\begin{align*}
\phi(x_1,x_2)&=\psi(x_1,x_2)+1-d^*\\
\tilde\phi(x_1,x_2)&=\tilde\psi(x_1,x_2)+1-d^*
\end{align*}
for all $(x_1,x_2)\in E$. This ensures that the supremum, which is attained precisely on $G(\bar f)$ for $\phi$ and
precisely on $G(\bar f)\setminus U$ for $\tilde\phi$, is equal to $1$.

\end{proof}

In the proof of Theorem \ref{Renz's theorem 1 in our case, healed version}, we will employ a well known fact
stating that the identity mapping perturbed by a contracting mapping is a homeomorphism (even a diffeomorphism
provided that the contracting mapping is smooth). This fact is stated and proved in Lemma 3 of Renz's Ph.D. thesis
\cite{Renz}.

\begin{lemma}\label{Renz's lemma 3}
Let $E_1$ be a normed linear space and $E_2$ be a Banach space. Let $E=E_1\times E_2$ and let $d:E\to E_2$ be a
continuous mapping satisfying the following condition
$$\|d(x_1,x_2)-d(x_1,x'_2)\|\leq {1\over 2} \|x_2-x'_2\|$$
for all $x_1\in E_1$ and $x_2,x'_2\in E_2$. Then the mapping defined by $h(x_1,x_2)=(x_1,x_2-d(x_1,x_2))$ is a
homeomorphism of $E$ onto itself. Moreover, $h$ is a $C^p$ diffeomorphism when restricted to any open set (onto its
image) on which $d$ is $C^p$ smooth.
\end{lemma}

Let us now present the proof of Theorem \ref{Renz's theorem 1 in our case, healed version}.

\begin{proof}[Proof of Theorem \ref{Renz's theorem 1 in our case, healed version}]
Basically, we will follow the proof of Theorem 1 of Renz's Ph.D. thesis \cite{Renz}.

First, we apply Lemma \ref{smooth extension} to $f:X_1\rightarrow E_2$ to obtain a continuous mapping $\bar f:
E_1\rightarrow E_2$ such that $\bar f|X_1=f$ and $\bar f|E_1\setminus X_1$ is $C^p$ smooth. Then, we apply Lemma
\ref{Variation-Renz's lemma 1} to the mapping $\bar f$ to obtain a mapping $F:\mathbb{R}\times E_1\to E_2$
satisfying conditions $(1)$--$(4)$ of Lemma \ref{Variation-Renz's lemma 1}. Next, we apply Lemma \ref{Modified
Renz's lemma 2} to $\bar f$ to obtain functions $\phi$ and $\tilde\phi$ satisfying conditions $(1)$--$(4)$ of Lemma
\ref{Modified Renz's lemma 2}. Now, we define
\begin{align*}
d(x_1,x_2)&=F(\phi(x_1,x_2),x_1)\\
\tilde d (x_1,x_2)&=F(\tilde\phi(x_1,x_2),x_1).
\end{align*}
Let us check that Lemma \ref{Renz's lemma 3} is applicable to $d$ and $\tilde d$ so that
$$
h(x_1,x_2)=(x_1,x_2-d(x_1,x_2))
$$
and
$$
\varphi(x_1,x_2)=(x_1,x_2-\tilde d (x_1,x_2))
$$
are homeomorphisms (which, additionally, will satisfy the conditions enumerated in Theorem \ref{Renz's theorem 1 in
our case, healed version}).

Both functions $d$ and $\tilde d$ are continuous  as compositions of continuous functions. We compute $D_2 d$ and
$D_2\tilde d$ to obtain
\begin{align*}
D_2 d(x_1,x_2)& =D_1 F(\phi(x_1,x_2),x_1)\circ D_2\phi(x_1,x_2)\\
D_2 \tilde d(x_1,x_2)& =D_1 F(\tilde\phi(x_1,x_2),x_1)\circ D_2\tilde\phi(x_1,x_2).
\end{align*}
The estimates of the norms of $D_1 F$, $D_2\phi$, and $D_2\tilde\phi$ yields $\|D_2d(x_1,x_2)\|,\,\|D_2\tilde
d(x_1,x_2)\|\leq {1\over 4}$ when $(x_1,x_2)\notin G(\bar f)$. Since $d$ and $\tilde d$ are continuous and
$E\setminus G(\bar f)$ is dense in $E$, by the mean value theorem, we can write
\begin{align*}
\|d(x_1,x_2)-d(x_1,x'_2)\|&\leq {1\over 4}\|x_2-x'_2\|\\
\|\tilde d(x_1,x_2)-\tilde d(x_1,x'_2)\|&\leq {1\over 4}\|x_2-x'_2\|
\end{align*}
for all $x_1\in E_1$ and all $x_2$, $x'_2\in E_2$. Hence, Lemma \ref{Renz's lemma 3} applies and yields that $h$
and $\varphi$ are homeomorphisms.

Let us show conditions $(1)$--$(7)$ of Theorem \ref{Renz's theorem 1 in our case, healed version}.

First, we will verify condition $(1)$. If $(x_1,x_2)\in G(f)$, then $d(x_1,x_2)=F(1,x_1)=f(x_1)=x_2$ and
\begin{equation}\label{Gf goes to X_1}
h(x_1,x_2)=(x_1,x_2-d(x_1,x_2))=(x_1,0).
\end{equation}
If $(x_1,x_2)\in G(f)\setminus U$, then $\tilde d(x_1,x_2)=F(1,x_1)=x_2$ and
\begin{equation}\label{Gf minus U goes to Y_1}
\varphi(x_1,x_2)=(x_1,x_2-\tilde d(x_1,x_2))=(x_1,0).
\end{equation}

Condition $(3)$ is obvious. Since $\phi$ and $\tilde\phi$ are equal outside $U$ we obtain (2).

Let us see that $d|E\setminus G(f)$ and $\tilde d|E\setminus (G (f)\setminus U)$ are $C^p$ diffeomorphisms. If
$(x_1,x_2)\notin G(\bar f)$ then $\phi$ and $\tilde \phi$ are $C^p$ smooth and $\phi(x_1,x_2)$,
$\tilde\phi(x_1,x_2)<1$ by condition $(2)$ and $(1)$ of Lemma \ref{Modified Renz's lemma 2}. It follows that
$F(\phi(x_1,x_2),x_1)$ and $F(\tilde\phi(x_1,x_2),x_1)$ are $C^p$ smooth in a neighborhood of $(x_1,x_2)$. Thus $h$
and $\varphi$ are $C^p$ on $E\setminus G(\bar f)$.

On the other hand, we have $\phi|G(\bar f)\setminus G(f)=1$. Then, by continuity of $\phi$ and condition $(1)$ of
Lemma \ref{Variation-Renz's lemma 1}, we infer that $F(\phi(x_1,x_2),x_1)=\bar f(x_1)$ and, consequently,
$h(x_1,x_2)=(x_1,x_2-\bar f(x_1))$ in a neighborhood of $G(\bar f)\setminus G(f)$. We have proved that $d$ is $C^p$
smooth on $E\setminus G(f)$.

It remains to show that $\varphi|U$ is $C^p$ smooth. By condition $(1)$ of Lemma \ref{Modified Renz's lemma 2}, we
have $\tilde\phi|U<1$; by condition $(2)$ of  Lemma \ref{Variation-Renz's lemma 1}, $\tilde d|U$ is $C^p$ smooth.
The proof that $\tilde d$ and, therefore, $\varphi$ restricted to $E\setminus (G(f)\setminus U)$ is $C^p$ smooth is
complete.

Now, Lemma \ref{Renz's lemma 3} tells us that $h$ and $\varphi$ are $C^p$ diffeomorphisms of $E\setminus G(f)$ onto
$h(E\setminus G(f)) $ and $E\setminus (G(f)\setminus U)$ onto $\varphi(E\setminus (G(f)\setminus U)) $. So to get
(4) and (5) it is sufficient to show that $h(G(f))=X_1\times\{0\}$ and $\varphi(G(f)\setminus
U)=(X_1\times\{0\})\setminus h(U)$. The first equality is clear from equation \eqref{Gf goes to X_1}. For the
second one, observe that \eqref{Gf minus U goes to Y_1} tells us that $\varphi(G(f)\setminus U)=(X_1\setminus
U_0)\times\{0\}=Y_1\times\{0\}$. So, we must check that
$$
(X_1\times\{0\})\setminus h(U)=(X_1\setminus U_0)\times\{0\},
$$
or, what is the same, that $\pi_1(h(U))=U_0$. The latter follows
from condition $(3)$ and the fact that $\pi_1(U)=U_0$.\\
Let us finish the proof by showing (6) and (7). Firstly let us check that
$||h^{-1}(x_1,x_2)-\varphi^{-1}(x_1,x_2)||\leq \varepsilon$ for every $(x_1,x_2)\in E_1\times E_2$. Since $h^{-1}$
preserves the first coordinate, we can write $h^{-1}(x_1,x_2)=(x_1,y_2)$ and $ \varphi^{-1}(x_1,x_2)=(x_1,z_2)$
where $y_2,z_2\in E_2$ are such that
\begin{align*}
y_2-d(x_1,y_2)&=x_2\\
z_2-\tilde d(x_1,z_2)&=x_2.
\end{align*}
We then have that $||h^{-1}(x_1,x_2)-\varphi^{-1}(x_1,x_2)||\leq \varepsilon$ if and only if
$||y_2-z_2||\leq\varepsilon$ and if and only if $||d(x_1,y_2)-\tilde d(x_1,z_2)||\leq\varepsilon$. Since $r\mapsto
F(r,x_1)$ is $1$-Lipschitz, this is true if $|\phi(x_1,y_2)-\tilde\phi(x_1,z_2)|\leq\epsilon$, or what is the same
if $|\psi(x_1,y_2)-\tilde\psi(x_1,z_2)|\leq\epsilon$. And this is the case because
\begin{align*}
|\psi(x_1,y_2)-\tilde\psi(x_1,z_2)|&\leq |\psi(x_1,y_2)|+|\tilde \psi(x_1,z_2)|=|\sum^{\infty}_{n=1}\psi_n(x_1,y_2)| +|\sum^{\infty}_{n=1}\lambda_n(x_1)\psi_n(x_1,z_2)|\leq\\
& \leq\sum^{\infty}_{n=1}c_n+\sum^{\infty}_{n=1}c_n\leq{\epsilon\over 2}+{\epsilon\over 2}=\epsilon.
\end{align*}
Secondly, let us see that for every $\eta>0$ there exists $\delta>0$ such that if
$||(x_1,x_2)-(x_1,x_2')||=||x_2-x_2'||\leq\delta$ then $||h^{-1}(x_1,x_2)-h^{-1}(x_1,x_2')||\leq\eta$. It will be
enough to set $\delta={\eta\over 2}$.  Indeed, take $(x_1,x_2),(x_1,x_2')\in E_1\times E_2$ such that
$||(x_1,x_2)-(x_1,x_2')||=||x_2-x_2'||\leq{\eta\over 2}$. Write $ h^{-1}(x_1,x_2)=(x_1,y_2)$ and
$h^{-1}(x_1,x_2')=(x_1,y_2')$ where $y_2,y_2'\in E_2$ are such that
\begin{align*}
y_2-d(x_1,y_2)&=x_2\\
y_2'-\tilde d(x_1,y_2')&=x_2'.
\end{align*}
Then we have that
\begin{align*}
||h^{-1}(x_1,x_2)-h^{-1}(x_1,x_2')||&=||y_2-y_2'||\leq ||x_2-x_2'||+||d(x_1,y_2)-d(x_1,y_2')||\leq\\
&\leq||x_2-x_2'||+{1\over 2}||y_2-y_2'||,
\end{align*}
which implies that $||y_2-y_2'||\leq 2||x_2-x_2'||=\eta$, and the proof is complete.
\end{proof}

\medskip

\subsection{An extracting scheme tailored for closed subsets of a subspace of infinite codimension}

In this subsection we will establish the following.

\begin{thm}\label{general extracting scheme for closed subsets of a subspace of infinite codimension}
Let $E_1$ and $E_2$ be Banach spaces such that $E_2$ is infinite-dimensional and admits a (not necessarily
equivalent) $C^p$ smooth norm, where $p\in\N\cup\{\infty\}$. Define $E=E_1 \times E_2$ and, for $i=1,2$, write
$\pi_i:E\to E_i$ for the natural projections, that is, $\pi_i(x_1, x_2)=x_i$ for $(x_1, x_2)\in E$. Let $W_1$ be an
open subset of $E_1$, and $\psi:W_1\to [0, \infty)$ be a continuous function such that $\psi$ is of class $C^p$ on
$\psi^{-1}(0, \infty)$. Denote $K=\psi^{-1}(0)\times\{0\}$. Then, there exists a $C^p$ diffeomorphism $h$ from
$\left(W_1\times E_2\right)\setminus K$ onto $W_1\times E_2$ which satisfies $\pi_1\circ h=h$ and is the identity  off
of a certain open subset $U$ of $W_{1}\times E_2$.

Specifically, the set $U$ is defined as follows
$$
U:=\{x=(x_1, x_2)\in W_{1}\times E_2 \, : \, S(\psi(x_1), \omega(x_2))<1\},
$$
where $S$ is a certain $C^\infty$ norm on $\mathbb{R}^2$ and $\omega:E_2\to[0,\infty)$ is a certain (not
necessarily symmetric) subadditive and positive-homogeneous functional of class $C^p$ on $E_2\setminus\{0\}$; see
Lemmas \ref{functional omega} and \ref{properties of smooth squares} for precise definitions.
\end{thm}
By a $C^p$ smooth norm on $E_2$ we mean a (possibly nonequivalent) norm on $E_2$ which is of class $C^p$ on
$E_{2}\setminus\{0\}$.

We will need to use the following three auxiliary results from \cite{Azagra, AzagraDobrowolski}.

\begin{lemma}\label{fixed point lemma}
Let $F:(0,\infty)\longrightarrow[0,\infty)$ be a continuous function such that, for every $\beta\geq\alpha>0$,
$$
F(\beta)-F(\alpha)\leq\frac{1}{2}(\beta-\alpha), \hspace{4mm} \textrm{and} \hspace{4mm} \limsup_{t\to 0^{+}}F(t)>0.
$$
Then there exists a unique $\alpha>0$ such that $F(\alpha)=\alpha$.
\end{lemma}
\begin{proof}
See \cite[Lemma 2]{Azagra}.
\end{proof}

It is not known whether every infinite-dimensional Banach space with a $C^1$ equivalent norm possesses a $C^1$
smooth non-complete norm.\footnote{For $C^k$ with $k\geq 2$ in place of $C^1$, the answer to this question is
positive; see \cite{AlessandroHajek}.} The following lemma shows that for every Banach space with a $C^{p}$ smooth
norm there exists a kind of $C^p$ {\em asymmetric non-complete subadditive functional} which successfully replaces
the smooth non-complete norm in Bessaga's technique  \cite{Be} for extracting points.

\begin{lemma}\label{functional omega}
Let $(E_2, \|\cdot\|)$ be an infinite-dimensional Banach space which admits a (not necessarily equivalent) $C^{p}$
smooth norm, where $p\in\N\cup\{\infty\}$. Then there exists a continuous function $\omega:E_2\longrightarrow
[0,\infty)$ which is $C^{p}$ smooth on $E_2\setminus\{0\}$ and satisfies the following properties:
\begin{enumerate}
\item $\omega(x+y)\leq\omega(x)+\omega(y)$, and, consequently,
$\omega(x)-\omega(y)\leq\omega(x-y)$, for every $x,y\in E_2$;
\item $\omega(rx)=r\omega(x)$ for every $x\in E_2$, and $r\ge 0$;
\item $\omega(x)=0$ if and only if $x=0$;
\item $\omega(\sum_{k=1}^{\infty}z_{k})\leq\sum_{k=1}^{\infty}\omega(z_{k})$
for every convergent series $\sum_{k=1}^{\infty}z_{k}$ in $(E_2, \|\cdot\|)$; and
\item for every $\varepsilon>0$, there exists a sequence of vectors
$(y_{k})\subset E_2$ such that
$$
\omega(y_k)\leq\frac{\varepsilon}{4^{k+1}}, \;\ ;
$$
for every $k\in\mathbb{N}$, and
$$
\liminf_{n\to\infty}\omega(y-\sum_{j=1}^{n}y_{j})>0
$$
for every $y\in E_2$.
\end{enumerate}
Notice that $\omega$ need not be a norm in $E_2$, as in general we have $\omega(x)\neq\omega(-x)$.
\end{lemma}
\begin{proof}
See \cite[Lemma 2.3]{AzagraDobrowolski}
\end{proof}

Using the properties of the functional $\omega$ we can construct an {\em extracting curve} as follows.
\begin{lemma}\label{deleting path}
Let $(E_2, \|\cdot\|)$ be a Banach space, and let $\omega$ be a functional satisfying conditions $(1)$, $(2)$, and
$(5)$ of Lemma \ref{functional omega}. Then there exists a $C^{\infty}$ curve $\gamma:(0,\infty)\longrightarrow
E_2$ such that
\begin{enumerate}
\item $\omega(\gamma(\alpha)-\gamma(\beta))\leq\frac{1}{2}(\beta-\alpha)$
if $\beta\ge\alpha>0$;
\item $\limsup_{t\to 0^{+}}\omega(y-\gamma(t))>0$ for every $y\in E_2$; and
\item $\gamma(t)=0$ if $t\geq 1$.
\end{enumerate}
\end{lemma}
\begin{proof}
Let $\theta:[0,\infty)\longrightarrow[0,1]$ be a non-increasing $C^{\infty}$ function such that $\theta=1$ on $[0,
1/2]$, $\theta=0$ on $[1,\infty)$ and $\sup\{|\theta'(t)| : t\in [0,\infty)\}\leq 4$. Let us choose a sequence of
vectors $(y_k)\subset E_2$ which satisfies condition $(5)$ of Lemma \ref{functional omega} for $\varepsilon=1$, and define
$\gamma:(0,\infty)\longrightarrow E_2$ by the following formula
$$
\gamma(t)=\sum_{k=1}^{\infty}\theta(2^{k-1}t)y_{k}.
$$
It is not
difficult to check that this curve satisfies the properties of the statement. See \cite[Lemma
2.5]{AzagraDobrowolski} for details.
\end{proof}

We will also need a technical tool (see, for instance, \cite[Lemmas 2.27 and 2.28]{AzagraMontesinos}) that allows
us to obtain, on the product space $E_1\times E_2$, a norm which preserves the smoothness properties that the
corresponding norms of the factors may have. Notice that the natural formula $\left( \|x_1\|_1^{2}+
\|x_{2}\|_{2}^{2}\right)^{1/2}$ defines a $C^1$ norm in $E_1\times E_2\setminus\{0\}$ but, in general, this norm
will not be $C^2$ on this set, even if $\|\cdot\|_1$ and $\|\cdot\|_{2}$ are $C^{\infty}$ on $E_1\setminus\{0\}$
and $E_2\setminus\{0\}$, respectively, because the function  $x_2\mapsto \|x_2\|_2^2$ may not be $C^2$ smooth on
all of $E_2$ even though it is $C^{\infty}$ smooth on $E_{2}\setminus\{0\}$. As a matter of fact, it is not
difficult to show that, for every Banach space $(E, \|\cdot\|)$, if $\|\cdot\|^2$ is twice Fr\'echet differentiable
at $0$, then $E$ is isomorphic to a Hilbert space; see, for instance \cite[Exercise 10.4, pp. 475-476]{FHHMZ}.

\begin{defn}\label{el cuerpo V}
We will say that a subset $\mathcal{S}$ of the plane $\R^2$ is a smooth square provided that:
\begin{enumerate}
\item [(i)] $\mathcal{S}\subset \R^2$ is a bounded, symmetric convex
body with $0\in\textrm{int}(\mathcal{S})$, and whose boundary $\partial\mathcal{S}$ is $C^\infty$ smooth.
\item [(ii)] $(x,y)\in\partial\mathcal{S}\Leftrightarrow
(\epsilon_1 x,\epsilon_2 y)\in\partial\mathcal{S}$ for each couple
 $(\epsilon_1,\epsilon_2)\in\{-1,1\}^2$ (that is, $\mathcal{S}$ is symmetric about the cordinate axes).
\item [(iii)] $[-\frac{1}{2},\frac{1}{2}]\times\{-1,1\}\cup
 \{-1,1\}\times[-\frac{1}{2},\frac{1}{2}] \subset \partial\mathcal{S}$.
\item [(iv)] $\mathcal{S}\subset [-1,1]\times[-1,1].$
\end{enumerate}
\end{defn}
Of course, it is elementary to produce smooth squares in $\R^2$.

The following lemma enumerates the essential properties of a smooth square. Recall that the Minkowski functional of
a convex body $\mathcal{A}$ such that $0\in\textrm{int}(\mathcal{A})$ is defined by
$$\mu_{\mathcal{A}}(x)=\inf\{t>0 : \frac{1}{t}x\in \mathcal{A}\}.$$

\begin{lemma}\label{properties of smooth squares}
Let $\mathcal{S}\subset\R^2$ be a smooth square. Then its Minkowski functional $\mu_\mathcal{S}:\R^2\to\R$ is a
$C^\infty$ smooth norm on $\R^2$ such that, for every $(x,y)\in\mathbb{R}^2$, we have
\begin{enumerate}
\item $\mu_\mathcal{S}(x,y)\leq |x|+|y|\leq
2\mu_\mathcal{S}(x,y);$
\item $\max(|x|,|y|)\leq\mu_\mathcal{S}(x,y)\leq 2
\max(|x|,|y|);$
\item $\mu_\mathcal{S}(0,y)=|y|,\ \ \mu_{\mathcal{S}}(x,0)=
|x|$;
\item For every $(x_0,y_0)\in \R^2\setminus\{(0,0)\}$, there exists $\sigma>0$  so that
$$ \mu_\mathcal{S}(x,y)=|x|\ \ \text{if}\ \
\max(|x-x_0|,|y|)\leq\sigma\ \ \ \text{and} \ \ \ \mu_\mathcal{S}(x,y)=|y|\ \  \text{if}\ \
\max(|x|,|y-y_0|)\leq\sigma.
$$
\item The functions $(0,\infty)\ni t\to\mu_\mathcal{S}(x,ty)$ and  $(0,\infty)\ni t\to
\mu_\mathcal{S}(tx,y)$ are both nondecreasing.
\end{enumerate}
\end{lemma}
Note that property $(4)$ (which is related to properties (iii) and (iv) of Definition \ref{el cuerpo V}) means that
every sphere of $(\R^2, \mu_{\mathcal{S}})$ centered at the origin, which coincides with $\lambda(\partial
\mathcal{S})$ for some $\lambda>0$, is orthogonal to the coordinate axes and is locally flat on a neighborhood of
the intersection of $\lambda(\partial\mathcal{S})$ with the lines $\{x=0\}\cup\{y=0\}$. By using this property it
is easy to show that, for any couple of Banach spaces $(E_1, \|\cdot\|_1)$ and $(E_2, \|\cdot\|_2)$ with $C^p$
smooth norms, the expression
$$
\mu_{\mathcal{S}}\left(\|x_1\|_1, \|x_2\|_2\right)
$$
defines an equivalent norm of class $C^p$ in $E_1 \times E_2$.
\begin{proof}
Properties $(1)-(3)$ and $(5)$ are easy to show. Let us prove $(4)$. Assume for instance that $x_0\neq 0$, and set
$\sigma=|x_0|/4$. If $\max(|x-x_0|,|y|)\leq\sigma$, then we have
$$
\frac{|y|}{|x|}\leq \frac{\sigma}{|x_0|-\sigma}=\frac{|x_0|/4}{|x_0|-|x_0|/4}=\frac{1}{3}<\frac{1}{2},
$$
hence
$$
\left( \frac{x}{|x|}, \frac{y}{|x|}\right)\in \left\{-1,1\right\}\times \left[-\frac{1}{2},
\frac{1}{2}\right]\subset\partial\mathcal{S},
$$
and it follows that $\mu_{\mathcal{S}}(x,y)=|x|$.
\end{proof}

The lemma below shows how, with the help of a smooth square, we can combine the given $C^p$ smooth function
$\psi:W_1\to [0, \infty)$ together with the $C^p$ smooth functional $\omega:E_2\to[0,\infty)$ obtained in Lemma
\ref{functional omega}, in order to obtain a $C^p$ smooth function on $W_1\times E_2$ which behaves more or less
like $\psi(x_1)+\omega(x_2)$ (or, equivalently, like $\left( \psi(x_1)^{2}+ \omega(x_{2})^{2}\right)^{1/2}$).


\begin{lemma}\label{smooth squares mix bodies nicely}
Let $E=E_1\times E_2$ be a Banach space, $\rho_1: E_1\to [0, \infty)$ and $\rho_2: E_2\to [0, \infty)$ continuous
functions which are of class $C^p$ on $E_1\setminus\rho_{1}^{-1}(0)$ and $E_2\setminus\rho_{1}^{-1}(0)$,
respectively. Then, for any smooth square $\mathcal{S}$ of $\mathbb{R}^{2}$, the function $\rho:E_{1}\times
E_{2}\to [0, \infty)$ defined by
$$
\rho(x)=\rho(x_1, x_2)=\mu_\mathcal{S}(\rho_1(x_1),\rho_2(x_2)) ,\ \ x=(x_1,x_2)\in E_1\times E_2,
$$
is continuous on $E$ and of class $C^p$ on $E\setminus\left( \rho_{1}^{-1}(0)\times \rho_{2}^{-1}(0) \right)$.

The same is true if we replace $E_1$ with an open subset $W_1$ of $E_1$.
\end{lemma}
\begin{proof}
It is clear that $\rho$ is continuous on $E$, and that it is $C^p$ smooth on $\{(x_1, x_2)\in E : \rho_1(x_1)\neq
0\neq \rho_2(x_2)\}$. Let us see that $\rho$ is also $C^p$ smooth on a neighborhood of the set
$$
\left( \{(x_1, x_2)\in E : \rho_1(x_1)=0\}\cup \{(x_1, x_2)\in E : \rho_2(x_2)=0\}\right)\setminus
\left(\rho_{1}^{-1}(0)\times\rho_{2}^{-1}(0)\right).
$$
Suppose for instance that $\rho_{1}(x_1)\neq 0=\rho_{2}(x_2)$. Then, by continuity of $\rho_1, \rho_2$ and by
property $(4)$ of Lemma \ref{properties of smooth squares}, there exist a neighborhood $U$ of the point $(x_1,
x_2)$ such that $U\subset \{(y_1, y_2)\in E: \rho_1(y_1)\neq 0\}$ and
$$
\rho(y_1, y_2)= \rho_1(y_1)
$$
for all $(y_1, y_2)\in U$. It follows that $\rho$ is of class $C^p$ on $U$. The case $\rho_{1}(x_1)=0\neq
\rho_{2}(x_2)$ can be treated similarly.
\end{proof}

Now we are ready to prove Theorem \ref{general extracting scheme for closed subsets of a subspace of infinite
codimension}.

\begin{proof}[Proof of Theorem \ref{general extracting scheme for closed subsets of a subspace of infinite codimension}]
From now on we will fix a smooth square $\mathcal{S}$ on $\R^2$, and we will denote
$$
S=\mu_{\mathcal{S}}.
$$
Thus, by Lemma \ref{smooth squares mix bodies nicely}  applied to $\rho_1=\psi$ and $\rho_2=\omega$ (recall that
$\omega$ was constructed in Lemma \ref{functional omega}), the function
$$
\rho(x_1, x_2) :=S(\psi(x_1), \omega(x_2))
$$
is continuous on $W_1\times E_2$ and of class $C^p$ on $\left(W_1\times E_2\right) \setminus \left(
\psi^{-1}(0)\times\{0\}\right)$.

Let us define $h:(W_1\times E_2)\setminus K\to E$ by
$$
h(x_1, x_2)=\left(x_1, \, x_2+\gamma\circ\rho(x_1,x_2)\right)=\left( x_1, \, x_2 +\gamma\left(S(\psi(x_1),
\omega(x_2))\right) \right),\ \ (x_1,x_2)\in(W_1\times E_2)\setminus K,
$$
where $\gamma$ is provided by Lemma  \ref{deleting path}. Note that
$$
K= \psi^{-1}(0)\times\{0\}=\rho^{-1}(0). $$

Let $(y_1, y_2)$ be an arbitrary point of $W_1\times E_2$, and let $F_{y_1, y_2}:(0,\infty)\longrightarrow
[0,\infty)$ be defined by
\begin{equation}\label{definition of F}
F_{y_1, y_2}(\alpha) = \rho(y_1,y_2-\gamma(\alpha))=S\left(\psi(y_1), \omega(y_2-\gamma(\alpha)) \right)
\end{equation}
for $\alpha>0$. Let us see that $F_{y_1, y_2}(\alpha)$ satisfies the conditions of Lemma \ref{fixed point lemma}.
As for the first condition, we consider two cases: if $\omega(y_2-\gamma(\beta))\leq \omega(y_2-\gamma(\alpha))$
then, since the function $(0, \infty)\ni t\mapsto S(\psi(y_1), t)$ is increasing (see condition $(5)$ of Lemma
\ref{properties of smooth squares}), we have that
$$
S(\psi(y_1), \omega(y_2-\gamma(\beta))) \leq S(\psi(y_1), \omega(y_2-\gamma(\alpha))),
$$
and therefore
$$
F_{y_1, y_2}(\beta) - F_{y_1, y_2}(\alpha)\leq 0\leq\frac{1}{2}(\beta-\alpha)
$$
trivially for all $\beta\geq\alpha>0$. Otherwise, we have $\omega(y_2-\gamma(\beta))-
\omega(y_2-\gamma(\alpha))>0$, and therefore, using the fact that $S$ is a norm in $\R^2$, condition $(3)$ of Lemma
\ref{properties of smooth squares}, the properties of the functional $\omega$, and condition $(1)$ of Lemma
\ref{deleting path}, we obtain
\begin{eqnarray*}
& & S(\psi(y_1), \omega(y_2-\gamma(\beta)))- S(\psi(y_1), \omega(y_2-\gamma(\alpha))) \leq
S\left(0, \omega(y_2-\gamma(\beta))-\omega(y_2-\gamma(\alpha))\right)= \\
& & \omega(y_2-\gamma(\beta))- \omega(y_2-\gamma(\alpha))\leq
 \omega\big( y_2-\gamma(\beta)-(y_2-\gamma(\alpha))\big)
=\omega(\gamma(\alpha)-\gamma(\beta))\leq\frac{1}{2}(\beta-\alpha)
\end{eqnarray*}
for every $\beta\geq\alpha>0$. In either case we have that
\begin{equation}\label{F is semicontractive}
F_{y_1, y_2}(\beta) - F_{y_1, y_2}(\alpha)\leq\frac{1}{2}(\beta-\alpha)
\end{equation}
for all $\beta\geq\alpha>0$.

On the other hand, by condition $(2)$ of Lemma \ref{deleting path} we know that
$$
\limsup_{\alpha\to 0^{+}}\omega(y_2-\gamma(\alpha))>0,
$$
and therefore, by condition $(2)$ of Lemma \ref{properties of smooth squares}, we have
$$
\limsup_{\alpha\to 0^{+}}F_{y_1, y_2}(\alpha)= \limsup_{\alpha\to 0^{+}}S(\psi(y_1),
\omega(y_2-\gamma(\alpha)))\geq \limsup_{\alpha\to 0^{+}}\omega(y_2-\gamma(\alpha))>0,
$$
so that $F_{y_1, y_2}$ also satisfies the second condition of Lemma \ref{fixed point lemma}.

Then, applying Lemma \ref{fixed point lemma}, we deduce that the equation  $F_{y_1, y_2}(\alpha)=\alpha$ has a
unique solution. This means that, for each $(y_1, y_2)\in W_1\times E_2$, a number $\alpha(y_1, y_2)>0$ with the
property
\begin{equation}\label{alpha is a fixed point}
S\left(\psi(y_1), \omega\big( y_2-\gamma(\alpha(y_1, y_2))\big)\right)=\rho\left(y_1,y_2-\gamma(\alpha(y_1,
y_2))\right)=\alpha(y_1, y_2),
\end{equation}
is uniquely determined.

Let us see why these facts imply that $h$ is a $C^{p}$ diffeomorphism from $W_{1}\times E_2 \setminus K$ onto
$W_1\times E_2$. Assume first that $h(x_1, x_2)=(y_1, y_2)=h(z_1, z_2)$, that is to say $x_1=y_1=z_1$, and
\begin{equation}\label{one-to-one}
 x_2 +\gamma\left(\rho(y_1,x_2)\right)=y_2=z_2 + \gamma\left(
\rho(y_1,z_2)\right),
\end{equation}
or equivalently
$$
(y_1, x_2)=\left( y_1, y_2-\gamma(\rho(y_1,x_2)\right) \ \ \text{and}\ \ (y_1,z_2)=\left( y_1, y_2-\gamma(
\rho(y_1,z_2)\right).
$$
Applying $\rho$ to all sides of the above equations and using \eqref{definition of F}, we obtain $$ \rho
(y_1,x_2)=\rho\left( y_1, y_2-\gamma(\rho(y_1,x_2))\right)=F(\rho (y_1,x_2)) \ \text{and}\ \rho
(y_1,z_2)=\rho\left( y_1, y_2-\gamma(\rho(y_1,z_2))\right)=F(\rho (y_1,z_2)).
$$
 It follows that both $\rho (y_1,x_2)$ and $\rho (y_1,z_2)$ are
fixed points of $F_{y_1, y_2}$. By the uniqueness of the fixed point, we conclude that
$$
\alpha(y_1, y_2)=\rho (y_1,x_2)=\rho (y_1,z_2).
$$
Now applying \eqref{one-to-one}, we have
$$
x_2=z_2\ \ \text{and}\ \  x_2=y_2-\gamma(\alpha(y_1, y_2)).
$$
This shows that $h$ is one to one, and also that, given $(y_1, y_2)\in W_1\times E_2$ we have
$$
h\left(y_1, y_2-\gamma(\alpha(y_1, y_2))\right)=(y_1, y_2).
$$
Hence $h$ is also onto, and $h^{-1}:W_1\times E_2\to W_{1}\times E_2 \setminus K$ is given by
$$
h^{-1}(y_1, y_2)=\left(y_1, y_2-\gamma(\alpha(y_1, y_2))\right).
$$
It is clear that $h$ is of class $C^p$. In order to see that $h^{-1}$ is $C^p$ as well, let us define $\Phi:
W_1\times E_2\times (0,\infty)\longrightarrow\mathbb{R}$ by
$$
\Phi(y_1, y_2, \alpha)=\alpha -S(\psi(y_1), \omega(y_2-\gamma(\alpha))) =\alpha-\rho(y_1,y_2-\gamma(\alpha)).
$$
On the one hand, according to \eqref{alpha is a fixed point} and the fact that $S$ is a norm in $\R^2$, we have
$$
(\psi(y_1), \omega(y_2-\gamma(\alpha(y_1, y_2)))\neq (0,0)
$$
for every $(y_1, y_2)\in W_1\times E_2$. Since $S$ is $C^{\infty}$ smooth away from $(0,0)$, this implies that
$\Phi$ is $C^p$ smooth on a neighborhood of every  point $(y_1, y_{2},\alpha(y_1, y_{2}))$ in $W_1\times E_2\times
(0,\infty)$. On the other hand, we know that $F_{y_1, y_2}(\beta)-F_{y_1,
y_2}(\alpha)\leq\frac{1}{2}(\beta-\alpha)$ for $\beta\geq\alpha>0$, which implies that $F_{y_1,
y_2}'(\alpha)\leq\frac{1}{2}$ for every $\alpha$ in a neighborhood of $\alpha(y_1, y_2)$, and therefore
$$
\frac{\partial\Phi(y_1, y_2,\alpha)}{\partial\alpha}=1-F_{y_1, y_2}'(\alpha)\geq 1-1/2>0.
$$
Hence, by the implicit function theorem, the mapping $(y_1, y_2)\rightarrow\alpha(y_1, y_2)$ is of class $C^{p}$ on
$W_1\times E_2$, and, since $\gamma$ is $C^{p}$ smooth, so is $h^{-1}$.

Finally, it is obvious that $\pi_1\circ h=\pi_1$, and the fact that $\gamma(t)=0$ whenever $t\geq 1$ implies that
$h$ is the identity off of the set $\{x=(x_1, x_2)\in W_{1}\times E_2 \, : \, S(\psi(x_1), \omega(x_2))<1\}$.
\end{proof}

\medskip

\subsection{Extracting pieces of continuous graphs of infinite codimension}

Now we will prove the following extractibility result.

\begin{thm}\label{Renz's Theorem 4 for smooth graphs on sets of zeros of smooth functions}
Let $E=E_1\times E_2$ be a product of Banach spaces such that $E_1$ admits $C^p$ smooth partitions of unity and
$E_2$ admits a $C^p$ (not necessarily equivalent) norm. Assume that $X_1$ is a closed subset of $E_1$, that $f:
E_1\to E_2$ is a continuous mapping, and that $E_2$ is infinite-dimensional. Define
$$
X=\{ (x_1, x_2)\in E_{1}\times E_{2} \, : \, x_1\in X_1, x_2=f(x_1)\}.
$$
Let $U$ be an open subset of $E$ and $\varepsilon>0$. Then there exists a $C^p$ diffeomorphism $g$ from $E\setminus
X$ onto $E\setminus (X\setminus U)$ such that $g$ is the identity on $(E\setminus U)\setminus X$ and moves no point
more than $\varepsilon$.
\end{thm}
\begin{proof}
We may of course assume $U\cap X\neq\emptyset$ (as otherwise the result holds trivially with $g$ equal to the
identity map).
\begin{claim}\label{Claim that we may assume f equal to 0}
It is sufficient to prove the result for $f=0$ and such that the extracting diffeomorphism preserves the first
coordinate, that is $g(x_1,x_2)=(x_1,\pi_2(g(x_1,x_2)))$.
\end{claim}
\begin{proof}
Let $h$ and $\varphi$ be homeomorphisms given by Theorem \ref{Renz's theorem 1 in our case, healed version} such
that $||\varphi^{-1}(x)-h^{-1}(x)||\leq {\varepsilon\over 2}$ for every $x\in E$. By the uniform continuity of
$h^{-1}(x_1,x_2)$ with respect to the second variable $x_2\in E_2$, we may choose $\delta>0$ such that if
$||(x_1,x_2)-(x_1,x_2')||\leq\delta$ then $$||h^{-1}(x_1,x_2)-h^{-1}(x_1,x_2')||\leq{\varepsilon\over 2}.$$
Assuming the result is true for $f=0$ we can find a $C^p$ diffeomorphism $g: E\setminus (X_1 \times\{0\})\to
E\setminus ((X_1 \times\{0\})\setminus h(U))$ such that $g$ is the identity on $(E\setminus h(U))\setminus
(X_1\times\{0\})$,  moves no point more than $\delta$ and preserves the first coordinate. Then the composition
$$
\varphi^{-1}\circ g\circ h: E\setminus X\to (E\setminus (X\setminus U)
$$
defines a $C^p$ diffeomorphism with the required properties. Observe that
$$||g(h(x))-h(x)||=||(x_1,\pi_2(g(h(x))))-(x_1,\pi_2(h(x)))||\leq\delta,$$ hence
$$||\varphi^{-1}(g(h(x)))-x||\leq||\varphi^{-1}(g(h(x)))-h^{-1}(g(h(x)))||+||h^{-1}(g(h(x)))-x||\leq{\varepsilon\over 2}+{\varepsilon\over 2}=\varepsilon$$
for every $x\in E\setminus X$.
\end{proof}

So it will be enough to see that if $X_1$ is a closed subset of $E_1$ and $W$ is an open subset of $E$ such that
$W\cap X_{1}\times \{0\}\neq\emptyset$ then there exists a $C^p$ diffeomorphism $g$ from $E\setminus (X_1\times
\{0\})$ onto $E\setminus ((X_1\times\{0\})\setminus W)$ such that $g$ is the identity on $(E\setminus W)\setminus
(X_1\times\{0\})$ and moves no point more than $\delta$.

To this end we next construct some auxiliary functions following Renz's strategy \cite[pp. 54-59]{Renz}. In what
follows  $\omega$ will denote the smooth asymmetric subadditive functional on $E_2$ given by Lemma \ref{functional
omega}.
\begin{lemma}\label{existence of varphi in Renz's 4}
There exists a continuous function $\varphi:E_1\to [0,{\delta\over 2}]$ such that:
\begin{enumerate}
\item $\varphi$ is of class $C^p$ on $E_1\setminus \partial \varphi^{-1}(0)$.
\item $W\cap (X_1\times\{0\})\subset\{ (x_1, x_2)\in E \, : \, \|x_2\|<\varphi(x_1)\}
\subset W$.
\end{enumerate}
\end{lemma}
\begin{proof}
Let $\pi_1: E\to E_1$ denote the canonical projection defined by $\pi_1(x_1, x_2)=x_1$. The set
$$
W_1 :=\pi_{1}\left( W\cap (E_1 \times\{0\})\right)
$$
is open in $E_1$, and the function $G:E_1\to [0, \infty)$ defined by
$$
G(x_1)=\min\{{\delta\over 2}, \textrm{dist}\left( (x_1, 0), E\setminus W\right)\}
$$
is continuous and satisfies that $G>0$ on $W_1$ and $G=0$ on $\pi_1 \left( (E_1\times \{0\})\setminus W\right)$.
Since $E_1$ has $C^p$ smooth partitions of unity and $G$ is continuous and strictly positive on $W_1$, we can find
a $C^p$ smooth function $F$ on $W_1$ such that
$$
0<\frac{1}{4} G(x_1) < F(x_1)<\frac{1}{2} G(x_1)
$$
for every $x_1\in W_1$. Now let us define $\varphi: E_1\to [0, 1]$ by
$$
\varphi(x_1)=
\begin{cases}
F(x_1) & \mbox{ if } x_1\in W_1 \\
0  & \mbox{ if }  x_1\in E_1\setminus W_1.
\end{cases}
$$
It is immediately seen that $\varphi$ is continuous, and of course $\varphi$ is of class $C^p$ on
$E_1\setminus\partial \varphi^{-1}(0)=W_1\cup \textrm{int}(E_1\setminus W_1)$. Since $\varphi(x_1)=F(x_1)>0$ for
all $x_1\in W_1$, it is obvious that
$$
W\cap (X_1\times\{0\})\subset W_1\times\{0\}\subset \{ (x_1, x_2)\in E \, : \, \|x_2\|<\varphi(x_1)\}.
$$
On the other hand, if $\|x_2\|<\varphi(x_1)$ then observe first that $x_1\in W_1$ (because if $\varphi(x_1)=0$ the
inequality is impossible). We then must have $(x_1, x_2)\in W$, as otherwise we would get
\begin{eqnarray*}
& & \textrm{dist}\left( (x_1, 0), E\setminus W\right)\leq \textrm{dist}\left( (x_1, 0),(x_1, x_2) \right)+
\textrm{dist}\left( (x_1, x_2), E\setminus W\right)=\\
& & \|x_2\|+0=\|x_2\|<\varphi(x_1)<\frac{1}{2}G(x_1)\leq \frac{1}{2} \textrm{dist}\left( (x_1, 0), E\setminus
W\right),
\end{eqnarray*}
which is absurd. This shows that $\{ (x_1, x_2)\in E \, : \, \|x_2\|<\varphi(x_1)\} \subset W$ and concludes the
proof of the lemma.
\end{proof}

We will also need to use a diffeomorphism $h_2$ of $E_2$ onto itself which carries the unit ball of $E_2$ onto the
convex body $\{x_2\in E_2 \, : \, \omega(x)\leq 1\}$ and such that $h_2(0)=0$. The existence of $h_2$ is ensured by
the following lemma. We say that a convex body $U$ which contains $0$ as an interior point is {\em radially
bounded} provided that for every $x\in U$ the set $\{tx : t\in [0, \infty)\}\cap U$ is bounded.

\begin{lemma}\label{convex sets lemma}
Let $X$ be a Banach space, and let $U_{1}, U_{2}$ be radially bounded, $C^{p}$ smooth convex such that the origin
is an interior point of both $U_{1}$ and $U_{2}$. Then there exists a $C^{p}$ diffeomorphism $g:X\longrightarrow X$
such that $g(U_{1})=U_{2}$, $g(0)=0$, and $g(\partial{U_{1}})=\partial{U_{2}}$.
\end{lemma}
\begin{proof}
If $U$ and $V$ are $C^p$ smooth, radially bounded convex bodies such that the origin is an interior point of both
$U$ and $V$, and we additionally assume that $U\subseteq V$, such a diffeomorphism can be constructed as follows:
let $\theta(t)$ be a non-decreasing real function of class $C^{\infty}$ defined for $t>0$, such that $\theta(t)=0$
for $t\leq 1/2$ and $\theta(t)=1$ for $t\geq 1$, and define
$$
g(x)=\left(\theta(\mu_{U}(x))\frac{\mu_{U}(x)}{\mu_{V}(x)} +1-\theta(\mu_{U}(x))\right)x
$$
for $x\neq 0$, and $g(0)=0$. Here $\mu_A$ denotes the  Minkowski functional of $A$.

In the general case, let $U=\{x\in X \, : \, \mu_{U_{1}}(x)+\mu_{U_{2}}(x)\leq 1\}$, then $U\subseteq U_{j}$, for
$j=1, 2$, and there exist diffeomorphisms $g_{1}, g_{2}:X\to X$ such that $g_{j}(U)=U_{j}$ and
$g_{j}(\partial{U})=\partial{U_{j}}$, $j=1, 2$. Then $g=g_{2}\circ g_{1}^{-1}$ does the job. See
\cite{Dobrowolski1977} for details.
\end{proof}

The following lemma is an immediate consequence of the existence of partitions of unity in $E_1$.
\begin{lemma}\label{if the space has partitions of unity then psi exists}
Suppose that $E_1$ is a Banach space with $C^p$ smooth partitions of unity, and let $X_1$ be a closed subset of
$E_1$. Then there exists a continuous function $\eta:E_1\to [0, \infty)$ such that:
\begin{enumerate}
\item $X_1=\eta^{-1}(0)$;
\item $\eta$ is of class $C^p$ on $E_1\setminus X_1$.
\end{enumerate}
\end{lemma}

We are ready to proceed with the proof of Theorem \ref{Renz's Theorem 4 for smooth graphs on sets of zeros of
smooth functions}. Let $\eta$ be a function as in the statement of Lemma \ref{if the space has partitions of unity
then psi exists}, and pick a diffeomorphism $h_2:E_2\to E_2$ such that $h_2(0)=0$ and
$$
h_2 \left(\{x_2\in E_2 : \|x_2\|\leq 1\}\right)=\{x_2\in E_2 \, : \, \omega(x_2)\leq 1\}.
$$
Let us define
$$
A :=\varphi^{-1}\left((0,1]\right)\times E_2=W_1\times E_2,
$$
and
$$
\Phi(x_1, x_2)=\left( x_1, h_2\left(\frac{1}{\varphi (x_1)} x_2\right)\right), \,\,\, (x_1, x_2)\in A.
$$
It is clear that $\Phi:A\to A$ is a $C^p$ diffeomorphism, with inverse
$$
\Phi^{-1}(y_1, y_2)=\left( y_1, \varphi(y_1) h_{2}^{-1}(y_2)\right),
$$
and also that
$$
\Phi\left( (X_1\times\{0\})\cap A\right)=(X_1\times\{0\})\cap A.
$$
Next let us define $\psi: \varphi^{-1}\left( (0, 1]\right)=\pi_1(A)=W_1 \to [0, \infty)$ by
$$
\psi(x_1)=\frac{\eta(x_1)}{\varphi(x_1)},
$$
and notice that $\psi$ is continuous, that
$$
\psi^{-1}(0)=\pi_1 \left( (X_{1}\times \{0\})\cap A\right)=X_{1}\cap W_1,
$$
and that $\psi$ is of class $C^p$  outside $\psi^{-1}(0)$. By Theorem \ref{general extracting scheme for closed
subsets of a subspace of infinite codimension} we can find a $C^p$ diffeomorphism $H$ from
$A\setminus(X_1\times\{0\})$ onto $A$ such that $H$ is the identity outside $\{ (x_1, x_2)\in
A\setminus(X_1\times\{0\})\, : \, S(\psi(x_1), \omega(x_2))<1\}$, where $S$ is a smooth square. Since $\Phi:A\to A$
is a $C^p$ diffeomorphism which takes $(X_1\times\{0\})\cap A$ onto itself, we have that the composition
$\Phi^{-1}\circ H\circ \Phi$ defines a $C^p$ diffeomorphism from $A\setminus (X_1\times\{0\})$ onto $A$. Now we
extend this diffeomorphism outside $A\setminus (X_1\times \{0\})$ by defining $g:E\setminus (X_1\times \{0\})\to E$
by
$$
g(x)=
\begin{cases}
\Phi^{-1}\circ H\circ \Phi(x) & \mbox{ if } x\in A\setminus (X_1\times \{0\}) \\
x  & \mbox{ if }  x\in E\setminus (A\cup (X_1\times\{0\})).
\end{cases}
$$
This mapping is clearly a bijection. Thus, in order to see that $g$ is a $C^p$ diffeomorphism, it is enough to see
that $g$ is locally a $C^p$ diffeomorphism. We already know this is so for all points of $E\setminus \big(\partial
(A \setminus (X_1\times\{0\}))\cup(X_1\times\{0\})\big)$. Let us show that this is also true for every point $(x_1,
x_2)$ of $\partial (A \setminus (X_1\times\{0\})) \setminus (X_1\times\{0\})$. We have $\varphi(x_1)=0$, and also
either $\eta(x_1)>0$ or $\|x_2\|>0$. Then
$$
\lim_{A\ni (y_1, y_2)\to (x_1, x_2)}\max\left\{\frac{\eta(y_1)}{\varphi(y_1)},
\frac{\|y_2\|}{\varphi(y_1)}\right\}=\infty,
$$
hence there exists a neighborhood $V$ of $(x_1, x_2)$ in $E\setminus (X_1\times\{0\})$ such that
$$
\max\left\{\frac{\eta(y_1)}{\varphi(y_1)}, \frac{\|y_2\|}{\varphi(y_1)}\right\}>1 \,\,\, \textrm{ for all } (y_1,
y_2)\in V\cap A.
$$
By Lemma \ref{properties of smooth squares}(2) it follows that
$$
S\left( \psi\circ\pi_{1}\circ\Phi(y), \omega\circ\pi_2\circ\Phi(y)\right)>1 \,\,\, \textrm{ for all } y\in V\cap A,
$$
hence that $H$ is the identity on $\Phi(V\cap A)$, and consequently that $g$ is the identity on $V$, and in
particular a $C^p$ diffeomorphism locally at $(x_1, x_2)$. Thus $g:E\setminus (X_1\times \{0\})\to E$ is a $C^p$
diffeomorphism.

Furthermore, if $\|x_2\|\geq \varphi(x_1)$, $(x_1, x_2)\in A\setminus (X_1\times\{0\})$, then we have
$\omega(\pi_2\circ\Phi(x_1, x_2))\geq 1$. Hence, as above by Lemma \ref{properties of smooth squares}(2), we
conclude that $H(\Phi(x_1, x_2))=\Phi(x_1, x_2)$, and it follows that $g(x_1, x_2)=(x_1, x_2)$. Thus $g$ is the
identity off of the set $\{(x_1, x_2)\in E\setminus(X_1\times\{0\}) \, : \, \|x_2\|<\varphi(x_1)\}$. Since
$$
\{(x_1, x_2)\in E\setminus(X_1\times\{0\}) \, : \, \|x_2\|<\varphi(x_1)\} \subset (E\setminus W)\setminus
(X_1\times\{0\}),
$$
$g$ is the identity off of the set $(E\setminus W)\setminus (X_1\times\{0\})$ as well.

Finally let us check that $g$ does not move any point more than $\delta$. We know that if $g$ moves a point
$(x_1,x_2)\in E\setminus (X_1\times\{0\})$ then $||x_2||<\varphi(x_1)$, and also that $g_2$ only moves the second
coordinate $x_2$, that is $g(x_1,x_2)=(x_1,\pi_2(g(x_1,x_2)))$. Hence
$$||g(x_1,x_2)-(x_1,x_2)||\leq ||\pi_2(g(x_1,x_2))-x_2||\leq ||\pi_2(g(x_1,x_2))||+||x_2||\leq \varphi(x_1)+\varphi(x_1)\leq\delta.$$
The proof of Theorem \ref{Renz's Theorem 4 for smooth graphs on sets of zeros of smooth functions} is complete.
\end{proof}

\medskip
\subsection{Patching local diffeomorphisms together}

Throughout this section $E$ will be a Banach space. Our goal in this section is to complete the proof of Theorem
\ref{final extractibility theorem rough version general form}.

First let us introduce the following definitions.

\begin{defn}\label{definition of extraction property}
{\em We will say that a subset $X$ of $E$ has the strong $C^p$ extraction property  with respect to an open set $U$
if $X\subseteq U$, $X$ is relatively closed in $U$, and for every open set $V\subseteq U$ and every subset $Y\subseteq
X$ relatively closed in $U$ there exists a $C^{p}$ diffeomorphism $\varphi$ from $U\setminus Y$ onto
$U\setminus(Y\setminus V)$ which is the identity
on $(U\setminus V)\setminus Y$. If in addition for any $\varepsilon>0$ we can require the diffeomorphism $\varphi$ not to move any point more than $\varepsilon$, we will say that $X$ has the $\varepsilon$-strong $C^p$ extraction property with respect to $U$.\\
We will also say that such a closed set $X$ has locally the strong (or the $\varepsilon$-strong) local $C^p$ extraction
property if for every point $x\in X$ there exists an open neighborhood $U_x$ of $x$ such that $X\cap
\overline{U_x}$ has the  strong ($\varepsilon$-strong, respectively) $C^p$ extraction property with respect to every
open set $U$ with $X\cap \overline{U_x}\subseteq U$. (Equivalently, there exists an open neighborhood $U_x$ of $x$
such that $X\cap {U_x}$ has the strong ($\varepsilon$-strong respectively) $C^p$ extraction property with respect
to every open set $U$ for which $X\cap{U_x}$ is a relatively closed subset of $U$.)}
\end{defn}

\begin{rem}
Let $(U,W)$, $W\subset U$, be a pair of open sets in a Banach space $E$. We say that $(U,W)$ has the strong $C^p$
expansion property if, for every open subsets $V$ and $U'$ of $U$, $W\subset U'$, there exists a $C^{p}$
diffeomorphism $U'\cap V\to V$ which, by letting $\varphi(x)=x$ for $x\in U'\setminus V$, extends to $C^p$
diffeomorphism $\varphi:U'\to U'\cup V$.

In particular, letting $U'=W$, there exists a $C^p$ diffeomorphism $W\cap V\to V$ which extends to a $C^p$
diffeomorphism of $W$ onto $W\cup V$ via the identity off $W\cap V$. Hence, $W$ is smoothly expanded to $W\cup V$;
this justifies the term of $C^p$ expansion. Should this expansion be valid for all open sets $U'$, $W\subset U'$,
then we would have the strong $C^p$ expansion property.

Notice that a relatively closed subset $X$ has the strong $C^p$ extraction property with respect to $U$ if and only
if $(U,W)=(U,U\setminus X)$ has the strong $C^p$ expansion property.
\medskip

We say that an open subset $W$ of $E$ has locally the strong $C^p$ expansion property if every $x\in E\setminus W$
has an open neighborhood $U_x$ such that $(U,U_x\cap W)$ has the strong expansion property for every open set
$U\supset U_x\cap W$.

Notice that a closed set $X$ has locally the strong $C^p$ extraction property if and only if $W=E\setminus X$ has
locally the strong $C^p$ expansion property.
\end{rem}

Some basic properties that can be derived from Definition \ref{definition of extraction property} are listed in the
following lemma.

\begin{lem}\label{Properties of C^p extractibility}
Let us suppose that $X\subset E$ has the $\varepsilon$-strong $C^p$ extraction property with respect to an open set
$U$ of $E$. Then:
\begin{enumerate}
\item For every closed set $Y\subseteq X$, $Y$ has the $\varepsilon$-strong
$C^p$ extraction property with respect to $U$;
\item For every open subset $U'\subseteq U$, $X\cap U'$ has
the $\varepsilon$-strong $C^p$ extraction property with respect to $U'$.
\item If $h$ is a $C^p$ diffeomorphism defined on $U$ and such that $h(U)$ is open,
then $h(X)$ has the strong $C^p$ extraction property with respect to $h(U)$.
\end{enumerate}
\end{lem}

\begin{proof}\hspace{6cm}

(1) This follows directly from the definition.\\

(2) Take an open subset $V'\subseteq U'$, a subset $Y\subseteq X\cap U'$ relatively closed in $U'$. Since $X$ has
the strong $C^p$ extraction property  with respect to $U$ there exists a $C^{p}$ diffeomorphism $\varphi$ from
$U\setminus \overline{Y}$ onto $U\setminus(\overline{Y}\setminus V')$ which is the identity on $(U\setminus
V')\setminus \overline{Y}$.  When restricting $\varphi$ to $U'\setminus Y$ we actually get a $C^{p}$ diffeomorphism
from $U'\setminus Y$ onto $U'\setminus (Y\setminus V')$ which is the
identity on $(U'\setminus V')\setminus Y$. \\

(3) Take an open subset $V$ of $h(U)$, a subset $Y\subseteq h(X)$ relatively closed in $h(U)$.  Since $X$ has the
strong $C^p$ extraction property  with respect to $U$ and $h^{-1}(Y)$ is relatively closed in $U$, there exists a
$C^{p}$ diffeomorphism $\varphi$ from $U\setminus h^{-1}(Y)$ onto $U\setminus(h^{-1}(Y)\setminus h^{-1}(V))$ which
is the identity on $(U\setminus h^{-1}(V))\setminus h^{-1}(Y)$. Then the mapping
$$g:=h\circ\varphi\circ h^{-1}:h(U)\setminus Y\longrightarrow h(U)\setminus (Y\setminus V)$$
is a surjective $C^{p}$ diffeomorphism which restricts to the identity on $(h(U)\setminus V)\setminus Y$.
\end{proof}
\begin{rem}\label{composition of diffeos do not move the points too much}
{\em In Lemma \ref{Properties of C^p extractibility} (3) we do not have in general the $\varepsilon$-strong $C^p$
extraction property of $h(X)$ with respect to $h(U)$, but we still have the following: suppose $h$ does not move any points more than some $\varepsilon>0$. For every $\eta>0$ in the proof of Lemma \ref{Properties of C^p
extractibility} (3) we can assume that $\varphi$ does not move any point more than $\eta$. Hence $g\circ
h=(h\circ \varphi\circ h^{-1})\circ h=h\circ\varphi$ does not move any point more than $\varepsilon+\eta$.}
\end{rem}
Let us state the main result of this section, which is crucial in the proof of Theorem \ref{final extractibility
theorem rough version general form} provided underneath.
\begin{thm}\label{abstract genereal extractibility}
Let $E$ be a Banach space and $X$ be a closed subset of $E$ which has locally the $\varepsilon$-strong $C^p$
extraction property. Let $U$ be an open subset of $E$ and $\mathcal{G}=\left\lbrace G_r\right\rbrace _{r\in\Omega}$
be an open cover of $E$. Then there exists a $C^p$ diffeomorphism $g$ from $E\setminus X$ onto $E\setminus
(X\setminus U)$ which is the identity on $(E\setminus U)\setminus X$ and is limited by $\mathcal{G}$.
\end{thm}

\begin{proof}[Proof of Theorem \ref{final extractibility
theorem rough version general form}] Let us show that $X$ from the statement of Theorem \ref{final extractibility
theorem rough version general form} has locally the strong $C^p$ extraction property.

To this end, fix $x\in X$ and choose a neighborhood $U_x$ such that $X\cap U_x\subset G(f_x)$; we can assume that
$U_x=\overline{U_x}$. Further, we can assume that $f_x$ is defined and continuous on the whole $E_{(1,x)}$. We will
show that $X':=X\cap \overline{U_x}$ has the $\varepsilon$-strong $C^p$ extraction property with respect to every
open set $U$ with $X'\subseteq U$. Notice that $X'=G(f_x|X'_1)$ for a certain closed $X'_1\subset E_{(1,x)}$.
Furthermore, if $Y'\subset X'$ is relatively closed in $U$, then $Y'$ is closed in $X'$. Hence, $Y'=G(f_x|Y'_1)$
for a certain closed $Y'_1\subset X'_1\subset E_{(1,x)}$. Let $V$ be an open subset of $U$. Take now
$\varepsilon>0$ and  apply Theorem \ref{Renz's Theorem 4 for smooth graphs on sets of zeros of smooth functions} to
$E_1:=E_{(1,x)}$, $E_2:=E_{(2,x)}$, $f:=f_x$, $X:=Y'$, and $V$ (in place of $U$) to obtain a $C^p$ diffeomorphism
$g$ from $E\setminus Y'$ onto $E\setminus (Y'\setminus V)$ such that $g$ is the identity on $(E\setminus
V)\setminus Y'$ and moves no point more than $\varepsilon$. Then $\varphi=g|U$ is as required in the definition of
the $\varepsilon$-strong $C^p$ extraction of $X'$ with respect to $U$.

Now, an application of Theorem \ref{abstract genereal extractibility} concludes our proof.
\end{proof}

The remaining part of this section is devoted to proving Theorem \ref{abstract genereal extractibility}. Firstly,
the fact that we are working with a set $X$ that has locally the strong $C^p$ extraction property, and the
requirement that our final $C^p$ diffeomorphism must be limited by a given open cover $\mathcal{G}$ forces us to
employ {\em good} refinements of covers of the Banach space $E$. In the separable case star-finite refinements
provide an adequate tool to face the problem (see West \cite{West}). Recall that a cover is said to be {\em
star-finite} provided that each element of the cover intersects at most finitely many others. However, in the
nonseparable case, getting a star-finite refinement of an open cover, in general, is not possible. We will use
sigma-discrete refinements as shown in the following.

\begin{lem}\label{sigma-discrete refinements}
Let $E$ be a Banach space and $X$ be a closed subset of $E$ which has locally the $\varepsilon$-strong $C^p$
extraction property. Let $\mathcal{G}=\left\lbrace G_r\right\rbrace_{r\in\Omega}$ be an open cover of $E$, where
the cardinality of the indexing set $\Omega$ is the density of $E$. Then there exist collections $\left\lbrace
X_i\right\rbrace _{i\geq 1}$, $\left\lbrace W_i\right\rbrace _{i\geq 1}$, $\left\lbrace V_i\right\rbrace _{i\geq
1}$, such that:

\begin{enumerate}
\item $X_i\subseteq W_i\subseteq \overline{W_i}\subseteq V_i$ for all $i\in\mathbb{N}$
;
\item $\left\lbrace V_i\right\rbrace _{i\geq 1}$ and $\left\lbrace W_i\right\rbrace _{i\geq 1}$
are star-finite open covers of $E$;
\item $\left\lbrace X_i\right\rbrace _{i\geq 1}$ is a cover of $X$ by closed subsets of
$X$;
\item  Each $W_i$ and $V_i$ admits an open discrete cover $\{W_{i,r}\}_{r\in \Omega}$
and $\{V_{i,r}\}_{r\in \Omega}$, repectively; more precisely,
$$
W_i=\bigcup_{r\in \Omega}W_{i,r}\ \ \text{and}\ \ V_i=\bigcup_{r\in\Omega}V_{i,r},
$$
$$
\overline{W_{i,r}}\subseteq V_{i,r}\ \ \text{for\  every}\ \ r\in\Omega,
$$ and
$$
\operatorname{dist}(V_{i,r},V_{i,r'})\geq {1\over 2^{i+1}}\ \ \text{for\ every}\ \ r,r'\in\Omega, r\neq r';
$$
\item $\left\lbrace W_{i,r}\right\rbrace _{i\geq 1,r\in \Omega}$ and
$\left\lbrace V_{i,r}\right\rbrace _{i\geq 1,r\in\Omega}$ are open refinements of $\mathcal{G}$;
\item Each $X_i$ can be written as $X_i=\bigcup_{r\in \Omega}X_{i,r}$, where $X_{i,r}$ is a closed subset of
$X$ satisfying the following requirements
$$
X_{i,r}\subseteq W_{i,r}\subseteq \overline{W_{i,r}}\subseteq V_{i,r}
$$
and
$$
X_{i,r}\ \  \text{has\  the}\ \varepsilon-\text{strong}\ \ C^p\ \ \text{extraction\ property\ with\ respect\ to}\ \
V_{i,r}.
$$
\end{enumerate}
\end{lem}
\begin{proof}
For each $x\in E$, let $U_x$ be an open neighborhood of $x$ such that $\overline{U_x}\subseteq G$ for some $G\in
\mathcal{G}$ and also satisfying that
\begin{enumerate}
\item if $x\in X$ then $X\cap \overline{U_x}$ has the $\varepsilon$-strong $C^p$ extraction property
with respect to every open set $U$ with $X\cap \overline{U_x}\subseteq U$, and
\item if $x\notin X$ then $X\cap\overline{U_x}=\emptyset$.
\end{enumerate}
Since the cardinality of $\Omega$ is the density of $E$, we can extract a subcover $\mathcal{U}=\left\lbrace
U_r:\,r\in\Omega\right\rbrace $ from $\left\lbrace U_x:\,x\in E\right\rbrace$. Now we use a result of Rudin
\cite{Rudin} (see also \cite[p. 390]{HajJoh}) to obtain two open refinements $\left\lbrace A_{j,r}\right\rbrace
_{j\geq 1,r\in\Omega}$ and $\left\lbrace B_{j,r}\right\rbrace _{j\geq 1,r\in\Omega}$ of $\mathcal{U}$ such that
\begin{enumerate}
\item $A_{j,r}\subseteq B_{j,r}\subseteq U_r$ for all $j\in\mathbb{N}$ and
$r\in\Omega$;
\item $\operatorname{dist}(A_{j,r},E\setminus B_{j,r})\geq{1\over 2^{j}}$ for all $j\in\mathbb{N}$ and
$r\in\Omega$;
\item $\textrm{dist}(B_{j,r},B_{j,r'})\geq{1\over 2^{j+1}}$ for all $j\in\mathbb{N}$ and  $r,r'\in\Omega$, $r\neq
r'$;
\item Letting $A_j=\bigcup_{r\in\Omega}A_{j,r}$ and $B_j=\bigcup_{j\in\Omega}B_{j,r}$
each collection $\left\lbrace A_{j}\right\rbrace _{j\geq 1}$, $\left\lbrace B_{j}\right\rbrace _{j\geq 1}$  forms a
locally finite open cover of $E$.
\end{enumerate}
Observe that $\overline{A_j}\subseteq B_j$ for every $j\in\N$.

For every $j$, there exists a sequence of open sets $B^{n}_{j}$, $n\ge j$, so that
$$
\overline{A_j}\subset B^{j}_{j}\subset \overline{B^{j}_{j}}\subset B^{j+1}_{j}\subset
\overline{B^{j+1}_{j}}\subset\cdots \subset {B^{n}_{j}}\subset \overline{B^{n}_{j}}\subset
B^{n+1}_{j}\subset\cdots\subset B_j.
$$

 For each $j$, write $\mathcal{B}_j=\left\lbrace B^{n}_{j}: \, n\ge
j\right\rbrace $. Clearly, $\mathcal{B}=\bigcup_{j=1}^\infty \mathcal{B}_j$ is an open cover of $E$; likewise, the
family $\left\lbrace \overline{B}\cap X: \, B\in\mathcal{B}\right\rbrace $ is a closed cover of $X$.

Defining for each $n\in\mathbb{N}$: $Y_n:=\bigcup^{n}_{j=1}B^{n}_{j}$, $H_n:=Y_n\setminus \overline{Y_{n-3}}$ and
$K_n:=\overline{Y_n} \setminus Y_{n-1}$ (let $Y_{-2}=Y_{-1}=Y_0=\emptyset$), we have the following properties:
\begin{itemize}
\item $E=\bigcup^{\infty}_{n=1}Y_n$;
\item $\overline{Y_n}\subseteq Y_{n+1}$ for all
$n\in\mathbb{N}$;
\item  $K_n\subseteq H_{n+1}$ for all $n\in\mathbb{N}$;
\item $E=\bigcup^{\infty}_{n=1}K_n$;
\item $H_m\cap H_n=\emptyset$ for all $m, n$ with $|m-n|\geq 3$.
\end{itemize}

Hence, the collection
$$
\bigcup_{n=1}^\infty\{ K_n\cap \overline{B^{n}_{j}}:\,j=1,\dots,n\}
$$
is a closed cover of $E$ and therefore
$$
\bigcup_{n=1}^\infty\{ H_{n+1}\cap B^{n+1}_{j}:\,j=1,\dots,n\}
$$
is an open cover of $E$. Both covers are countable and star-finite, and they are refinements of $\left\lbrace
B_j\right\rbrace _{j\geq 1}$. We call the first one $\left\lbrace T_i\right\rbrace _{i \geq 1}$ and the second one
$\left\lbrace V_i\right\rbrace _{i\geq 1}$, that is, for every $i$ there corresponds a unique pair $(j,n)$, $n\ge
j$, with $T_i=K_n\cap\overline{B^{n}_{j}}$ and $V_i=H_{n+1}\cap B^{n+1}_{j}$. Consequently, we have $T_i\subseteq
V_i$ for every $i\in\N$.

Now for each $i\in\N$ we take $j=j(i)\in\N$ such that $T_{i}\subseteq V_i\subseteq B_j$. Let us assume without loss
of generality that $j(i)\leq i$. We can write
$$T_i=\bigcup_{r\in\Omega}T_i\cap B_{j,r}\;\; \text{and}\;\; V_i=\bigcup_{r\in\Omega}V_i\cap B_{j,r}$$
and we define $T_{i,r}=T_i\cap B_{j,r}$ and $V_{i,r}=V_i\cap B_{j,r}$ for every $i\in\N$ and $r\in\Omega$. Clearly
we have that $T_{i,r}\subseteq V_{i,r}$ for all $i\in\N$ and $r\in\Omega$. Also $\textrm{dist}
(V_{i,r},V_{i,r'})\geq {1\over 2^{j+1}}\geq {1\over 2^{i+1}}$ for all $i\in\N$ and $r,r'\in\Omega$, $r\neq r'$.

\medskip
Finally let us define $X_{i,r}=X\cap T_{i,r}$. Bearing in mind that $T_{i,r}\subset V_{i,r}\subset B_{j,r}\subset
U_x$ for some $x\in X$ and that $T_{i,r}$ is closed, we obtain $X_{i,r}=X\cap T_{i,r}\subset X\cap\overline{U_x}$.
Since $X\cap\overline{U_x}$ has the $\varepsilon$-strong $C^p$ extraction property with respect to every open set
$U$ with $X\cap\overline{U_x}\subset U$, applying Lemma \ref{Properties of C^p extractibility}(1), we get that
$X_{i,r}$ has the $\varepsilon$-strong $C^p$ extraction property with respect to every such an open set $U$.
Finally, applying Lemma \ref{Properties of C^p extractibility}(2), $X_{i,r}$ has the strong $C^p$ extraction
property with respect to every open set $U'$ with $X_{i,r}\subset U'$. In particular, $X_{i,r}$ has the
$\varepsilon$-strong $C^p$ extraction property with respect to $V_{i,r}$.
\medskip

Let us consider now for each $i\in\N$ and $r\in\Omega$ an open set $W_{i,r}$ with $T_{i,r}\subseteq
W_{i,r}\subseteq \overline{{W_{i,r}}}\subseteq V_{i,r}$ and call $W_i=\bigcup_{r\in\Omega}W_{i,r}$.  We still have
that $\left\lbrace W_{i,r}\right\rbrace_{i\geq 1, r\in\Omega}$ is a refinement of $\mathcal{G}$, and that
$\left\lbrace W_{i}\right\rbrace_{i\geq 1}$ is a star-finite open cover of $E$.

 Then the collections
$\left\lbrace X_{i,r}\right\rbrace_{i\geq 1, r\in\Omega}$ , $\left\lbrace W_{i,r}\right\rbrace_{i\geq 1,
r\in\Omega}$ , $\left\lbrace V_{i,r}\right\rbrace_{i\geq 1, r\in\Omega}$ have the required properties. Note that
each $X_{i,r}$ has the strong $C^p$ extraction property with respect to every open set containing it.
\end{proof}

\begin{lem}\label{tool 1}
Let  $X_i$, $V_i$, and $V_{i,r}$ be as in  Lemma \ref{sigma-discrete refinements}. Then, for every
$i,j\in\mathbb{N}$, $X_i\cap V_j$, has the strong $C^p$ extraction property with respect to $V_j$. Moreover, if
$\varphi_{i,j}$ is a $C^p$ diffeomorphism satisfying the definition of the $\varepsilon$-strong extraction property
for these sets, then $\varphi(V_{i,r})\subset V_{i,r}$ and $\varphi(V_{j,r})\subset V_{j,r}$ for every
$r\in\Omega$.
\end{lem}

\begin{proof}
Take $V\subseteq V_j$ an open set. Let $r\in\Omega$ and consider the open set $V_{j,r}$. For each $s\in \Omega$ the
set $X_{i,s}\cap V_{j,r}$ has the $\varepsilon$-strong $C^p$ extraction property with respect to the open set
$V_{i,s}\cap V_{j,r}$. We have
$$X_{i,s}\cap V_{j,r}\subseteq W_{i,s}\cap V_{j,r}\subseteq \overline{W_{i,s}}\cap V_{j,r}\subseteq V_{i,s}\cap V_{j,r}.$$
There exists a $C^{p}$ diffeomorphism $h_{r,s}:(V_{i,s}\cap V_{j,r})\setminus X_{i,s}\rightarrow V_{i,s}\cap
V_{j,r}\setminus(X_{i,s}\setminus(W_{i,s}\cap V))$ which is the identity on $((V_{i,s}\cap V_{j,r}\setminus
(W_{i,s}\cap V))\setminus X_{i,s}.$ Outside $V_{i,s}\cap V_{j,r}$ we define $h_{r,s}$ to be the identity. Since
$\overline{W_{i,s}}\subseteq V_{i,s}$ we have a well-defined $C^p$ diffeomorphism
$$h_{r,s}:V_{j,r}\setminus X_{i,s}\rightarrow V_{j,r}\setminus (X_{i,s}\setminus (W_{i,s}\cap V))$$
which is the identity on $(V_{j,r}\setminus(W_{i,s}\cap V))\setminus
 X_{i,s}$. In particular $h_{r,s}$ is the identity on $V_{i,s'}\cap
V_{j,r}$ for every $s'\in\Omega, s\neq s'$.

Having defined $h_{r,s}$ for each $s\in\Omega$ in the way described above, we finally  define
$$
h_r=\bigcirc_{s\in\Omega}h_{r,s}
$$
as an infinite composition of $h_{r,s}$, $s\in\Omega$. It is easy to see that $h_r$ is well defined and provides a $C^p$
diffeomorphism of $V_{j,r}\setminus X_i$ onto $V_{j,r}\setminus (X_i\setminus V)$
 which is the identity on $(V_{j,r}\setminus V)\setminus X_i$.

By the discreteness of the family $\{V_{j,r}\}_{r\in\Omega}$, the formula $h(x)=h_r(x)$, $x\in V_{j,r}\setminus
X_i$, defines a $C^p$ diffeomorphism of $\bigcup_{r\in\Omega}V_{j,r}\setminus X_i=V_j\setminus X_i$ onto
$\bigcup_{r\in\Omega} \left(V_{j,r}\setminus (X_i\setminus V)\right)=V_j\setminus(X_i\setminus V)$ which is the
identity on $\bigcup_{r\in\Omega}\left((V_{j,r}\setminus V)\setminus X_i\right)=(V_j\setminus V)\setminus X_i$.

To end the proof observe that each $h_{r,s}$ is the identity outside $V_{i,s}\cap V_{j,r}$, so $h$ sends $V_{i,r}$
into $V_{i,r}$ and $V_{j,r}$ into $V_{j,r}$ for every $r\in\Omega$.
\end{proof}

Notice that in the previous two Lemmas \ref{tool 1} and \ref{sigma-discrete refinements}  one can replace the
$\varepsilon$-strong $C^p$ extraction property with just the strong $C^p$ extraction property.

The last tool that we need to introduce before going into the proof of Theorem \ref{abstract genereal
extractibility} is the next lemma (see Statement A in West's paper \cite[pp. 289]{West}).

\begin{lem}\label{Statement A'_2}
Let $V_0,\dots, V_n$ be open sets of $E$, and $X_0,\dots,X_n$ be subsets of $V_0,\dots,V_n$. Take also an open set
$U$. Suppose each $X_i$ is relatively closed in $V=\bigcup^{n}_{i=0}V_i$ and has the $\varepsilon$-strong $C^p$
extraction property with respect to $V$. Then for every $\varepsilon>0$ there exists a $C^{p}$ diffeomorphism of
$V\setminus X_0$ onto $V\setminus(X_0\setminus U)$ which is the identity outside $V_0\cap U$, carries $X_i\setminus
X_0$ into $V_i$ for each $i=1,\dots, n$, and moves no point more than $\varepsilon$.
\end{lem}

\begin{proof}
We will divide the set $X_0$ that we want to extract as follows. For each $j=0,\dots,n$, let
$$Q_j :=\left\lbrace Z:\,Z=\bigcap^{n-j}_{i=0}X_{p(i)}\setminus \bigcup^{n}_{i=n-j+1}
X_{p(i)}\;\mbox{for some permutation}\;p \;\mbox{of}\; \left\lbrace 0,\dots,n \right\rbrace \;\mbox{carrying} \; 0
\;\mbox{to}\;0\right\rbrace;
$$
these are families of subsets of $X_0$. Let $Q=\bigcup^{n}_{j=0}Q_j$. The family $Q$ is a pairwise disjoint cover of
$X_0$ with cardinality $\le 2^n$. Order $Q$ in such a manner that if $j<k$ then all elements of $Q_j$ precede those
of $Q_k$. We then will list the elements of $Q$ as $Z_1,Z_2,\dots,Z_{2^n}$ (bearing in mind that some $Z_m$'s may
repeat in that listing); that is, $Q=\left\lbrace Z_m\right\rbrace ^{2^n}_{m=1}$ and  $\bigcup^{2^n}_{i=1}Z_i=X_0$.
Note that $Q_0=\{Z_1\}$ and $Q_n=\{Z_{2^n}\}$, where
$$
Z_1=\bigcap^{n}_{i=0} X_i\ \ \text{ and}\ \ Z_{2^n}=X_0\setminus \bigcup^{n}_{i=1}X_i.
$$
Likewise, for each $m=1,\dots,2^n$, if $Z_m\in Q_j$ and $p$ is a permutation for which
$Z_m=\bigcap^{n-j}_{i=0}X_{p(i)}\setminus \bigcup^{n}_{i=n-j+1} X_{p(i)}$, we define
$$
N_m=\bigcap^{n-j}_{i=0} V_{p(i)}\subset V_0.
$$
The family $\{N_m\}_{m=1}^{2^n}$ is an open cover of $V_0$; we also have $N_1=\bigcap^{n}_{i=0}V_i$ and $N_{2^n}=V_0$. Denote by
$Q^{*}_{k}$ the union of all the elements of $Q_k$, $k=0,1,\dots,n$; notice that $Q^{*}_{0}=Z_1$ and $Q_n^*=Z_{2^n}$.
For each $j>0$, the elements of $Q_j$ form a pairwise disjoint family of relatively closed subsets of the open set
$V\setminus\bigcup^{j-1}_{k=0}Q^{*}_{k}$.  Also each $Z_m$ in $Q_j$ lies in $N_m$. Therefore, for each $j>0$, there
exists a collection of pairwise disjoint open sets $M_m$ in $V\setminus\bigcup^{j-1}_{k=0}Q^{*}_{k}$, one for each
$Z_m$ in $Q_j$ (that is, if $Z_m=Z_{m'}\in Q_j$, $m\not=m'$, then $M_m=M _{m'}$), such that
$$Z_m\subseteq M_m\subseteq N_m\setminus \bigcup^{n}_{i=n-j+1}X_{p(i)},$$
where $j$ is such that $Z_m$ is in $Q_j$ and $p$ is a permutation defining $Z_m$ as above.

The set $Z_1=\bigcap^{n}_{i=0}X_i\subseteq N_1=\bigcap^{n}_{i=0}V_i$ is relatively closed in $V$ and has the
$\varepsilon$-strong  $C^p$ extraction property with respect to $V$,  so there is a $C^{p}$ diffeomorphism
$$h_1: V\setminus Z_1\rightarrow V\setminus (Z_1\setminus U)$$
which is the identity outside $(N_1\cap U)\cup Z_1$ and moves no point more than ${\varepsilon\over 2^n}$; in
particular, $h_1(Z_2)\setminus U=Z_2\setminus U$ should $Z_2\not=Z_1$.

We will apply induction to prove that for $1<m\leq 2^n$ there exists a $C^{p}$ diffeomorphism
$$h_m:(V\setminus(\bigcup^{m-1}_{i=1}Z_i\setminus U))\setminus
h_{m-1}\circ\cdots\circ h_1(Z_m)\rightarrow V\setminus (\bigcup^{m}_{i=1}Z_i\setminus U)$$ which is the identity
outside $\left(h_{m-1}\circ\cdots\circ h_1(M_m)\cap N_m\cap U\right)\cup h_{m-1}\circ\cdots\circ h_1(Z_m)$,
satisfies

$$
h_m\circ h_{m-1}\circ\cdots\circ h_1(V\setminus\bigcup^{m}_{i=1}Z_i)=V\setminus(\bigcup^{m}_{i=1}Z_i\setminus U)\ \
\text{and}\ \ h_m\circ h_{m-1}\circ\cdots\circ h_1(Z_{m+1})\setminus U=Z_{m+1}\setminus U,
$$
and such that $h_m\circ\cdots\circ h_1$ moves no point more than ${m\varepsilon\over 2^n}$.

Suppose this is true for every  $1\le k\le m-1$, and let us check so it is for $m$. Assume $Z_m\in Q_j$;
additionally, we can assume that $Z_{m-1}\not=Z_m$, otherwise, the identity in place of $h_m$ will do. By the
definition of the $\varepsilon$-strong $C^p$ extractibility and Lemma \ref{Properties of C^p extractibility}(2),
$Z_m$ has the $\varepsilon$-strong $C^p$ extraction property with respect to
$V\setminus\bigcup^{j-1}_{k=0}Q^{*}_{k}$. Furthermore, one more application of Lemma \ref{Properties of C^p
extractibility}(2) yields that $Z_m\subset V\setminus \bigcup^{m-1}_{i=1}Z_i$ is relatively closed and has the
$\varepsilon$-strong $C^p$ extraction property with respect to $V\setminus \bigcup^{m-1}_{i=1}Z_i$ . By Lemma
\ref{Properties of C^p extractibility} (3), $h_{m-1}\circ\cdots\circ h_1(Z_m)\subset h_{m-1}\circ\cdots\circ
h_1(V\setminus \bigcup^{m-1}_{i=1}Z_i)$ is relatively closed and has the strong $C^p$ extraction property with
respect to the open set
$$h_{m-1}\circ\cdots\circ h_1(V\setminus \bigcup^{m-1}_{i=1}Z_i)=V\setminus(\bigcup^{m-1}_{i=1}Z_i\setminus U).$$

Considering the open set $h_{m-1}\circ\cdots\circ h_1(M_m)\cap N_m\cap U$, there exists a $C^p$ diffeomorphism
$h_m$ from
$$(V\setminus(\bigcup^{m-1}_{i=1}Z_i\setminus U))\setminus
h_{m-1}\circ\cdots\circ h_1(Z_m)$$ onto
\begin{equation}\label{Complicated expresion pof Lemma 2' of Renz}
(V\setminus (\bigcup^{m-1}_{i=1}Z_i\setminus U))\setminus (h_{m-1}\circ\cdots\circ
h_1(Z_m)\setminus(h_{m-1}\circ\cdots\circ h_1(M_m)\cap N_m\cap U)
\end{equation}
which is the identity outside  $\left(h_{m-1}\circ\cdots\circ h_1(M_m)\cap N_m\cap U\right)\cup
h_{m-1}\circ\cdots\circ h_1(Z_m)$. Furthermore by Remark \ref{composition of diffeos do not move the points too
much} and the fact that $h_{m-1}\circ\cdots\circ h_1$ moves no point more than ${(m-1)\varepsilon\over 2^n}$
(induction hypothesis), we have that $h_m\circ (h_{m-1}\circ\cdots\circ h_1)$ moves no point more than
${\varepsilon\over 2^n}+{(m-1)\varepsilon\over 2^n}={m\varepsilon\over 2^n}$.

Finally note that the expression \eqref{Complicated expresion pof Lemma 2' of Renz} is equal to $V\setminus
(\bigcup^{m}_{i=1}Z_i\setminus U)$ because $h_{m-1}\circ\cdots\circ h_1(Z_m)\subseteq h_{m-1}\circ\cdots\circ
h_1(M_m) \cap N_m$ and the fact that $h_{m-1}\circ\cdots\circ
h_1(Z_m)\setminus U=Z_m\setminus U$.\\

To conclude define
$$h=h_{2^n}\circ\cdots\circ  h_1:V\setminus X_0\rightarrow V\setminus (X_0\setminus U).$$
This is a $C^{p}$ diffeomorphism of $V\setminus X_0$ onto $V\setminus (X_0\setminus U)$ which is the identity
outside $(V_0\cap U)\cup X_0$ and moves no point more than $\varepsilon$.

To end the proof we will show that $h$ carries $X_i\setminus X_0$ into $V_i$ for each $i=1,\dots, n$. Take $x\in
X_i\setminus X_0$ for some $i=1,\dots,n$ and let us see that $h(x)\in V_i$. If $h_1(x)\neq x$ then $h_1(x)\in N_1$,
so because $N_1\subseteq V_i$ we have that $h_1(x)\in V_i$. Suppose  now that $h_{m-1}\circ\cdots\circ h_1(x)\in
V_i$. If $h_m\circ\cdots\circ h_1(x)\neq h_{m-1}\circ\cdots\circ h_1(x)$ then we have that $x\in M_m$ and
$h_m\circ\cdots\circ h_1(x)\in N_m$. We have that $x\in X_i\cap M_m$ and $M_m\subseteq N_m\setminus
\bigcup^{n}_{i=n-j+1} X_{p(i)}$ where $j$ is such that $Z_m$ is in $Q_j$ and $p$ is a permutation that defines
$Z_m$. Obviously we must have that $i=p(i_0)$ where $i_0=\{1,\dots,n-j\}$. Also $N_m=\bigcap^{n-j}_{k=0}
V_{p(k)}\subseteq V_{i}$, hence $h_{m}\circ\cdots\circ h_1(x)$  must also lie in $V_i$. Applying induction we have
that $h(x)\in V_i$.

\end{proof}

\begin{rem}\label{weak version of statement A_2'}
{\em Notice that if we only have the strong $C^p$ extraction property of the sets $X_i$ with respect to $V$, we get the same result except that for any $\varepsilon>0$ we can not assure that the final extracting diffeomorphism moves all points less than $\varepsilon$. In such a case we will say that we have a {\em weak version} of Lemma \ref{Statement A'_2}.\\
However, suppose we have that $X_0\subseteq V_0, X_1\subseteq V_1,\dots,X_n\subseteq V_n$ are of the form
$g(X_i)\subseteq g(V_i)$, $i=0,1,\dots,n$, where $g$ is a $C^p$ diffeomorphism that moves points less than some
$\delta_1>0$. In such a case, using Remark \ref{composition of diffeos do not move the points too much}, for any
$\delta_2>0$ we can make the final extracting diffeomorphism $h$ of the proof of Lemma \ref{Statement A'_2} satisfy
that $h\circ g$ moves no point more than $\delta_1+\delta_2$. In particular if $g(x)=x$ then we have that
$||h(g(x))-x||\leq \delta_2$.}
\end{rem}

\begin{rem}\label{discrete}
{\em Assume, additionally, that the set $V$ of Lemma \ref{Statement A'_2} is of the form
$V=\bigcup_{r\in\Omega}V_r$, where $\Omega$ is a set of indexes, each $V_r$ is an open set , and $V_r\cap
V_{r'}=\emptyset$ for every $r,r'\in\Omega$, $r\neq r'$. Then we can also require that the extracting $C^{p}$
diffeomorphism of $V\setminus X_0$ onto $V\setminus(X_0\setminus U)$ sends each set $V_r\setminus X_0$ into $V_r$
for every $r\in\Omega$. To prove this, fix $r\in\Omega$ and replace the sequences $V_0,V_1,\dots,V_n$ and
$X_0,X_1,\dots X_n$ with $V_0\cap V_r,V_1\cap V_r\dots,V_n\cap V_r$ and $X_0\cap V_r,X_1\cap V_r, \dots, V_n\cap
V_r$, respectively. Further, observe that $\bigcup^{n}_{i=0}V_i\cap V_r=V_r$ and that each $X_i\cap V_r$ has the
strong $C^p$ extraction property with respect to $V_r$ by Lemma \ref{Properties of C^p extractibility} (2).
According to the assertion of Lemma \ref{Statement A'_2}, we conclude that there exists a $C^p$ diffeomorphism
$h_r:V_r\setminus X_0\to V_r\setminus(X_0\setminus U)$ satisfying the required conditions. Finally, it is enough to
set $h:V\setminus X_0 \to V\setminus(X_0\setminus U)$ by letting $h(x)=h_r(x)$ for $x\in V_r\setminus X_0$. }
\end{rem}

The rest of the proof of Theorem \ref{abstract genereal extractibility} goes as in \cite[Theorem 1]{West}, with
some modifications due to the facts that we work here with an open set $U$ not necessarily containing $X$, and that
$E$ is not necessarily separable.

\begin{proof}[Proof of Theorem \ref{abstract genereal extractibility}]
Apply Lemma \ref{sigma-discrete refinements} to the given cover $\mathcal{G}$ to find collections $\left\lbrace
X_i\right\rbrace _{i\geq 1}$, $\left\lbrace W_i\right\rbrace_{i\geq 1}$ and $\left\lbrace V_i\right\rbrace _{i\geq
1}$ of subsets of $E$ satisfying conditions $(1)$--$(6)$ of that lemma. By Lemma \ref{tool 1}, for all
$i,j\in\mathbb{N}$, $X_i\cap V_j$ has the  $\varepsilon$-strong $C^p$ extraction property with respect to $V_j$.
Moreover, if $\varphi_{i,j}$ is a $C^p$ diffeomorphism with this property, then $\varphi(V_{i,r})\subset V_{i,r}$
and $\varphi(V_{j,r})\subset V_{j,r}$ for every $r\in\Omega$.

Let us now define the required $C^p$ diffeomorphism $g: E\setminus X\to E\setminus (X\setminus U)$ which is the
identity on $(E\setminus U)\setminus X$ and is limited by $\mathcal{G}$.

\begin{enumerate}
\item For a given $V_1$, define $I_1=\{1_1=1,1_2,\dots, 1_{n(1)}\}\subset \N$ to be the finite set of positive integers such that  $W_{1_1},W_{1_2},\dots,W_{1_{n(1)}}$ are
the only $W_i's$ sets for which $V_1\cap W_i\not=\emptyset$ (if there were infinitely many such $W_i's$, then we would have $V_1\cap V_i's\not=\emptyset$ for infinitely many $i$'s, which would contradict the star-finiteness of $\left\lbrace
V_i\right\rbrace _{i\geq 1}$; obviously, we assume that $V_i\not=V_{i'}$ for $i\not=i'$). Since $\left\lbrace
W_i\right\rbrace _{i\geq 1}$ is a cover we have $\bigcup_{i\in I_1}W_{i}\cap V_1=V_1$. (A priori $V_1$ can be
covered by a proper subfamily of $\{V_2,\dots,V_n\})$. Assuming that $i_1$ is the greatest number in $I_1$ (in
particular $i_1\geq 1$) we set
$$\varepsilon_1={1\over 2}\cdot{1\over 2^{i_1+1}}>0.$$
We want to apply Lemma \ref{Statement A'_2} for the sets
$$ X_{1_1}=X_1=X_1\cap V_1,X_{1_2}\cap V_1,\dots, X_{1_{n(1)}}\cap V_1$$
which play the role of $X_0,\dots,X_n$ in the statement of the Lemma \ref{Statement A'_2}, and for the sets
$$W_{1_1}=W_1=W_1\cap V_1,W_{1_2}\cap V_1,\dots,W_{1_{n(1)}}\cap V_1,$$
 which  play the role of $V_0,\dots,V_n$ respectively, and for the positive number $\varepsilon_1>0$. Observe that
each $X_i\cap V_1$, $i\in I_1$, has the  strong $C^p$ extraction property with respect to $V_1$. Hence, applying
Lemma \ref{Statement A'_2}, we find a $C^{p}$ diffeomorphism $g_1$ of $V_1\setminus X_1$ onto $V_1\setminus
(X_1\setminus U)$ which is the identity outside $(W_1\cap U)\cup X_1$, carries
$(X_l\setminus X_1)\cap V_1$ into $W_l\cap V_1$ for each $l> 1$ and moves no point more than $\varepsilon_1$.  \\
By Remark \ref{discrete} we may also assume that $g_1$ sends each set $V_{1,r}\setminus X_1$ into $V_{1,r}$. This
means in particular that $g_1$ refines
$\mathcal{G}$. Also, since $g_1$ moves no point more than $\varepsilon_1$ we cannot have that $x\in V_{i,r}$ and $g_1(x)\in V_{i,r'}$ for some $i\in I_1$ and different $r,r'\in\Omega$ (recall that $\textrm{dist} (V_{i,r},V_{i,r'})\geq {1\over 2^{i+1}}>\varepsilon_1$).\\
Since $W_1\cap U\subseteq \overline{W_1}\cap U\subseteq V_1\cap U$, by making $g_1$ be the identity outside
$V_1\setminus X_1$, there exists a well-defined natural extension of $g_1$ from  $V_1\setminus X_1$ to $E\setminus
X_1$. Now we have a $C^p$ diffeomorphism $g_1$ such that
\begin{enumerate}
\item $g_1$ acts from $E\setminus X_1$ onto $E\setminus (X_1\setminus U)$.
\item $g_1$ is the identity on $E\setminus\left[ (W_1\cap U)\cup X_1\right] .$ In particular $g_1$ is the identity on $(E\setminus U)\setminus X_1$.
\item $g_1$ carries
$(X_l\setminus X_1)\cap V_1$ into $W_l\cap V_1$ for each $l> 1$.
\item We require that if $X_1=\emptyset$ or $X_{1,r}=\emptyset$ then $g_1$ is the identity
on $V_1$ or $V_{1,r}$, respectively.
\item  $g_1$ moves no point more than $\varepsilon_1={1\over 2}\cdot{1\over 2^{i_1+1}}$.
\item $g_1$
sends each set $V_{1,r}\setminus X_1$ into $V_{1,r}$ for every $r\in\Omega$, so $g_1$ refines $\mathcal{G}$.
\end{enumerate}

\item Consider now the set $V_2$ and define $I_2=\{2_1=2,2_2,\dots,2_{n(2)}\}\subset \N$ to be the finite set of natural numbers such that $W_{2_1},W_{2_2},\dots,W_{2_{n(2)}}$ are
the only $W_i's$ sets for which $V_2\cap W_i\not=\emptyset$ (if there were infinitely many such $W_i's$ then $V_2\cap V_i's\not=\emptyset$ for infinitely many $i$'s, which would contradict the star-finiteness of
$\left\lbrace V_i\right\rbrace _{i\geq 1}$). Assume that $i_2$ is the greatest number in $ I_2$ (in particular
$i_2\geq 2$), and set
$$\varepsilon_2={1\over 2^2}\cdot{1\over 2^{i_2+1}}>0.$$ Since $\left\lbrace W_{i}\right\rbrace _{i\geq 1}$ is a cover we
have
$$\bigcup_{i\in I_2}g_1(V_2\setminus X_1)\cap W_i\cap V_2 =g_1(V_2\setminus X_1)\cap V_2.$$
Again, we want to apply Lemma 2.26 for the sets

$$\{g_1((X_i\setminus X_1)\cap V_2)\cap V_2:\,i\in I_2\}$$
playing the role of $X_0,\dots, X_n$ in the statement of the lemma, for
$$\{g_1(V_2\setminus X_1)\cap W_i\cap V_2:\,i\in I_2\}$$
playing the role of the sets $V_0\dots,V_n$ respectively. Here we should recall that $g_1(X_i\setminus X_1)\subseteq W_i$.\\
Observe that by Lemma \ref{Properties of C^p extractibility} (3), each $g_1((X_i\setminus X_1)\cap V_2)\cap V_2$
has the strong $C^p$ extraction property with respect to the open set $g_1(V_2\setminus X_1)\cap V_2$. Applying the
weak version of Lemma \ref{Statement A'_2} to these sets we get a $C^{p}$ diffeomorphism  $g_2$ of
$$ \left[ g_1(V_2\setminus X_1)\cap V_2\right] \setminus \left[ g_1(X_2\setminus X_1)\cap V_2\right]  =g_1(V_2\setminus (X_1\cup X_2))\cap V_2$$
onto
$$\left[ g_1(V_2\setminus X_1)\cap V_2\right] \setminus \left[ (g_1(X_2\setminus X_1)\cap V_2)\setminus U\right] $$
which is the identity outside
$$ \left(g_1(V_2\setminus (X_1\cup X_2))\cap W_2\cap V_2\cap U\right)\cup  \left( g_1(X_2\setminus X_1)\cap V_2\right)$$
and carries
$$g_1(X_k\setminus (X_1\cup X_2)\cap V_2)\cap V_2$$
into
$$g_1(V_2\setminus X_1)\cap W_k\cap V_2$$
for each $k>2$. Moreover using Remark \ref{weak version of statement A_2'} one can also assume that $g_2\circ g_1$ moves no point more than $\varepsilon_1+\varepsilon_2$.\\
Because $\overline{g_1(V_2\setminus (X_1\cup X_2))\cap W_2}\cap U\subseteq g_1(V_2\setminus (X_1\cup X_2))\cap
V_2\cap U$,  by letting $g_2$ be the identity outside $g_1(V_2\setminus (X_1\cup X_2))\cap V_2$ there exists a
well-defined natural extension of $g_2$ to $g_1(E\setminus (X_1\cup X_2))$. To sum up we have the following
properties:
\begin{enumerate}
\item $g_2$ acts from
$$g_1(E\setminus (X_1\cup X_2))$$
onto
\begin{align*}
g_1(E\setminus X_1)\setminus\left[ g_1(X_2\setminus X_1)\setminus U\right] &=E\setminus (X_1\setminus U)\setminus\left[ (g_1((X_2\setminus X_1)\cap U)\cup g_1((X_2\setminus X_1))\setminus U)\setminus U\right] =\\
&=E\setminus (X_1\setminus U)\setminus \left[ (X_2\setminus X_1)\setminus U\right] =E\setminus ((X_1\cup
X_2)\setminus U).
\end{align*}
(Here we are using the fact that $g_1((X_2\setminus X_1)\cap U)\subseteq U$ and that $g_1$ is the identity outside
$U$).
\item $g_2$ is the identity on $E\setminus\left[\left(g_1(W_2\setminus (X_1\cup X_2))\cap W_2\cap U\right)\cup  \left( g_1(X_2\setminus X_1)\cap V_2\right) \right] $. Since
$$\left(g_1(W_2\setminus (X_1\cup X_2))\cap W_2\cap U\right)\cup  \left( g_1(X_2\setminus X_1)\cap V_2\right) \subseteq U\cup X_1\cup X_2,$$ in particular $g_2$ is the identity on $E\setminus (U\cup (X_1\cup X_2))$. Because $g_1$ is the identity outside $U$ then $g_2$ is the identity on $g_1(E\setminus (U\cup(X_1\cup X_2))=g_1((E\setminus U)\setminus (X_1\cup X_2))$.
\item $g_2$ carries
$$g_1(X_l\setminus (X_1\cup X_2))\cap V_2$$
into
$$ W_l\cap V_2$$
for each $l>2$.
\item If $X_{2,r}=\emptyset$ then we require that $g_2$ is the
identity on $g_1(V_{2,r}\setminus X_1)\cap V_{2}$ and in $g_1(V_2\setminus X_1)\cap V_{2,r}$.
\item We have
\begin{equation}\label{g_2 does not move too much the points}
\left\{
           \begin{array}{ll}
         ||g_2(g_1(x))-x||\leq\varepsilon_1+\varepsilon_2 ={1\over 2}\cdot{1\over 2^{i_1+1}}+{1\over 4}\cdot{1\over 2^{i_2+1}} & \text{for every } x \in E\setminus (X_1\cup X_2) \\
        & \\
         ||g_2(g_1(x))-x||<{1\over 2^{2+1}}    & \text{for every } x \in V_2\setminus (X_1\cup X_2).
           \end{array}
         \right.
\end{equation}
The first inequality is clear. For the second one, when $x\notin V_2\setminus V_1$ we have $g_1(x)=x$, hence
$$||g_2(g_1(x))-x||=||g_2(x)-x||\leq \varepsilon_2={1\over 4}\cdot{1\over 2^{i_2+1}}<{1\over 2^{2+1}},$$
and if $x\in V_1\cap V_2$ then $V_1\cap V_2\neq\emptyset$, so $i_1\geq 2$ and
\begin{align*}
||g_2(g_1(x))-x||\leq {1\over 2}\cdot{1\over 2^{i_1+1}}+{1\over 4}\cdot{1\over 2^{i_2+1}}<\left( {1\over 2}+{1\over
4}\right) \cdot {1\over 2^{2+1}}.
\end{align*}

\item The composition $g_2\circ g_1$ is a mapping from $E\setminus (X_1\cup X_2)$ onto $E\setminus ((X_1\cup X_2)\setminus U)$ and it is the identity outside $(E\setminus U)\setminus (X_1\cup X_2)$. Let us check that it refines $\mathcal{G}$. Take $x\in E\setminus (X_1\cup X_2)$. If $g_2(g_1(x))=g_1(x)$, since $g_1$ refines $\mathcal{G}$ we are done. Otherwise we have $g_2(g_1(x))\neq g_1(x)$, and then $g_1(x),g_2(g_1(x))\in g_1(V_2\setminus (X_1\cup X_2))\cap V_2$. We have that $x,g_2(g_1(x))\in V_2$. We must have $x,g_2(g_1(x))\in V_{2,r}$ for some $r\in \Omega$ (as otherwise $x\in V_{2,r}$ and $g_2(g_1(x))\in V_{2,r'}$ a contradiction with \eqref{g_2 does not move too much the points} since $\textrm{dist} (V_{2,r},V_{2,r'})\geq {1\over 2^{2+1}}$).

\end{enumerate}
\item We go on doing this process by successive applications of Lemma \ref{Statement A'_2}. We want to apply induction to prove that for  each $j\geq 3$ we can find a $C^{p}$ diffeomorphism $g_j$ such that
\begin{enumerate}
\item $g_j$ acts from
$$g_{j-1}\circ\cdots\circ g_1\left( E\setminus
\bigcup_{k\leq j}X_k\right)$$ onto
\begin{align*}
E\setminus \left( \bigcup_{k\leq j}X_k\setminus U\right).
\end{align*}
\item $g_j$ is the identity on $g_{j-1}\circ\cdots\circ g_1\left( (E\setminus U)\setminus (\bigcup_{k\leq j}X_k)\right)$.
\item $g_j$ carries $$g_{j-1}\circ\cdots\circ g_1(X_l\setminus \bigcup_{k\leq j}X_k)\cap V_j$$
into
$$W_l\cap V_j$$
for eack $l>j$.
\item If
$X_{j,r}=\emptyset$ then $g_j$ is the identity on $g_{j-1}\circ\cdots\circ g_1(V_{j,r}\setminus \bigcup_{k<
j}X_k)\cap V_j$ and on $g_{j-1}\circ\cdots\circ g_1(V_j\setminus \bigcup_{k< j}X_k)\cap V_{j,r}$.
\item We have
\begin{equation}\label{g_j does not move the points too much}
\left\{
           \begin{array}{ll}
         ||g_j\circ\cdots\circ g_1(x)-x||\leq \sum^{j}_{k=1} {1\over 2^k}\cdot {1\over 2^{{i_k}+1}    } & \text{for every } x \in E\setminus (\bigcup_{k\leq j}X_k) \\
        & \\
         ||g_j\circ\cdots\circ g_1(x)-x||<{1\over 2^{j+1}}    & \text{for every } x \in V_j\setminus (\bigcup_{k\leq j}X_k),
           \end{array}
         \right.
\end{equation}
where $i_k$ is the greatest number such that $V_{i_k}\cap V_k\neq\emptyset.$ Let us call for every $k=1,\dots,j$,
 $$\varepsilon_k={1\over 2^k}\cdot{1\over 2^{i_k+1}}>0.$$
\item $g_j\circ\cdots\circ g_1$ refines $\mathcal{G}$.
\end{enumerate}
Suppose this is true for $j-1$ and let us check this is so for $j$.\\

The idea is the same as in steps (1) and (2). We first find the set $I_j=\{j_1=j,j_2,\dots,j_{n(j)}\}\subset \N$
such that $W_{j_1},W_{j_2},\dots,W_{j_{n(j)}}$ are the only $W_i's$ sets for which $V_j\cap W_i\not=\emptyset$.
Assume $i_j$ is the greatest number in $ I_j$ (in particular $i_j\geq j$) and set
$$\varepsilon_j={1\over 2^j}\cdot{1\over 2^{i_j+1}}>0.$$ We have that
$$\bigcup_{i\in I_j}g_{j-1}\circ\cdots\circ g_1(V_j\setminus \bigcup_{k<j}X_k)\cap W_i\cap V_j =g_{j-1}\circ\cdots\circ g_1(V_j\setminus \bigcup_{k<j}X_k)\cap V_j.$$ We want to apply Lemma 2.26 for the sets

$$\{g_{j-1}\circ\cdots\circ g_1((X_i\setminus \bigcup_{k<j} X_k)\cap V_j)\cap V_j:\,i\in I_j\}$$
playing the role of $X_0,\dots, X_n$ in the statement of the lemma, for
$$\{g_{j-1}\circ\cdots\circ g_1(V_j\setminus \bigcup_{k<j} X_k)\cap W_i\cap V_j:\,i\in I_j\}$$
playing the role of the sets $V_0\dots,V_n$ respectively. \\
Observe that by Lemma \ref{Properties of C^p extractibility} (3), each $g_{j-1}\circ\cdots\circ g_1((X_i\setminus
\bigcup_{k<j}X_k)\cap V_j)\cap V_j$ has the strong $C^p$ extraction property with respect to the open set
$g_{j-1}\circ\cdots\circ g_1(V_j\setminus\bigcup_{k<j} X_k)\cap V_j$. Applying the weak version of Lemma
\ref{Statement A'_2} to these sets we get a $C^p$ diffeomorphism $g_j$ from

$$\left[ g_{i-1}\circ\cdots\circ g_1(V_i\setminus \bigcup_{j<i}X_j)\cap V_i\right] \setminus
\left[g_{j-1}\circ\cdots\circ g_1(X_j\setminus \bigcup_{k< j}X_k)\cap V_j\right] =g_{j-1}\circ\cdots\circ
g_1(V_j\setminus \bigcup_{k\leq j}X_k)\cap V_j$$ onto
$$\left[ g_{j-1}\circ\cdots\circ g_1(V_j\setminus \bigcup_{k<j}X_k)\cap V_j\right] \setminus \left[g_{j-1}\circ\cdots\circ g_1(X_j\setminus \bigcup_{k< j}X_k)\cap V_j\setminus U\right] $$
which is the identity outside
$$ \left((g_{j-1}\circ\cdots\circ g_1(V_j\setminus \bigcup_{k\leq j}X_k)\cap W_j\cap U\right)\cup \left(g_{j-1}\circ\cdots
\circ g_1(X_j\setminus \bigcup_{k< j}X_k)\cap V_j\right)$$ and carries
$$g_{j-1}\circ\cdots\circ g_1(X_l\setminus \bigcup_{k\leq j}X_k)\cap V_j$$
into
$$W_l\cap V_j$$
for each $l>j$. This last property establishes (c). \\
Define $g_j$ to be the natural extension of $g_{j-1}\circ\cdots\circ g_1(V_j\setminus \bigcup_{k\leq j}X_k)\cap
V_j$ to $g_{j-1}\circ\cdots\circ g_1(E\setminus \bigcup_{k\leq j}X_k)$,  so now $g_j$ is defined from
$$g_{j-1}\circ\cdots\circ g_1\left( E\setminus
\bigcup_{k\leq j}X_k)\right)=$$ onto
\begin{align*}
&\left[ g_{j-1}\circ\cdots\circ g_1(E\setminus \bigcup_{k<j}X_k)\right] \setminus \left[g_{j-1}\circ\cdots\circ g_1(X_j\setminus \bigcup_{k< j}X_k)\setminus U\right] =\\
&E\setminus (\bigcup_{k<j}X_j\setminus U) \setminus \left[\left( g_{j-1}\circ\cdots\circ g_1((X_j\setminus \bigcup_{k< j}X_k)\cap U)\cup g_{j-1}\circ\cdots\circ g_1((X_j\setminus \bigcup_{k< j}X_k)\setminus U)\right) \setminus U\right] =\\
&=E\setminus (\bigcup_{k<j}X_k\setminus U) \setminus \left[(X_j\setminus \bigcup_{k<j}X_k)\setminus U
\right]=E\setminus \left( \bigcup_{k\leq j}X_k\setminus U\right).
\end{align*}
Here we are using that $g_{j-1}\circ\cdots\circ g_1((X_j\setminus \bigcup_{k<j}X_k)\cap U\subseteq U $ and also (a) from the induction hypothesis. This establishes (a).\\
We also have that $g_j$ is the identity on $(E\setminus U)\setminus (\bigcup_{k\leq j}X_k)=g_{j-1}\circ\cdots\circ g_1\left( (E\setminus U)\setminus (\bigcup_{k\leq j}X_k)\right)$, which establishes (b).\\
If $X_{j,r}=\emptyset$ then we let $g_j$ be the identity on  $g_{j-1}\circ\cdots\circ g_1(V_{j,r}\setminus
\bigcup_{k< j}X_k)\cap V_j$ and on $g_{j-1}\circ\cdots\circ g_1(V_j\setminus \bigcup_{k< j}X_k)\cap
V_{j,r}$. This last property implies (d).\\
For the property (e), using Remark \ref{weak version of statement A_2'} one can assume that $g_j\circ\cdots\circ
g_1$ moves no point more than $\sum^{j}_{k=1}\varepsilon_k$, because by the induction hypothesis
$g_{j-1}\circ\cdots g_1$ moves no point more than $\sum^{j-1}_{k=1}\varepsilon_k$. Let us check then that the
second property in \eqref{g_j does not move the points too much} is satisfied. Take $x\in V_j\setminus
(\bigcup_{k\leq j}X_j)$ and observe that if $g_k\circ\cdots\circ g_1(x)\neq g_{k-1}\circ\cdots\circ g_1(x)$ for
some $k=1,\cdots ,j$ means that $x\in V_k$, hence $k\in I_j$ and $i_k\geq j$. We can write
\begin{equation*}
||g_j\circ\cdots \circ g_1(x)-x||\leq \sum_{k\in I_j, k\leq j}\left({1\over 2^{k}} \cdot{1\over
2^{i_k+1}}\right)\leq \sum^{j}_{k=1}{1\over 2^k}\cdot{1\over 2^{j+1}}<{1\over 2^{j+1}}  .
\end{equation*}
\\
Now, using our induction hypothesis (f) that $g_{j-1}\circ\cdots\circ g_1$ refines $\mathcal{G}$, we can prove that
$g_j\circ\cdots\circ g_1$ still refines $\mathcal{G}$.  If we take an $x\in E\setminus (\bigcup_{k\leq j} X_k)$ and
$g_{j}\circ\cdots\circ g_1(x)=g_{j-1}\circ\cdots\circ g_1(x)$, since $g_{j-1}\circ\cdots g_1$ refines $\mathcal{G}$
we are done. Otherwise  $g_{j-1}\circ\cdots g_1(x),g_j\circ\cdots g_1(x)\in g_{j-1}\circ\cdots\circ
g_1(V_j\setminus \bigcup_{k<j}X_k)\cap V_j$, so $x,g_j\circ\cdots \circ g_1(x)\in
V_j$. We must have $x,g_j\circ\cdots\circ g_1(x)\in V_{j,r}$ for some $r\in \Omega$. Indeed, otherwise $x\in V_{j,r}$ and $g_j\circ\cdots\circ g_1(x)\in V_{j,r'}$ for different $r,r'\in\Omega$, a contradiction with \eqref{g_j does not move the points too much} since $\textrm{dist} (V_{j,r},V_{j,r'})\geq {1\over 2^{j+1}}$. \\
Hence we have finished our induction process.
\end{enumerate}
To conclude, note that $g_j\circ\cdots\circ g_1(x)\neq g_{j-1}\circ\cdots\circ g_1(x)$ implies $x\in V_j$. This
fact ensures the existence of a well-defined $C^{p}$ diffeomorphism $g(x)=\lim_{j\to\infty}g_j\circ\cdots\circ
g_1(x)$ from $E\setminus X$ onto $E\setminus (X\setminus U)$. The mapping $g$ is the identity on $(E\setminus
U)\setminus X$ and because $\left\lbrace V_{j,r}\right\rbrace _{j\geq 1,r\in\Omega}$ is a refinement of
$\mathcal{G}$, $g$ is limited by $\mathcal{G}$.
\end{proof}

\bigskip


\section{Proof of Theorem \ref{Main theorem for reflexive spaces}}

First of all notice that since the result we want to establish is invariant by diffeomorphisms, it is enough to
prove it for a $C^1$ manifold $M$ diffeomorphic to $E$ in place of $E$. It will be very convenient for us to do so
with $M=S^{+}$, the upper sphere of $E\times\R$.

Let $\|\cdot\|$ denote an equivalent norm in $E$ which is LUR and $C^1$ (we will also denote by $\|\cdot\|$ the
norm of $F$; this will do not do any harm because there will be no risk of confusion). Since $E^{*}$ is separable
there always exists such a norm; see \cite[Corollary II.4.3]{DGZ} for instance.

Let us define $Y=E\times\R$, with norm
$$
|(u,t)|=\left( \|u\|^2 +t^2 \right)^{1/2},
$$
and let us denote the upper sphere of $Y$ by
$$
S^{+}:=\{ (u,t) : u\in E, t>0, \|u\|^{2}+t^2=1\}.
$$
Observe that $S^{+}$ is the graph of
$$
s(u)=\sqrt{1-\|u\|^{2}},
$$
which is a $C^1$ function\footnote{Smoothness of $s(u)$ at $u=0$ is a consequence of the facts that $\|\cdot\|^2$
is trivially differentiable at $0$, and that an everywhere differentiable convex function is always of class $C^1$;
see for instance \cite[Corollary 4.2.12]{BorweinVanderwerff}.} defined on the open unit ball $B_E$ of $E$; hence,
$$
d(u)=(u,s(u))=(u,\sqrt{1-\|u\|^{2}}),
$$
$u\in B_E$, defines a $C^1$ diffeomorphism of $B_E$ onto $S^+$. Since $B_E$ is obviously $C^1$ diffeomorphic with
$E$, the upper sphere $S^+$, which is a $C^1$ submanifold of codimension $1$ of $Y$, is also diffeomorphic to $E$.

Therefore it will be enough to prove that every continuous mapping $f:S^{+}\to F$ can be $\varepsilon$-approximated
by a mapping $\varphi:S^{+}\to F$ which is of class $C^1$ and has no critical points. As explained in the
introduction, this will be done in three steps, the first of which consists in finding a smooth approximation of
$f$ whose critical set is a set that we can extract with the help of Theorem \ref{final extractibility theorem
rough version general form}.

In order to find such a smooth approximation, as in \cite{AzagraJimenez} we will have to use a partition of unity
$\{\psi_n : n\in\N\}$ in $S^{+}$ made out of slices of the unit ball of $Y$ by linear functionals $g_k\in Y^{*}$,
so that the derivative at $y\in S^{+}$ of a local sum $\sum \psi_k$ will belong to the span of the restrictions to
$T_{y}S^{+}$ (the tangent space to $S^{+}$ at $y$) of a finite collection of $g_k$. However, the construction of
the partition of unity, the technical properties that we will require, and the use that we will make of it, will be
much simpler than in that paper.

To construct our partition of unity $\{\psi_n\}_{n\in\N}$ we next translate an old standard argument (going back to
Eells in the Hilbert space case, and probably first appearing in \cite[p. 28-30]{Lang1962}, later generalized by
Bonic and Frampton \cite{BonicFrampton} for separable Banach spaces with smooth bump functions; we follow
\cite[Theorem 8.25]{FHHMZ}) to the upper sphere, as in \cite{Vanderwerff, Frontisi, AzagraJimenez}.

Let us denote\ $S:=S_{|\cdot|}$, the unit sphere of $(Y,|\cdot|)$, and $S^*:=S_{|\cdot|^*}$, the unit sphere of
$(Y^*,|\cdot|^*)$. The duality mapping of the norm $|\cdot|$, defined as
\begin{align*}
 D&: S \longrightarrow S^* \\ D&(x)=|\cdot|'(x),
\end{align*}
is $|\cdot|-|\cdot|^*$ continuous since the norm $|\cdot|$ is of class $C^1$.

Since the norm  $|\cdot|$ is locally uniformly convex we can find, for every $x\in S^+$, open slices $R_x=\{y\in S:
\ g_x(y)>\delta_x\}\subset S^+$  and $P_x=\{y\in S: \ g_x(y)>\delta_x^2\}\subset S^+$, where $g_x=D(x)\in Y^{*}$,
$0<\delta_x<1$,  and $|g_x|^*=1=g_x(x)$, so that the oscillation of the functions $f$ and $\varepsilon$ on every
$P_x$ is less than $\varepsilon(x)/16$. We also assume, with no loss of generality, that $\textrm{dist}(P_x,\,
E\times\{0\}\,)>0$.

Since $Y$ is separable we can select a countable subfamily of $\{R_x\}_{x\in S^+}$ which covers $S^+$.  Let us
denote this countable subfamily by $\{R_n\}_n$, where $R_n=R_{x_n}=\{y\in S: g_n(y)>\delta_n\}$ and $g_{n}(x_n)=1$.
Recall that the oscillation of the functions $f$ and $\varepsilon$ on every $ P_n=P_{x_n}=\{y\in S:
g_n(y)>\delta_n^2\}$ is less than $\varepsilon(x_n)/16$; this implies that
$$
\frac{15}{16}\, \varepsilon(x_n)\leq \varepsilon(x)\leq \frac{17}{16}\,\varepsilon(x_n) \textrm{ and }
\|f(x)-f(y)\|\leq\frac{\varepsilon(x_n)}{16}
$$
for every $x, y\in P_n$. Note also that $\{P_n\}_{n\in\N}$ is an open cover of $S^{+}$.

For each $k\in\N$, let $\theta_k :\R\to [0, 1]$ be a $C^{\infty}$ function such that $\theta_k(t)=1$ if and only if
$t\geq \delta_k$, and $\theta(t)=0$ if and only if $t\leq \delta_{k}^{2}$. Next, for each $k\in\N$ we define
$\varphi_{k}:S^{+}\to [0,1]$ by
$$
\varphi_{k}(x)=\theta_k( g_{k}(x)),
$$
and note that the interior of the support of $\varphi_k$, which coincides with $\varphi_k^{-1}((0,1])$, is the open
slice $P_k=\{y\in S: g_{k}(x)>\delta_{k}^{2}\}$.

Now, for $k=1$, define $h_1:S^{+}\to\R$ by
$$
h_1(x)=\varphi_1(x).
$$
Notice that the interior of the support of $h_1$ is the open set $U_1:=P_1$.

For $k\geq 2$ let us define $h_k:S^{+}\to\R$ by
$$
h_k(x)=\varphi_k(x)\prod_{j<k}\left(1-\varphi_{j}(x)\right),
$$
and notice that the interior of the support of $h_k$ is the set
$$
U_k :=\{y\in S^{+}: g_{k}(y)>\delta_{k}^2 \textrm{ and } g_{j}(y) <\delta_j \textrm{ for all } j<k\}.
$$
\begin{claim}\label{locally finiteness of the partition of unity}
The family $\{U_k\}_{k\in \mathbb N}$ is a locally finite open covering of $S^{+}$ that refines
$\{P_{k}\}_{k\in\N}$. Therefore the functions
$$
 \psi_{n}:=\frac{h_n}{\sum_{k=1}^{\infty}h_k}, \,\,\, n\in\N,
$$
define a $C^{1}$ partition of unity in $S^{+}$ subordinate to $\{P_{k}\}_{k\in\N}$.
\end{claim}
\begin{proof}
Given $j\in\N$, if $x, y\in \{y: g_{j}(y)>\delta_j\}$ and $k>j$ then $\varphi_{j}(y)=1$, hence $h_k(y)=0$. Since
$\{y: g_{j}(y)>\delta_j\}$ is a neighborhood of $x$ in $S^{+}$ this implies that the family of supports of the
$h_k$ is locally finite. On the other hand, if $h_k(x)>0$ for some $x, k$ then $\varphi_k(x)>0$, hence $x\in P_k$,
and this shows that the family of the open supports of the functions $h_k$, which coincides with
$\{U_k\}_{k\in\mathbb{N}}$, refines $\{P_{k}\}_{k\in\N}$. It only remains to prove that the family of the open
supports of the functions $h_k$ is indeed a cover, that is, for every $x\in S^{+}$ there exists some $n_x$ such
that $h_{n_x}(x)>0$. We argue by contradiction: assume we had $h_k(x)=0$ for all $k\in\N$, then we can show by
induction that  $\varphi_{n}(x)=0$ for all $n\in\N$, which implies that $x\notin P_n$ for all $n$ and contradicts
the fact that $\{P_n\}_{n\in\N}$ covers $S^{+}$. Indeed, for $n=1$ we have $0=h_1(x)=\varphi_1(x)$. Now suppose
that we have $\varphi_1(x)=\varphi_2(x)=...=\varphi_n(x)=0$. Then
$$
0=h_{n+1}(x)=\varphi_{n+1}(x)\prod_{j<n+1}\left(1 -\varphi_j(x)\right)=\varphi_{n+1}(x)=0,
$$
so it is also true that $\varphi_{n+1}(x)=0$.
\end{proof}

We will employ the following remarkable fact.
\begin{claim}\label{Dh_j vanishes where h_j=0}
For every $k\in \N$ and every $y\in S^+$, we have
$$
h_k(y)=0 \Longrightarrow Dh_k(y)=0.
$$
\end{claim}
\begin{proof}
Let us assume that $h_k(y)=0$. Suppose first that $\varphi_k(y)=0$. Computing the derivative of $h_k$ at $y$ we get
that
$$Dh_k(y)=D\varphi_k(y)\prod_{j<k}\left(1-\varphi_{j}(y)\right)+\varphi_k(y)D\left( \prod_{j<k}\left(1-\varphi_{j}(y)\right)\right) =D\varphi_k(y)\prod_{j<k}\left(1-\varphi_{j}(y)\right).$$
But $\varphi_k(y)=\theta_k(g_k(y))=0$ implies that $D\varphi_k(y)=0$, hence $Dh_k(y)=0$.\\
If $\varphi_k(y)\neq 0$ we must have $\varphi_{j}(y)=1$ for some $j<k$. But again it follows that
$D\varphi_j(y)=0$, so one can check that also
$$Dh_k(y)=\varphi_k(y)D\left( \prod_{j<k}\left(1-\varphi_{j}(y)\right)\right) =0.$$
\end{proof}
Notice that, according to Claim \ref{locally finiteness of the partition of unity}, for every $x\in S^{+}$ there
exist a number $n=n_x$ and an open neighborhood $V_x$ of $x$ in $S^{+}$ such that $\sum_{j\leq n}h_j(y)>0$ for
every $y\in V_x$, and $h_{k}(y)=0$ for every $k> n$. This means that $\psi_k(y)=0$ for every $k>n$ and every $y\in
V_x$. More precisely, $n_x$ can be chosen as the first such $j$ so that $x\in \{y\in S^+: g_{j}(y)>\delta_j\}= R_j$
and $V_x=R_j$.

Let us also call $m=m_y$ the largest $j$ such that $h_j(y)\neq 0$. Note that $m_y$ is also the largest $j$ for
which $\psi_j(y)\neq 0$. Thus, for every $y\in V_x$, we have
$$
m_y=\max\{j: y\in\psi^{-1}_j((0,1])\}\leq n_x.$$

In order to calculate the derivatives of this partition of unity in this neighborhood $V_x$ of $x$, let us
introduce the functions
$$
H_k(y)=\frac{h_k (y)}{\sum_{j=1}^{n}h_j (y)}, \,\,\,\,\,\,\,\,\,\, y\in V_x, \,\,\, k=1, ..., n,
$$
which are well defined on $V_x$ and in fact can be extended as $C^1$ smooth functions to an open subset
$\mathcal{S}_n$ of $Y$ containing $V_x$. Specifically, noting that each $h_k$ is well-defined and $C^\infty$ smooth
on the whole $Y$, one can choose $\mathcal{S}_n=\bigcup_{i=1}^n h_i^{-1}((0,1])$. Slightly abusing notation, we
will keep denoting these extensions by $H_k$, and we will also think of the functions $h_j$, $\psi_j$, $g_j$,
$j\leq n$, as being $C^1$ smooth functions defined on this open set $\mathcal{S}_n$. Therefore, to calculate the
derivative of $\psi_k$ on $V_x$ for $k=1, ..., n$, we only have to calculate the derivative of $H_k$ at each $y\in
V_x$ for $k=1, ..., n$ and then restrict it to the tangent spaces $T_{y}S^{+}$. The exact expression for the
derivative of the functions $H_1, ..., H_n$ on $V_x$ will not be particularly interesting or useful to us. The only
thing we need to know is that there are $C^{1}$ smooth functions $\sigma_{k, j}$, $1\leq k, j\leq n$, (actually,
$\sigma_{k, j}$ will be $C^\infty$ smooth) defined on $\mathcal{S}_n$ so that
$$
DH_k(y)=\sum_{j=1}^{n}\sigma_{k,j}(y)\, g_{j}
$$
for $k=1, ..., n$, $y\in V_x$, and that, in fact, as an immediate consequence of  Claim \ref{Dh_j vanishes where
h_j=0} and the definition of $m_y$ we have
$$
DH_k(y)=\sum_{j=1}^{m_y}\sigma_{k,j}(y)\, g_{j}.
$$
for $k=1,\dots,m_y$, $y\in V_x$.

Note that even though each $H_k$ is $C^{\infty}$ smooth on an open subset $\mathcal{S}_n$ of $Y$, we cannot say
that $\psi_k$ is $C^{\infty}$ smooth too, because $S^{+}$ has not a $C^{\infty}$ smooth submanifold structure. The
functions $\psi_k$ are just $C^{1}$ because $S^{+}$ is just a $C^1$ manifold modeled on $E$. Now, if we want to
know what the derivative of the functions $\psi_k$, $k=1, ..., n$, on $V_x$ looks like, because $H_k=\psi_k$ on
$V_x$ and $V_x$ is open in $S^{+}$, we only have to restrict $DH_k$ to the tangent spaces $T_{y}S^{+}$ for each
$y\in V_{x}$. Hence we have
$$
D\psi_k(y)(v)=\sum_{j=1}^{n}\sigma_{k,j}(y) \, g_{j}(v)=\sum_{j=1}^{m_y}\sigma_{k,j}(y) \, g_{j}(v)
$$
for each $v\in T_{y}S^{+}$.

This expression can be somewhat misleading at first sight, because one might think that, for $k=1, ..., n$,
$D\psi_k(y)$ is just a linear combination of the functionals $g_{1}, ..., g_{n}$, and this is not exactly so. It is
a linear combination of the restrictions of $g_{1}, ..., g_{n}$ to $T_{y}S^{+}$, and therefore for every $y\in V_x$
it is a different linear combination of different linear functionals ${g_{1}}_{|_{T_{y}S^{+}}}$, ...,
${g_{1}}_{|_{T_{y}S^{+}}}$, each of them defined on a space depending on $y$.

In order to fully clarify this important point, let us calculate the tangent space $T_{y}S^{+}$ at $y=(u_y,t_y)$.
Since $S^{+}$ is the graph of the function $s(u)=\sqrt{1-\|u\|^{2}}$, the most natural representation of
$T_{y}S^{+}$ is given by
\begin{equation}\label{expression for the tangent hyperplane to S}
T_{y}S^{+}=\{ (u, t)\in Y= E\times \R \, : \, t=L_{y}(u),  u\in E\}=\{ (u, L_{y}(u)) : u\in E\} \subset Y,
\end{equation}
where $L_y$ is the derivative $Ds(u_y)$ of the function $s$ evaluated at the point $u_y=d^{-1}(y)\in B_E$ (recall
that $d(u)=\left(u,s(u)\right)$); in other words, if $y\neq (0,1)$,
$$
L_{y}(w)= - \frac{\|u_y\|}{\sqrt{1- \|u_y\|^{2}}} D\|\cdot\|(u_y)(w) =- \frac{\sqrt{1-t^2_y}}{{t_y}}
D\|\cdot\|(u_y)(w)
$$
for each $w\in E$. Of course we have $L_{(0,1)}(w)=Ds(0)(w)=0$ for each $w\in E$, and $T_{(0,1)}S^+=E\times \{0\}$.

This is the vectorial tangent hyperplane to $S^{+}$ at $y$, as opposed to the affine tangent hyperplane to $S^{+}$,
which is just $y+T_{y}S^{+}$. Since derivatives of mappings act on vectorial tangent hyperplanes we may forget the
affine hyperplanes $y+T_{y}S^{+}$ in what follows.

Now, recall that $g_{j}\in Y^{*}$ and therefore these functionals are of the form
$$
g_{j}(u,t)=g_{j}^{1}(u) +g_{j}^{2}t, \,\,\,\,\,\,\,\,\,\, (u,t)\in E\times\R=Y,
$$
where $g_{j}^{1}\in E^{*}$ and $g_{j}^{2}\in\R$.

Therefore the derivative of ${g_{j}}_{|_{S^{+}}}$ at $y\in S^{+}$ is given by
\begin{equation}\label{formula for Dgn restricted to TyS}
Dg_{j}(y)(u, L_{y}(u))= g_{j}(u, L_{y}(u))=g_{j}^{1}(u)+ g_{j}^{2}L_{y}(u)
\end{equation}
for every $v=(u, L_y(u))\in T_{y}S^{+}$. Thus, for every $y\in V_x$ and every $k=1, ..., n$, we have
$$
D\psi_k(y)(v)=D\psi_k(y)(u, L_{y}(u))=\sum_{j=1}^{n}\sigma_{k,j}(y) \, (g_{j}^{1}(u)+ g_{j}^{2}L_{y}(u))
=\sum_{j=1}^{m_y}\sigma_{k,j}(y) \, (g_{j}^{1}(u)+ g_{j}^{2}L_{y}(u))
$$
for every $v=(u, L_{y}(u))\in T_{y}S^{+}$.

Finally, let us note that the points $x_n\in R_n$ satisfy that
$$
\frac{15}{16}\, \varepsilon(x_n)\leq \varepsilon(y)\leq \frac{17}{16}\,\varepsilon(x_n) \,\,\, \textrm{ for every }
y\in P_n, \,\,\, \textrm{ and } \,\,\, \sup_{x, y\in P_n}\|f(x)-f(y)\|\leq \frac{\varepsilon(x_n)}{16}.
$$

The following lemma summarizes the properties of the partition of unity $\{\psi_n\}_{n\in\N}$ which will be most
useful to us.
\begin{lemma}\label{properties of the partition of unity on the upper sphere reflexive case}
Given two continuous functions $f:S^{+}\to F$ and $\varepsilon:S^{+}\to (0, \infty)$, there exists a collection of
norm-one linear functionals $\{g_k\}_{k\in\N}\subset Y^{*}$, an open covering $\{P_n\}_{n\in\N}$ of $S^{+}$, and a
$C^1$ partition of unity $\{\psi_{n}\}_{n\in\N}$ on $S^{+}$ such that:
\begin{enumerate}
\item $\{\psi_n\}_{n\in\N}$ is subordinate to $\{P_n\}_{n\in\N}$.
\item For every $x\in S^{+}$ there exist a neighborhood $V_x$ of $x$ in $S^{+}$ and a number $n=n_{x}\in\N$ such that $\psi_m=0$ on $V_x$ for all $m>n$, and the derivatives of the functions $\psi_1, ..., \psi_{n}$ on $V_x$ are of the form
$$
D\psi_{k}(y)(v)=\sum_{j=1}^{n}\sigma_{k, j}(y)\, g_{j}(v)=\sum_{j=1}^{m_y}\sigma_{k, j}(y)\, g_{j}(v),
$$
for $v\in T_{y}S^{+}$, the tangent hyperplane to $S^{+}$ at $y\in S^{+}\cap V_x$, and where $m_y\leq n$ is the
largest number such that $\psi_{m_y}(y)\neq 0$. More precisely, if $L_y$ denotes the derivative of the function
$s(u)=\sqrt{1-\|u\|^2}$ evaluated at the point $d^{-1}(y)$, where $d(u)=(u,s(u))$, we have
$$
D\psi_k(y)(v)=D\psi_k(y)(u, L_{y}(u))=\sum_{j=1}^{n}\sigma_{k,j}(y) \, (g_{j}^{1}(u)+ g_{j}^{2}L_{y}(u))
=\sum_{j=1}^{m_y}\sigma_{k,j}(y) \, (g_{j}^{1}(u)+ g_{j}^{2}L_{y}(u))
$$
for every $k=1, ..., n$, and for every $v=(u, L_{y}(u))\in T_{y}S^{+}$, where the functions $\sigma_{k,
j}:V_x\to\R$ are of class $C^{1}$, and $g_{j}^{1}\in E^{*}$, $g_{j}^{2}\in\R$, $j=1, ..., n$.
\item For every $n\in\N$ there exist a point $y_n:=x_n\in P_n$ such that
$$
\frac{15}{16}\, \varepsilon(y_n)\leq \varepsilon(y)\leq \frac{17}{16}\,\varepsilon(y_n) \,\,\, \textrm{ for every }
y\in P_n, \,\,\, \textrm{ and } \,\,\, \sup_{x, y\in P_n}\|f(x)-f(y)\|\leq \frac{\varepsilon(y_n)}{16}.
$$
\end{enumerate}
\end{lemma}

Now we are ready to start the construction of our approximating function $\varphi:S^{+}\to F$, which will be of the
form
$$
\varphi(x)=\sum_{n=1}^{\infty}\left( f(y_n) +T_{n}(x) \right) \psi_{n}(x),
$$
where the $y_n$ are the points given by condition $(4)$ of the preceding lemma, and the operators $T_{n}:Y\to F$
will be carefully defined below.

\medskip

{\bf Case 1: Assume that $F$ is infinite-dimensional.}

We will have to make repeated use of the following fact, whose proof is elementary and can be left to the interested reader.
\begin{lemma}\label{decomposition trick}
If $E$ is a Banach space which is isomorphic to $E\oplus E$ then for every finite-codimensional closed subspace $V$
of $E$, there exists a decomposition
$$
E=E_{1}\oplus E_{2}\oplus E_{3},$$ with factors $E_{1}, E_{2}, E_{3}$ isomorphic to $E$, such that $E_{1}\oplus
E_{2}\subset V$.
\end{lemma}

We start considering a decomposition
$$
E=E_{1,1}\oplus E_{1,2},
$$
with infinite-dimensional factors isomorphic to $E$, and we define a continuous linear surjection $S_1:E\to F$ such
that $S_1=0$ on $E_{1,2}$. This can be done by taking a continuous linear surjection $R_1:E_{1,1}\to F$ (which
exists because by assumption there exists such an operator from $E$ onto $F$ and $E_{1,1}$ is isomorphic to $E$),
and setting $S_1=R_1\circ P_{1,1}$, where $P_{1,1}:E_{1,1}\oplus E_{1,2}\to E_{1,1}$ is the projection onto the
first factor associated to this decomposition of $E$. Next, recall that the linear functionals $g_{n}\in
Y^{*}=(E\times\R)^{*}$ are of the form
$$
g_{j}(u,t)=g_{j}^{1}(u)+ g_{j}^{2}t,
$$
where $g_{j}^{1}\in E^{*}$ and $g_{j}^{2}\in\R$. Of course, $\bigcap_{j=1}^{2} \textrm{Ker}\, g_{j}^{1}$ is a
finite-codimensional subspace of $E$, so by using Lemma \ref{decomposition trick} (2) with
$V=E_{1,2}\cap\bigcap_{j=1}^{2}\textrm{Ker}(g_{j}^{1})$, and with $E_{1,2}$ in place of $E$, we may find a
decomposition of the second factor $E_{1,2}$,
$$
E_{1,2}= E_{2,1}\oplus E_{2,2} \oplus  E_{2,3},
$$
with factors $E_{2,1}$, $E_{2,2}$  and $E_{2,3}$ isomorphic to $E$, and
$$
E_{2,1} \oplus E_{2,2}\subset \bigcap_{j=1}^{2} \textrm{Ker}\, g_{j}^{1},
$$
and we may easily define a bounded linear operator $S_{2}$ from $E$ onto $F$ such that $S_2=0$ on $E_{1,1}\oplus
E_{2,2}\oplus E_{2,3}$ (this can be done by taking a surjective operator $R_{2}: E_{2,1}\to F$ and defining
$S_2=R_2\circ P_{2,1}\circ P_{1,2}$, where $P_{1,2}:E_{1,1}\oplus E_{1,2}\to E_{1,2}$ and $P_{2,1}: E_{2,1}\oplus
E_{2,2}\oplus E_{2,3}\to E_{2,1}$ are the projections associated to the corresponding decompositions).

We continue this process by induction: assuming that we have already defined decompositions $E=E_{1,1}\oplus
E_{1,2}$, \, $E_{1,2}=E_{2,1}\oplus E_{2,2}\oplus E_{2,3}$, \, $E_{2,3}=E_{3,1}\oplus E_{3,2}\oplus E_{3,3}$, ...,
$E_{n-1,3}=E_{n,1}\oplus E_{n,2}\oplus E_{n,3}$, and surjective operators
$$
S_{k}:E=E_{1,1}\oplus( E_{2,1}\oplus E_{2,2})\oplus ... \oplus (E_{k-1,1}\oplus E_{k-1,2})\oplus( E_{k,1}\oplus
E_{k,2}\oplus E_{k,3})\to F, \,\,\, k=2, ..., n,
$$
so that $S_k$ is zero on all the factors of this decomposition except $E_{k,1}$, and
$$
E_{k,1}\oplus E_{k,2}\subset\bigcap_{j=1}^{k}\textrm{Ker}(g_{j}^{1}),
$$
we again apply Lemma \ref{decomposition trick} (2) to write
$$
E_{n,3}=E_{n+1, 1}\oplus E_{n+1, 2}\oplus E_{n+1,3 },
$$
with factors isomorphic to $E$ and
$$
E_{n+1,1}\oplus E_{n+1,2}\subset\bigcap_{j=1}^{n+1}\textrm{Ker}(g_{j}^{1}),
$$
and we define a continuous linear surjection
$$
S_{n+1}:E=E_{1,1}\oplus (E_{2,1}\oplus E_{2,2})\oplus ... \oplus (E_{n,1}\oplus E_{n,2})\oplus (E_{n+1,1}\oplus
E_{n+1,2}\oplus E_{n+1,3})\to F,
$$
by setting it equal to $0$ on all the factors of this decomposition except $E_{n+1,1}$, which is mapped onto $F$.

Having this collection of surjective operators $S_{n}:E\to F$ at our disposal, we finally define $T_{n}:Y\to F$ by
$$
T_{n}(u,t)= \frac{\varepsilon(y_n)}{4\|S_{n}\|}\, S_{n}(u),
$$
and $\varphi:S^{+}\to F$ by
$$
\varphi(x)=\sum_{n=1}^{\infty}\left( f(y_n) +T_{n}(x) \right) \psi_{n}(x).
$$
It is clear that $\varphi$ is well defined and of class $C^1$.

\medskip

In the rest of the proof we will check that this mapping $\varepsilon$-approximates $f$ on $S^{+}$, and that the
set of critical points of $\varphi$ is a set which can be diffeomorphically extracted by using Theorem \ref{final
extractibility theorem rough version general form}. Then the proof will be completed by setting $g=\varphi\circ h$,
where $h:S^+\to S^+\setminus C_{\varphi}$ is a $C^{1}$ diffeomorphism which is close enough to the identity.
\medskip

\begin{claim}\label{varphi epsilon approximates f}
The mapping $\varphi$ approximates $f$.
\end{claim}
\begin{proof}
By condition $(3)$ of Lemma \ref{properties of the partition of unity on the upper sphere reflexive case}, we know
that the oscillation of $f$ in $P_n$ is less that $\varepsilon(y_n)/16$, and by definition of $T_n$, we have
$\|T_n\|\leq \varepsilon(y_n)/4$, hence $\|T_n(x)\|\leq \varepsilon(y_n)/4$ for all $x\in S^{+}$ too, because
$\|x\|=1$. Now, if $\psi_n(x)\not=0$, then $x\in P_n$ and
\begin{align}\label{akrk-F}
\|f(y_n)+T_n(x)-f(x)\|& \leq\|f(y_n)-f(x)\|+\|T_n(x)\| \\
\notag &\leq \varepsilon(y_n)/16 +\varepsilon(y_n)/4=\frac{5}{16}\varepsilon(y_n)<\frac{15}{32}\varepsilon(y_n)\leq
\varepsilon(x)/2.
\end{align}
Therefore
\begin{align*}
\|\varphi(x)-f(x)\|=\bigl\|\sum_{n=1}^\infty(f(y_n)+T_n(x)-f(x))\,\psi_n(x) \bigr\|\leq \sum_{n=1}^\infty
\|f(y_n)+T_n(x)-f(x)\|\,\psi_n(x)\leq \varepsilon(x)/2.
\end{align*}
\end{proof}

\medskip

Let us now consider the question as to how big the critical set $C_{\varphi}$ can be. We need to calculate the
derivative of our function $\varphi$. To this end we first have to examine the expressions for the derivatives of
the operators $T_n$ restricted to $S^{+}$. These are simpler than those of the $g_n$'s on $S^{+}$, because of the
way the operators $T_n :Y\to F$ have been defined. Indeed, for every $v=(u, L_y(u))\in T_{y}S^{+}$ we have that
\begin{equation}\label{formula for Tn restricted to TyS reflexive case}
T_{n}(v)=T_{n} (u, L_{y}(u))=\frac{\varepsilon(y_n)}{4\|S_n\|} \, S_{n}(u),
\end{equation}
and we have that $D({T_{n}}_{|_{S^{+}}})(y)$ is the restriction of $DT_n(y)=T_n$ to $T_{y}S^{+}$, that is to say,
if $v=(u, L_y(u))\in T_{y}S^{+}$ then
\begin{equation}\label{formula for DTn restricted to TyS reflexive case}
DT_{n}(y)(u, L_{y}(u))=\frac{\varepsilon(y_n)}{4\|S_n\|} \, S_{n}(u).
\end{equation}
Now, recall that, by condition $(2)$ of Lemma \ref{properties of the partition of unity on the upper sphere
reflexive case}, for every $x\in S^{+}$ there is a neighborhood $V_x$ of $x$ in $S^{+}$ and a number $n=n_{x}\in\N$
such that
$$
\varphi(y)=\sum_{j=1}^{n}\left( f(y_j) +T_{j}(y) \right) \psi_{j}(y)
$$
for every $y\in V_x$. Fix $y\in V_x$ and recall that for $m=m_y$, the largest number $j$ for which $h_j(y)\neq 0$
(or $\psi_j(y)\neq 0$), we have
$$
\varphi(y)=\sum_{j=1}^{m}\left( f(y_j) +T_{j}(y) \right) \psi_{j}(y).
$$
By using \eqref{formula for DTn restricted to TyS reflexive case} and the expression for $D\psi_j(y)$ given in
Lemma \ref{properties of the partition of unity on the upper sphere reflexive case}, we easily see that
\begin{align}\label{expression for the derivative of varphi reflexive case}
D\varphi(y)(u, L_{y}(u))&=\sum_{j=1}^{n}\psi_{j}(y)\, \frac{\varepsilon(y_j)}{4\|S_j\|}\, S_{j}(u) +
\sum_{j=1}^{n}\left(g_{j}^{1}(u)+ g_{j}^{2}L_{y}(u)\right)
\alpha_{n, j}(y)= \nonumber \\
&=\sum_{j=1}^{m}\psi_{j}(y)\, \frac{\varepsilon(y_j)}{4\|S_j\|}\, S_{j}(u) + \sum_{j=1}^{m}\left(g_{j}^{1}(u)+
g_{j}^{2}L_{y}(u)\right) \alpha_{m, j}(y)
\end{align}
for every $(u, L_{y}(u))\in T_{y}S^{+}$, $y\in V_x$, where the functions $\alpha_{n,j}:V_x\subset S^{+}\to F$  are
of class $C^1$ because $\alpha_{n,j}(y)=\sum_{i=1}^n(f(y_i) +T_{i}(y))\sigma_{i,j}(y)$.

\medskip

Let us now show that the critical set of $\varphi$ is relatively small.

When $n_x=1$ we have $\varphi(y)=f(y_1)+T_{1}(y)$ on $V_x$, so $D\varphi(y)$ is the restriction of $T_1$ to
$T_{y}S^{+}$, and equation \eqref{formula for Tn restricted to TyS reflexive case} for $n=1$ implies that this
restriction is a surjective operator. So it is clear that $\varphi$ has no critical point on $V_x$.

\begin{claim}\label{Cvarphi is locally contained in Ax}
If $n_x\geq 2$ then $C_{\varphi}\cap V_{x}$ is contained in the set
$$
A_{x}:= \left\{ y\in S^{+} : E_{n,2}\subset \textrm{Ker}\, L_{y}\right\}.
$$
Recall that $L_y=Ds(u_y)$, where $s(u)=\sqrt{1-\|u\|^2}$, $d(u)=(u,s(u))$, and $u_y=d^{-1}(y)$.
\end{claim}
\begin{proof}
Let us see that, if $y\in V_x\setminus A_{x}$ then $D\varphi(y):T_yS^+\to F$ is surjective, that is, for every
$w\in F$ there exists $v\in T_{y}S^{+}$ such that $D\varphi(y)(v)=w$. Let $m=m_y$ be the largest number such that
$\psi_m(y)\neq 0$. Recall that $m\leq n$.  Since the operator
$$
S_{m}:E=E_{1,1}(\oplus E_{2,1}\oplus E_{2,2})\oplus ... \oplus (E_{m-1,1}\oplus E_{m-1,2})\oplus( E_{m,1}\oplus
E_{m,2}\oplus E_{m,3})\to F
$$
is surjective and equal to zero on all the factors of this decomposition except $E_{m,1}$ (which is mapped onto
$F$), we may find $u_{m,1}\in E_{m,1}$ so that
$$
S_{m}(u_{m,1})=4\varepsilon(y_m)^{-1}\psi_{m}(y)^{-1}\|S_{m}\| \, w.
$$
Now, since $y\notin A_x$ there exists some $e_{n,2}\in E_{n,2}\setminus \textrm{Ker }L_y$, that is to say,
$L_{y}(e_{n,2})\neq 0$, and this implies that, if we put
$$
t_{0}:=-\frac{L_{y}(u_{m,1})}{L_{y}(e_{n,2})},
$$ then the vector
$$
u:= u_{m,1}+t_0 e_{n,2},
$$
satisfies that
$$
L_{y}(u)=0.
$$
But recall that, for every $k\le n$, we have  $E_{k,1}\oplus E_{k,2}\subset
\bigcap_{j=1}^{k}\textrm{Ker}(g_{j}^1)$; in particular, $E_{m,1}\oplus E_{m,2}\subset
\bigcap_{j=1}^{m}\textrm{Ker}(g_{j}^1)$ because $m\leq n$. Hence, $g_j^1(u)=0$ for every $1\le j\le m$. It follows
that $$\sum_{j=1}^{m}\left(g_{j}^{1}(u)+ g_{j}^{2}L_{y}(u)\right) \alpha_{m, j}(y)=0.$$

The rest of the operators $S_1, ..., S_{m-1}$ are zero on $E_{m, 1}\oplus E_{n,2}$, so we have
$$
S_{j}(u)=0 \textrm{ for every } j=1, ..., m-1,
$$
and since $S_m$ is zero on $E_{n,2}\subset E_{m,3}$ we also have $S_{m}(t_{0}e_{n,2})=0$. Therefore, by combining
these equalities with equation \ref{expression for the derivative of varphi reflexive case}, we obtain that
$$
D\varphi(y)(u, L_{y}(u))=w,
$$
and the proof of the claim is complete.
\end{proof}

\medskip

\begin{lemma}\label{Ax is a good graph}
For $x\in S^+$ with $n:=n_x\ge2$, the set $A_x$ of Claim \ref{Cvarphi is locally contained in Ax} is of the form
$$
A_x=d(G(f_x)\cap B_E),
$$
where $G(f_x)$ is the graph of a continuous mapping $f_x:E_{1,1}\oplus(E_{2,1}\oplus E_{2,2})\oplus \cdots \oplus
(E_{n,1}\oplus E_{n,3}) \to E_{n,2}$.
\end{lemma}
\begin{proof}
Note that, by Claim \ref{Cvarphi is locally contained in Ax},
$$
A_x=\{d(u)=\left(u, \sqrt{1-\|u\|^2}\right) : \|u\|<1, u\in\mathcal{A}_x\},
$$
where
$$
\mathcal{A}_x := \bigcap_{e\in E_{n,2}} \left\{u\in E\setminus \{0\} \, : \, \langle D\|\cdot\| (u), e\rangle
=0\right\}\cup\{0\}.
$$

Let us denote $E_{n,2}' =E_{1,1}\oplus (E_{2,1}\oplus E_{2,2 })\oplus\cdots \oplus (E_{n,1}\oplus E_{n,3})$, and
let us see that there exists a mapping $f_x: E_{n,2}'\to E_{n,2}$ such that $\mathcal{A}_x=G(f_x)=\{ w+f_x(w) :
w\in E_{n,2}'\}$.

Pick a point $w\in E_{n,2}'$. Note that the function $E_{n,2}\ni v\mapsto\psi_{w}(v) := \|w+v\|^2$ is convex and
continuous, and satisfies $\lim_{\|v\|\to\infty}\psi_{w}(v)=\infty$, hence, since $E_{n,2}$ is reflexive,
$\psi_{w}$ attains a minimum at some point $v_{w}\in E_{n,2}$; in fact this minimum point $v_w$ is unique because
the norm $\|\cdot\|$ is strictly convex. Let us denote
$$
f_x(w):=v_{w}.
$$
Note that the critical points of $\psi_{w}$, with $w\neq 0$, are exactly the points $v\in E_{n,2}$ such that
$$
{\frac{d}{dt} \|w+v+te\|^{2} \,}{|_{t=0}}=0 \textrm{ for every } e\in E_{n,2}
$$
or equivalently
$$
\|w+v\| \, \langle D\|\cdot\|(w+v), e\rangle=0 \textrm{ for every } e\in E_{n,2},
$$
which in turn is equivalent to saying that $w+v\in \mathcal{A}_x$; we let $f_x(0)=v_{0}=0$.

Therefore the unique point $v\in E_{n,2}$ so that $w+v\in \mathcal{A}_x$ is the point $v=f_x(w)$. This shows that
$\mathcal{A}_x$ is the graph of the function $f_x$.

Now let us see that the function $f_x: E_{n,2}'\to E_{n,2}$ is continuous. Suppose $f_x$ is discontinuous at $w_0$
and let $v_0:=f_x(w_0)$. Then there exist sequences $w_k\to w_0$ in $E_{n,2}'$ and $v_{k}:=f_x(w_k)$ in $E_{n,2}$
and a number $\varepsilon_0>0$ so that
\begin{equation}\label{vk does not converge to v0}
\| v_k -v_0 \| \geq \varepsilon_0 \textrm{ for all } k\in\N.
\end{equation}

From the previous argument we know that the point $v_k$ is characterized as being the unique point $v_k \in
E_{n,2}$ for which we have
\begin{equation}\label{characterization of vk}
\| w_k +v_k\|\leq \|w_k +v_k +e\| \textrm{ for all } e\in E_{n,2},
\end{equation}
and similarly $v_0$ is the unique point $v_0\in E_{n,2}$ for which
\begin{equation}\label{characterization of v0}
\| w_0 +v_0\|\leq \|w_0 +v_0 +e\| \textrm{ for all } e\in E_{n,2}.
\end{equation}
By taking $e=-v_k$ in \eqref{characterization of vk} we learn that
$$
\|v_k\|-\|w_k\|\leq \|w_k +v_k\|\leq \|w_k\|,
$$
hence $\|v_k\|\leq 2\|w_k\|$, and because $\|w_k\|$ converges to $\|w_0\|$ we deduce that $(v_k)$ is bounded. Since
$E_{n,2}$ is reflexive, this implies that $(v_k)$ has a subsequence that weakly converges to a point $\xi_0\in
E_{n,2}$. We keep denoting this subsequence by $(v_k)$.

Now, if we take $e=- v_k+e'$ in \eqref{characterization of vk}, with $e'\in E_{n,2}$, we obtain
$$
\|w_k+ v_k\|\leq  \|w_k +e'\| \textrm{ for all } e'\in E_{n,2}.
$$
This implies (using the facts that $v_k \rightharpoonup \xi_0$ and $w_k\to w_0$, and the weak lower semicontinuity
of the norm) that
\begin{equation}\label{liminf and vk}
\|w_0 +\xi_0\| \leq \liminf_{k\to \infty}\|w_k+ v_k\|\leq \liminf_{k\to\infty} \|w_k +e'\|=\|w_0+e'\| \textrm{ for
all } e'\in E_{n,2}.
\end{equation}
That is, we have shown that
\begin{equation}
\|w_0 +\xi_0\|\leq \|w_0+e'\| \textrm{ for all } e'\in E_{n,2}.
\end{equation}
By taking $e'=\xi_0 +\xi$ with $\xi\in E_{n,2}$ we conclude that
$$
\|w_0+\xi_0\|\leq \|w_0+\xi_0 +\xi\| \textrm{ for all } \xi\in E_{n,2}.
$$
According to \eqref{characterization of v0}, $v_0$ is the only point which can satisfy this inequality. Hence
$\xi_0=v_0$.

But \eqref{liminf and vk} tells us even more: by taking $e'=\xi_0$ we also learn that there exists a subsequence
$(w_{k_j})$ of $(w_k)$ such that
$$
\|w_{k_j}+v_{k_j}\|\to \|w_0 + \xi_0\|.
$$
Since we also know that $w_{k_j}+v_{k_j}$ converges to $w_0+\xi_0$ weakly and the norm $\|\cdot\|$ is locally
uniformly convex (hence $\|\cdot\|$ has the Kadec-Klee property), this implies that $w_{k_j}+v_{k_j}$ converges to
$w_0+\xi_0 =w_0+v_0$ in the norm topology as well. As we also have $\lim_{j\to\infty} w_{k_j}=w_0$ in norm, we
deduce that $\lim_{j\to\infty}\|v_{k_j}-v_0\|=0$, which contradicts \eqref{vk does not converge to v0}.
\end{proof}

\medskip

Now we can easily finish the proof of Theorem \ref{Main theorem for reflexive spaces}. By Claim \ref{Cvarphi is
locally contained in Ax} and Lemma \ref{Ax is a good graph}, we see that $C_{\varphi}$ is a diffeomorphic image in
$S^{+}$ of a relatively closed set $Z$ of the open unit ball $B_E$ of $E$ which has the property of being locally
contained in the graph of a continuous function defined on a complemented subspace of infinite codimension in $E$.
Indeed, let $Z:=d^{-1}(C_{\varphi})\subset B_E$. Since $C_{\varphi}$ is closed in $S^{+}$,  $Z$ is relatively
closed in $B_{E}$. Also if we take $z\in Z$ then, according to  Lemma \ref{Ax is a good graph} applied to
$x=d(z)\in S^+$, $n=n_x$, and $V_x$, for a neighborhood $U_z:=d^{-1}(V_x)$ of $z$,  we have $Z\cap U_z\subseteq
G(f_x)$, where
$$G(f_x)=\{u=(w,v)\in E'_{n,2}\oplus E_{n,2}=E:\,  v=f_x(w)\}.$$

Observe that $E$ has $C^1$ smooth partitions of unity since $E$ has a separable dual. Therefore we may apply
Theorem \ref{final extractibility theorem rough version general form} to find a $C^1$ diffeomorphism which extracts
$C_{\varphi}$ from $S^+$; more precisely, there exists a diffeomorphism $h:S^{+}\to S^+\setminus C_{\varphi}$
which, in addition, is limited by the open cover $\mathcal{G}$ that we next define. Recall that we have
\begin{equation}\label{varphi approximates up to varepsilon halves}
\|\varphi(x)-f(x)\|\leq \varepsilon(x)/2
\end{equation}
for all $x\in S^+$. Since $\varphi$ and $\varepsilon$ are continuous, for every $z\in S^+$ there exists
$\delta_{z}>0$ so that if $x,y\in B(z,\delta_{z})$ then $\|\varphi(y)-\varphi(x)\|\leq
\varepsilon(z)/4\leq\varepsilon(x)/2$. We set $\mathcal{G}=\{B(x,\delta_{x})\, :\, x\in S^{+}\}$.

Finally, let us define $$g=\varphi\circ h.$$

Since $h$ is limited by $\mathcal{G}$ we have that, for any given $x\in S^+$, there exists $z\in S^+$ such that $x,
h(x)\in B(z,\delta_{z})$, and therefore $\|\varphi(h(x))-\varphi(x)\|\leq\varepsilon(z)/4,$ that is, we have that
      $$
      \|g(x)-\varphi(x)\|\leq\varepsilon(z)/4\leq\varepsilon(x)/2.
      $$
By combining this inequality with \eqref{varphi approximates up to varepsilon halves}, we obtain that
$$
\|g(x)-f(x)\|\leq \varepsilon(x)
$$ for all $x\in S^{+}$.
Besides, it is clear that $g$ does not have any critical point: since $h(x)\notin C_{\varphi}$, we have that the
linear map $D\varphi(h(x)) :T_{h(x)}S^{+}\to F$ is surjective, and $Dh(x): T_{x}S^{+}\to T_{h(x)}S^{+}$ is a linear
isomorphism, so $Dg(x)=D\varphi(h(x))\circ D h(x)$ is a linear surjection from $T_{x}S^{+}$ onto $F$\, for every
$x\in S^{+}$.

\medskip

{\bf Case 2: Assume that $F=\R^m$.} The main idea of the proof is very similar to that of Case 1. The fact
that $F$ is finite dimensional will allow us dispense with the hypothesis that $E=E\oplus E$. We will use the same partition of unity $\{\psi_{n}\}_{n\in\N}$ provided by Lemma \ref{properties of the partition of
unity on the upper sphere reflexive case}. We will decompose $E$ inductively as follows. Since $Ker g^{1}_{1}$ has
infinite dimension we can write
$$E=E_1\oplus G_1,$$
where $E_1=\R^m$ and $G_1\subseteq \textrm{Ker} g^{1}_{1}$. Then $G_1\cap\bigcap^{2}_{j=1} \textrm{Ker} g^{1}_{j}$
has codimension $0$ or $1$ in $G_1$, which is infinite-dimensional, and we can write
$$E=E_1\oplus (E_{2,1}\oplus E_{2,2}\oplus G_2),$$
where $E_{2,1}=\R^m$, $E_{2,2}=\{0\}$ or $E_{2,2}=\R$,  $G_1=E_{2,1}\oplus E_{2,2}\oplus G_2$ for some $G_2$
with $\dim G_2=\infty$, and
$$E_{2,1}\oplus G_2=\bigcap^{2}_{j=1} \textrm{Ker}g^{1}_{j}\cap G_1 \subseteq \bigcap^{2}_{j=1}\textrm{Ker}
g^{1}_{j}.$$ Inductively, we can write
\begin{equation}\label{decomposition in the case F = Rm}
E=E_1\oplus (E_{2,1}\oplus E_{2,2})\oplus\cdots\oplus (E_{n-1,1}\oplus E_{n-1,2})\oplus(E_{n,1}\oplus E_{n,2}\oplus
G_n),
\end{equation}
where $E_1,E_{2,1}\dots, E_{n,1}=\R^m$, $E_{2,2},\dots, E_{n,2}$ are subspaces of dimension $0$ or $1$,
$$E_{n,2}\oplus G_n\subseteq \bigcap^{n}_{j=1}\textrm{Ker} g^{1}_{j},$$ and $$G_k=(E_{k+1,1}\oplus E_{k+1,2}\oplus
G_{k+1})$$ for every $k=1,\dots,n$.

Now, for each $n\in\N$, we define a continuous linear surjection $S_n:E\to F$ by setting it to be $0$ on all the
factors of the decomposition \eqref{decomposition in the case F = Rm} except on $E_{n,1}$, which is mapped onto
$F=\R^m$, and we construct our approximating function $\varphi$ exactly as in the proof of Theorem 1.6. At this
point, we only need to show the following variant of Claim \ref{Cvarphi is locally contained in Ax} (in which $G_n$
replaces the subspace $E_{n,2}$ of the previous proof).

\begin{claim}\label{Cvarphi is locally contained in Ax case 2}
If $n_x\geq 2$ then $C_{\varphi}\cap V_{x}$ is contained in the set
$$
A_{x}:= \left\{ y\in S^{+} :  G_n\subset \textrm{Ker}\, L_{y}\right\}.
$$
Recall that $L_y=Ds(u_y)$, where $s(u)=\sqrt{1-\|u\|^2}$, $d(u)=(u,s(u))$, and $u_y=d^{-1}(y)$.
\end{claim}
\begin{proof}
Let us see that, if $y\in V_x\setminus A_{x}$ then $D\varphi(y):T_yS^+\to F$ is surjective, that is, for every
$w\in F$ there exists $v\in T_{y}S^{+}$ such that $D\varphi(y)(v)=w$. Let $m=m_y$ be the largest number such that
$\psi_m(y)\neq 0$. Recall that $m\leq n$.  Since the operator
$$
S_{m}:E=E_{1,1}(\oplus E_{2,1}\oplus E_{2,2})\oplus ... \oplus (E_{m-1,1}\oplus E_{m-1,2})\oplus( E_{m,1}\oplus
E_{m,2}\oplus  G_m)\to F
$$
is surjective and equal to zero on all the factors of the decomposition \eqref{decomposition in the case F = Rm}
except on $E_{m,1}$ (which is mapped onto $F$), we may find $u_{m,1}\in E_{m,1}$ so that
$$
S_{m}(u_{m,1})=4\varepsilon(y_m)^{-1}\psi_{m}(y)^{-1}\|S_{m}\| \, w.
$$
Now, since $y\notin A_x$ there exists $e_n\in G_n\setminus \textrm{Ker }L_y$. If we set
$$
t_{0}:=-\frac{L_{y}(u_{m,1})}{L_{y}( e_n)},
$$ then the vector
$$
u:= u_{m,1}+t_0 e_n,
$$
satisfies that
$$
L_{y}(u)=0.
$$
But recall that, for every $k\le n$, we have  $E_{k,1}\oplus  G_k\subset \bigcap_{j=1}^{k}\textrm{Ker}(g_{j}^1)$;
in particular, $E_{m,1}\oplus G_m\subset \bigcap_{j=1}^{m}\textrm{Ker}(g_{j}^1)$ because $m\leq n$. Hence,
$g_j^1(u)=0$ for every $1\le j\le m$. It follows that $$\sum_{j=1}^{m}\left(g_{j}^{1}(u)+ g_{j}^{2}L_{y}(u)\right)
\alpha_{m, j}(y)=0.$$

The rest of the operators $S_1, ..., S_{m-1}$ are zero on $E_{m, 1}\oplus  G_n$, so we have
$$
S_{j}(u)=0 \textrm{ for every } j=1, ..., m-1,
$$
and since $S_m$ is zero on $G_n\subset G_m$ we also have $S_{m}(t_{0}e_n)=0$. Therefore, by combining these
equalities with equation \ref{expression for the derivative of varphi reflexive case}, we obtain that
$$
D\varphi(y)(u, L_{y}(u))=w,
$$
and the proof of the claim is complete.
\end{proof}

Then we also have the following.
\begin{lemma}\label{Ax is a good graph case 2}
For $x\in S^+$ with $n:=n_x\ge2$, the set $A_x$ of Claim \ref{Cvarphi is locally contained in Ax case 2} is of the
form
$$
A_x=d(G(f_x)\cap B_E),
$$
where $G(f_x)$ is the graph of a continuous mapping $f_x:E_{1,1}\oplus(E_{2,1}\oplus E_{2,2})\oplus \cdots \oplus
(E_{n,1}\oplus E_{n,2}) \to G_n$.
\end{lemma}
\begin{proof}
Repeat the proof Lemma \ref{Ax is a good graph}, just replacing $E_{n,2}$ with $G_n$.
\end{proof}

Let $Z:=d^{-1}(C_{\varphi})\subset B_E$. According to  Lemma \ref{Ax is a good graph case 2} applied to $x=d(z)\in
S^+$, $n=n_x$, and $V_x$, for a neighborhood $U_z:=d^{-1}(V_x)$ of $z$,  we have $Z\cap U_z\subseteq G(f_x)$, where
$$G(f_x)=\{u=(w,v)\in G'_{n}\oplus G_{n}=E:\,  v=f_x(w)\},$$
with $G_{n}'$ denoting $E_1\oplus (E_{2,1}\oplus E_{2,2})\oplus\cdots\oplus (E_{n-1,1}\oplus
E_{n-1,2})\oplus(E_{n,1}\oplus E_{n,2})$. Since $G_n$ is infinite-dimensional, we may use Theorem \ref{final
extractibility theorem rough version general form}, and the rest of the proof goes exactly as in Case 1.  \qed 

\medskip

\begin{rem}
{\em Observe that in the infinite-dimensional case we could have asked $A_x$ to be
$$A_x:=\{y\in S^+:\,E_{n,3}\subset Ker L_y\} ,$$
using $E_{n,3}$ instead of $E_{n,2}$ and requiring that $E_{n,1}\oplus E_{n,3}\subseteq \bigcap^{n}_{j=1}Ker
g^{1}_{j}$.}
\end{rem}

\medskip

\section{Proof of Theorem \ref{Main theorem for spaces with unconditional bases}}

First of all let us note that since $E$ admits a $C^1$ equivalent norm, $E$ cannot contain a closed subspace
isomorphic to $\ell_1$. Furthermore, as noted following the statement of Theorem \ref{Main theorem for spaces with
unconditional bases}, condition $(2)$ implies that the basis $\{e_{n}\}_{n\in\N}$ is unconditional, and
therefore by \cite[Theorem 1.c.9]{LindenstraussTzafriri}, is {\it shrinking}, that is, we have that
$E^{*}=\overline{\textrm{span}}\{e_{n}^{*} : n\in\N\}$, where $\{e_{n}^{*}\}_{n\in\N}$ are the biorthogonal
functionals associated to $\{e_n\}_{n\in\N}$ (in particular, $E^*$ is separable).

We keep using the notations $Y=E\times\R$ and $S^{+}$ from the proof of Theorem \ref{Main theorem for reflexive
spaces}. As in the case of Theorem \ref{Main theorem for reflexive spaces}, it will be enough to prove Theorem
\ref{Main theorem for spaces with unconditional bases} with $S^{+}$ in place of $E$. We define $e_{0}=(0,1)\in Y$,
$e_{0}^{*}:Y=E\times \R\to\R$ by
$$
e_{0}^{*}(u,t)=t,
$$
and by slightly abusing notation we identify $e_{n}\in E$ to $(e_{n}, 0)\in Y$ and also extend the $e_{n}^{*}\in
E^{*}$ to $e_{n}^{*}:Y=E\times\R\to\R$ by
$$
e_{n}^{*}(u,t)=e_{n}^{*}(u)=u_{n} \,\,\, \textrm{ for all } \,\,\, u=\sum_{j=1}^{\infty}u_j e_j, \in E, \, t\in\R.
$$
Then we may we consider $\{e_{n}\}_{n\in\N}\cup\{e_{0}\}$ as a basis of $E\times \R=Y$ with associated coordinate
functionals $\{e_{n}^{*}\}_{n\in\N}\cup\{e_{0}^{*}\}$, and we have that this basis is also shrinking.

Next we are going to construct a partition of unity in $S^{+}$, quite similar but not identical to that of the
proof of Theorem \ref{Main theorem for reflexive spaces}

Since the norm  $|\cdot|$ is locally uniformly convex we can find, for every $x\in S^+$, open slices $R_x=\{y\in S:
\ f_x(y)>\delta_x\}\subset S^+$  and $P_x=\{y\in S: \ f_x(y)>\delta_x^4\}\subset S^+$, where $f_x\in Y^{*}$,
$0<\delta_x<1$,  and $|f_x|^{*}=1=f_x(x)$, so that the oscillation of the functions $f$ and $\varepsilon$ on every
$P_x$ is less than $\varepsilon(x)/16$. We also assume, with no loss of generality, that $\textrm{dist}(P_x,\,
E\times\{0\}\,)>0$.

Since $Y$ is separable we can select a countable subfamily of $\{R_x\}_{x\in S^+}$, which covers $S^+$.  Let us
denote this countable subfamily by $\{R_n\}_n$, where $R_n=R_{x_n}=\{y\in S: f_n(y)>\delta_n\}$ and $f_{n}(x_n)=1$.
Recall that the oscillation of the functions  $f$ and $\varepsilon$ on every $ P_n=P_{x_n}=\{y\in S:
f_n(y)>\delta_n^4\}$ is less than $\varepsilon(x_n)/16$, and this implies that
$$
\frac{15}{16}\, \varepsilon(x_n)\leq \varepsilon(x)\leq \frac{17}{16}\,\varepsilon(x_n)\ \ \text{and}\ \
\|f(x)-f(y)\|\leq\frac{\varepsilon(x_n)}{16}
$$
for every $x,y\in P_n$. Note that $\{P_n\}_{n\in\N}$ is an open cover of $S^{+}$.

\noindent $\bold{\bullet}$ {\bf For $\bold{k=1}$}, since $\textrm{span}\{e_{n}^{*} : n\in\N\}$ is dense in $E^{*}$,
we may find numbers $N_1\in\N$, $\epsilon_1, \gamma_1\in (0, 1)$ with $\epsilon_1>\gamma_1$, and $\beta_{1,0}, ...,
\beta_{1, N_1}\in\R$ with $\beta_{1,0}>0$ so that the functional $g_1$ defined by
$$
g_{1}:=\sum_{j=0}^{N_1}\beta_{1, j}e_{j}^{*}
$$
has norm $1$ and satisfies
\begin{equation*}
\{x\in S:f_1(x)>\delta_1^2\}\subset\{x\in S: g_1(x)>{\epsilon_1}\} \subset\{x\in S:
g_1(x)>{\gamma_1}\}\subset\{x\in S:f_1(x)>\delta_1^3\}.
\end{equation*}

Let us define
\begin{align*} h_1&:\,S^+\longrightarrow \mathbb R \\ h_1&=\theta_1(g_1), \end{align*}
 where $\theta_1:\R\to [0,1]$ is a
$C^\infty$ function satisfying
\begin{align*}
\theta_1(t)&=0 \ \text{ if and only if}\  t\leq\gamma_1 \\
\theta_1(t)&=1 \ \text{ if and only if}\ t\geq \epsilon_1.
\end{align*}
Note that the interior of the support of $h_1$ is the open set $U_1:=\{x\in S^{+} : g_1(x)>\gamma_1\}$.

\noindent $\bold{\bullet}$ {\bf  For $\bold{k=2}$}. We may again use the density of $\textrm{span}\{e_{n}^{*} :
n\in\N\}$ in $E^{*}$, in order to find numbers $N_2\in\N$, $\gamma_2, \epsilon_2\in (0,1)$ with
$\gamma_2<\epsilon_2$, and $\beta_{2,0}, ..., \beta_{2, N_2}\in\R$ so that the linear functional
$$
g_{2}:=\sum_{j=0}^{N_2}\beta_{2, j}e_{j}^{*}
$$
has norm $1$ and satisfies
\begin{equation*}
\{x\in S:f_2(x)>\delta_2^2\}\subset\{x\in S: g_2(x)>{\epsilon_2}\}\subset\{x\in S: g_2(x)>{\gamma_2}\}\subset\{x\in
S:f_2(x)>\delta_2^3\}.
\end{equation*}
We may assume without loss of generality that $N_1\leq N_2$ (otherwise we may set $\beta_{2,j}=0$ for $N_2<j\leq
N_1$ and take a new $N_2$ equal to $N_1$).

Now we define
\begin{align*}h_2&:\,S^+\longrightarrow \mathbb R\\
h_2&=\theta_2(g_2)\left( 1-\theta_1(g_1)\right),
\end{align*}
where $\theta_2:\R\to [0,1]$ is a $C^\infty$ function satisfying:
\begin{align*}
\theta_2(t)&=0 \ \text{ if and only if }\  t\leq\gamma_2 \\
\theta_2(t)&=1 \ \text{ if and only if }\ t\geq \epsilon_2.
\end{align*}
Notice that the interior of the support of $h_2$ is the open set
$$
U_2=\{x\in S^+: g_1(x)< \epsilon_{1}\,,\ g_2(x)>{\gamma_{2}}\}.
$$

\noindent $\bold{\bullet}$ {\bf For $\bold{k=3}$,} By density of $\textrm{span}\{e_{n}^{*} : n\in\N\}$ in $E^{*}$
we may pick numbers $N_3\in\N$, $\gamma_3, \epsilon_3\in (0,1)$ with $\epsilon_3>\gamma_3$, and $\beta_{3,0}, ...,
\beta_{3, N_3}\in\R$ so that, for
$$
g_{3}:=\sum_{j=0}^{N_3}\beta_{3, j}e_{j}^{*}
$$
we have that $g_{3}\in S^{*}$ and
\begin{equation*}
\{x\in S:\ f_3(x)>\delta_3^2\}\subset\{x\in S:\ g_3(x)>{\epsilon_3}\}\subset\{x\in S:\
g_3(x)>{\gamma_3}\}\subset\{x\in S: f_3(x)>\delta_3^3\}
\end{equation*}
Again we may assume without loss of generality that $N_2\leq N_3$.

We define
\begin{align*}
h_3&:\, S^+\longrightarrow \mathbb R\\h_3&=\theta_3(g_3)\prod_{j=1}^{2}\left(1-\theta_j(g_j)\right),
\end{align*}
where $\theta_3:\R\to [0,1]$ is a $C^\infty$ function satisfying
\begin{align*}
\theta_3(t)&=0 \ \ \text{ if and only if } t \leq {\gamma_3}\\
\theta_3(1)&=1 \ \ \text{ if and only if } t \geq \epsilon_3.
\end{align*}
Clearly the interior of the support of $h_3$ is the set
\begin{equation*}
U_3=\{x\in S^+:\ g_1(x)<\epsilon_{1}\,, \ g_2(x)<\epsilon_{2} \ \text{ and } \ g_3(x)>\gamma_3\}.
\end{equation*}

We continue this process by induction.

\noindent $\bold{\bullet}$ Assume that, in the steps {\bf $j=2, ..., k$}, with $k\geq 2$, we have selected points
$y_j\in S^+$, positive integers $N_1\leq N_2 \leq ... \leq N_k$, and constants $\gamma_j, \epsilon_j\in (0,1)$,
$\beta_{j,i}\in\R$ so that the functionals
$$
g_{j}:=\sum_{i=0}^{N_j}\beta_{j,i}e_{i}^{*}
$$
belong to $S^{*}$ and satisfy
 \begin{align}
\{x\in S:\ f_j(x)>\delta_j^{2}\}\subset\{x\in S:\ g_j(x)>{\epsilon_j}\}\subset\{x\in S:\ g_j(x)>{\gamma_j}\}\subset
\{x\in S: \ f_j(x)>\delta_j^4\},
 \end{align}
for all $j=2,...,k$. Assume also that we have defined numbers $\gamma_j$ and functions
$$
h_j=\theta_j(g_j) \prod_{i<j}\left(1-\theta_i(g_i)\right),
$$
where $\theta_j:\R\to [0,1]$ are $C^\infty$ functions satisfying
\begin{align*}
\theta_j(t)&=0 \ \ \text{ if and only } t \leq {\gamma_j}\\
\theta_j(t)&=1 \ \ \text{ if and only } t \geq \epsilon_j.
\end{align*}

The interior of the support of $h_j$ is the set
\begin{equation*}
U_j=\{x\in S^+:\
 g_1(x)<\epsilon_{1}\,,..., \, g_{j-1}(x)<\epsilon_{j-1} \  \text{ and } \
 g_j(x)>\gamma_j\}.
\end{equation*}
Then we may again use the density of $\textrm{span}\{e_{n}^{*} : n\in\N\}$ in $E^{*}$, in order to find a positive
integer $N_{k+1}\geq N_k$,  and constants $\gamma_{k+1}, \epsilon_{k+1}\in (0,1)$, and $\beta_{k+1,0}, ...,
\beta_{k+1, N_{k+1}}\in\R$ so that, for
$$
g_{k+1}:=\sum_{j=0}^{N_{k+1}}\beta_{k+1, j}e_{j}^{*}
$$
we have that $g_{k+1}\in S^{*}$ and
\begin{eqnarray*}
 & & \{x\in S:\ f_{k+1}(x)>\delta_{k+1}^2\}
 \subset\{x\in S:\
g_{k+1}(x)>{\epsilon_{k+1}}\}\subset \\
& & \{x\in S:\ g_{k+1}(x)>{\gamma_{k+1}}\}\subset \{x\in S: \ f_{k+1}(x)>\delta_{k+1}^3\}.
\end{eqnarray*}

We now set
\begin{equation}U_{k+1}:=\{x\in S^+:\
 g_1(x)<\epsilon_{1}\,,..., \, g_{k}(x)<\epsilon_{k} \  \text{ and }\
 g_{k+1}(x)>\gamma_{k+1}\},
\end{equation}
and define
\begin{align*}
h_{k+1} & : S^+\longrightarrow \mathbb R
\\
h_{k+1} & =\theta_{k+1}(g_{k+1})\prod_{j<k+1}\left(1-\theta_j (g_j)\right).
\end{align*}
where \ $\theta_{k+1}:\R\to [0,1]$ is a $C^\infty$ function such that
\begin{align*}
\theta_{k+1}(t)&=0 \ \ \text{ if and only if } t \leq {\gamma_{k+1}}\\
\theta_{k+1}(t)&=1 \ \ \text{ if and only if } t \geq {\epsilon_{k+1}}.
\end{align*}
Clearly the interior of the support of $h_{k+1}$ is the set $U_{k+1}$.

Thus a sequence $\{h_{n}\}_{n\in\N}$ of $C^{1}$ smooth functions with the above properties is well defined by
induction.

As in the proof of Theorem \ref{Main theorem for reflexive spaces} it is not difficult to check that the family
$\{U_k\}_{k\in \mathbb N}$  is a locally finite open covering of $S^+$ refining $\{P_{k}\}_{k\in\N}$. Therefore the
functions
$$
\psi_{n}:=\frac{h_n}{\sum_{k=1}^{\infty}h_k}, \,\,\, n\in\N,
$$
define a $C^{1}$ partition of unity in $S^{+}$ subordinate to $\{P_{k}\}_{k\in\N}$.

We will also need the following fact.
\begin{claim}\label{Dh_j vanishes where h_j=0 theorem 1.7}
For every $k\in \N$ and every $y\in S^+$, we have
$$
h_k(y)=0 \Longrightarrow Dh_k(y)=0.
$$
\end{claim}
\begin{proof}
Proceed as in the proof of Claim \ref{Dh_j vanishes where h_j=0}.
\end{proof}
Again we have that for every $x\in S^{+}$ there exist a number $n=n_x$ and an open neighborhood $V_x$ of $x$ in
$S^{+}$ such that $\psi_k(y)=0$ for every $k>n$ and every $y\in V_x$. Let us also call $m=m_y$ the largest $j$ such
that $h_j(y)\neq 0$; this is also the largest $j$ for which $\psi_j(y)\neq 0$. Thus, for every $y\in V_x$, we have
$$
m_y=\max\{j: y\in\psi^{-1}_j((0,1])\}\leq n_x.$$ The derivatives of the functions $\psi_n$ can be calculated as in
the proof of Theorem \ref{Main theorem for reflexive spaces}. We have
$$
D\psi_k(y)(v)=\sum_{j=1}^{n}\lambda_{k,j}(y) \, g_{j}(v)=\sum_{j=1}^{m_y}\lambda_{k,j}(y) \, g_{j}(v)
$$
for each $v\in T_{y}S^{+}$, where the functions $\lambda_{k,j}:V_x\to \R$ are of class $C^1$.

The following lemma summarizes the properties of the partition of unity $\{\psi_n\}_{n\in\N}$ which will be most
useful to us.

\begin{lemma}\label{properties of the partition of unity on the upper sphere}
Given two continuous functions $f:S^{+}\to F$ and $\varepsilon:S^{+}\to (0, \infty)$, there exists a collection of
norm-one linear functionals $\{g_k\}_{k\in\N}\subset Y^{*}$ of the form
$$
g_{k}=\sum_{j=0}^{N_k}\beta_{k,j}e_{j}^{*},
$$
where $N_1\leq N_2 \leq N_3\leq ...$, an open covering $\{P_n\}_{n\in\N}$ of $S^{+}$, and a $C^1$ partition of
unity $\{\psi_{n}\}_{n\in\N}$ in $S^{+}$ such that:
\begin{enumerate}
\item $\{\psi_n\}_{n\in\N}$ is subordinate to $\{P_n\}_{n\in\N}$.

\item For every $x\in S^{+}$ there exist a neighborhood $V_x$ of $x$ in $S^{+}$ and a number $n=n_{x}\in\N$ such that $\psi_m=0$ on $V_x$ for all $m>n$, and the derivatives of the functions $\psi_1, ..., \psi_{n}$ on $V_x$ are of the form
$$
D\psi_{k}(y)(v)=\sum_{j=1}^{n}\lambda_{k, j}(y)\, g_{j}(v) =\sum_{j=1}^{m_y}\lambda_{k, j}(y)\, g_{j}(v),
$$
for $v\in T_{y}S^{+}$, the tangent hyperplane to $S^{+}$ at $y\in S^{+}\cap V_x$, where $m_y\leq n$ is the largest
number such that $\psi_{m_y}(y)\neq 0$. More precisely, if $L_y$ denotes the derivative of the function $u\mapsto
\sqrt{1-\|u\|^2}$ evaluated at the point $u_y$ such that $y=\left(u_y, \sqrt{1-\|u_y\|^{2}}\right)$, we have
$$
D\psi_k(y)(v)=D\psi_k(y)(u, L_{y}(u))= L_{y}(u)\mu_{k,0}(y)+\sum_{j=1}^{N_n}\mu_{k,j}(y)e_{j}^{*}(u)
=L_{y}(u)\mu_{k,0}(y)+\sum_{j=1}^{N_{m_y}}\mu_{k,j}(y)e_{j}^{*}(u)
$$
for every $k=1, ..., n$, and for every $v=(u, L_{y}(u))\in T_{y}S^{+}$, where the functions $\mu_{k, j}:V_x\to\R$
are of class $C^{1}$, $j=1, ..., n$.

\item For every $n\in\N$ there exist a point $y_n:=x_n\in P_n$ such that
$$
\frac{15}{16}\, \varepsilon(y_n)\leq \varepsilon(y)\leq \frac{17}{16}\,\varepsilon(y_n) \,\,\, \textrm{ for every }
y\in P_n, \,\,\, \textrm{ and } \,\,\, \sup_{x, y\in P_n}\|f(x)-f(y)\|\leq \frac{\varepsilon(y_n)}{16}.
$$
\end{enumerate}
\end{lemma}
Note also that the integer $n_x$ can be chosen as the first such $j$ so that $x\in \{y\in S^+:
g_{j}(y)>\epsilon_j\}$. Then $V_x$ can be chosen as $\{y\in S^+: g_{j}(y)>\epsilon_j\}$ and, hence,  we have
$R_j\subset V_x\subset P_j$ for such $V_x$.

\medskip

Now we are ready to start the construction of our approximating function $\varphi:S^{+}\to F$, which will be of the
form
$$
\varphi(x)=\sum_{n=1}^{\infty}\left( f(y_n) +T_{n}(x) \right) \psi_{n}(x),
$$
where the $y_n$ are the points given by condition $(3)$ of the preceding lemma, and the operators $T_{n}:Y\to F$
will be defined below. We have to distinguish two cases.

\medskip

{\bf Case 1: Assume that $F$ is infinite-dimensional.}

In order to define the operators $T_n$, we work with the infinite subset $\mathbb{P}$ of $\N$ given by assumption
$(3)$ of Theorem \ref{Main theorem for spaces with unconditional bases}, and we take a countable pairwise disjoint
family of infinite subsets of $\mathbb{P}$ which {\em goes to infinity}. More precisely, we write
$$\bigcup^{\infty}_{n= 1} I_n\subseteq \mathbb{P},$$ in such a way
that:
\begin{enumerate}
\item $I_n :=\left\lbrace n_i:\,i\in\N\right\rbrace $ is infinite for each $n\in \N$;
\item $I_n\cap I_m=\emptyset$ for all $n\neq m$; and
\item $\left\lbrace 1,\dots, N_n\right\rbrace \cap I_n=\emptyset$ for all $n\in\N$.
\end{enumerate}
Here $\{N_n\}_{n\in\N}$ is the non-decreasing sequence of positive integers that appears in the construction of the
functionals $g_n$ of Lemma \ref{properties of the partition of unity on the upper sphere}.

Now, by using assumption $(3)$ of the statement, we can find, for each number $n\in\N$, a linear continuous
surjection $S_n:E\to F$ of the form
$$
S_n=A_n\circ P_n,
$$
where $A_{n}$ is a bounded linear operator from $\overline{\textrm{span}}\{ e_{n_k} : k\in\N\}=
\overline{\textrm{span}}\{ e_{m} : m\in\ I_n\}$ onto $F$, and $P_{n}:E\to \overline{\textrm{span}}\{ e_{n_k} :
k\in\N\}$ is the natural projection associated to the unconditional basis $\{e_j\}_{j\in\N}$.

Now we finally define $T_{n}:Y\to F$ by
$$
T_{n}(u,t)= \frac{\varepsilon(y_n)}{4\|S_{n}\|}\, S_{n}(u),
$$
and $\varphi:S^{+}\to F$ by
$$
\varphi(x)=\sum_{n=1}^{\infty}\left( f(y_n) +T_{n}(x) \right) \psi_{n}(x).
$$
It is clear that $\varphi$ is well defined and of class $C^1$.

\begin{claim}
We have that $\|\varphi(x)-f(x)\|\leq \varepsilon(x)$ for every $x\in S^{+}$.
\end{claim}
\begin{proof}
This is shown exactly as in Claim \ref{varphi epsilon approximates f}.
\end{proof}

Let us now calculate the derivative of our function $\varphi$.  For every $v=(u, L_y(u))\in T_{y}S^{+}$ we have
that
\begin{equation}\label{formula for Tn restricted to TyS}
T_{n}(v)=T_{n} (u, L_{y}(u))=\frac{\varepsilon(y_n)}{4\|S_n\|} \, S_{n}(u),
\end{equation}
and we have that $D({T_{n}}_{|_{S^{+}}})(y)$ is the restriction of $DT_n(y)=T_n$ to $T_{y}S^{+}$, that is to say,
if $v=(u, L_y(u))\in T_{y}S^{+}$ then
\begin{equation}\label{formula for DTn restricted to TyS}
DT_{n}(y)(u, L_{y}(u))=\frac{\varepsilon(y_n)}{4\|S_n\|} \, S_{n}(u).
\end{equation}

We can now compute the derivative of $\varphi$ on $S^{+}$. Recall that, by condition $(2)$ of Lemma \ref{properties
of the partition of unity on the upper sphere}, for every $x\in S^+$ there is a neighborhood $V_x$ of $x$ in
$S^{+}$ and a number $n=n_{x}\in\N$ such that
$$
\varphi(y)=\sum_{j=1}^{n}\left( f(y_j) +T_{j}(y) \right) \psi_{j}(y)
$$
for every $y\in V_x$. Fix $y\in V_x$ and recall that for $m=m_y$ (the largest number $j$ for which
$\psi_j(y)\neq 0$), we have
$$
\varphi(y)=\sum_{j=1}^{m}\left( f(y_j) +T_{j}(y) \right) \psi_{j}(y).
$$

By using \eqref{formula for DTn restricted to TyS} and the expression for $D\psi_j(y)$ given in Lemma
\ref{properties of the partition of unity on the upper sphere}, we see that
\begin{align}\label{expression for the derivative of varphi}
D\varphi(y)(u, L_{y}(u))&= \left(\sum_{j=1}^{n}\psi_{j}(y)\, \frac{\varepsilon(y_j)}{4\|S_j\|}\, S_{j}(u)\right) +
L_{y}(u) \alpha_{N_n,0}(y)+ \sum_{j=1}^{N_n}\alpha_{N_{n}, j}(y)e_{j}^{*}(u)= \nonumber \\
& =\left(\sum_{j=1}^{m}\psi_{j}(y)\, \frac{\varepsilon(y_j)}{4\|S_j\|}\, S_{j}(u)\right) + L_{y}(u)
\alpha_{N_m,0}(y)+ \sum_{j=1}^{N_m}\alpha_{N_{m}, j}(y)e_{j}^{*}(u)
\end{align}
for every $(u, L_{y}(u))\in T_{y}S^{+}$, $y\in V_x$, where the functions $\alpha_{N_n,j}:V_x\subset S^{+}\to F$ are
of class $C^1$ (because we have $\alpha_{N_n,j}(y)=\sum_{i=1}^n\left(f(y_i)+T_i(y)\right)\mu_{i,j}(y)$ and
$\alpha_{N_n,0}(y)=\sum_{i=1}^n\left(f(y_i)+T_i(y)\right)\mu_{i,0}(y)$, where $\mu_{i,j}$ are as in Lemma
\ref{properties of the partition of unity on the upper sphere}).

\medskip

Let us now prove that the critical set of $\varphi$ is relatively small.
\begin{lemma}\label{the critical set of varphi is small}
The set $C_{\varphi}:=\{x\in S^{+} \, : \, D\varphi(x) \textrm{ is not surjective}\}$ is of the form
$$
C_{\varphi}=\{ \left(w, \sqrt{1- \|w\|^2}\right) : w\in A\},
$$
where $A\subset E$ is a relatively closed subset of the open unit ball of $E$ that is locally contained in a
complemented subspace of infinite codimension in $E$.
\end{lemma}
\begin{proof}
Observe that if $n=n_x=1$ then $\varphi(y)=f(y_1)+ T_1(y)$ for every $y\in V_x$, and because $T_1$ is surjective
$\varphi$ does not have any critical point in $V_x$. Now let us assume that $n=n_x\geq 2$. Let $y=( w,
\sqrt{1-\|w\|^2})$ be  point of $V_x$. Let $m=m_y$ be the largest number such that $\psi_m(y)\neq 0$. Recall that
$m\leq n$. We only need to show that if
$$
w \notin \overline{\textrm{span}}\left\lbrace e_j:\, j\in\mathbb{P} \text{ or } j=1, \dots, N_n \right\rbrace
$$
then for every $v\in F$ there exists $u\in E$ such that
$$
D\varphi(w, \sqrt{1-\|w\|^2})(u, L_{y}(u))=v,
$$
since this will mean that the set
$$
A:=\{w\in E: (w, \sqrt{1-\|w\|^2})\in C_{\varphi}\}
$$
will be locally contained in subspaces of the form $ \overline{\textrm{span}}\left\lbrace e_i:\, i \in\mathbb{P}
\text{ or } i=1, \dots, N_n \right\rbrace, $ which are complemented, and of infinite codimension, in $E$.

We will need to use the following.
\begin{fact}\label{fact about suppression basis}
For every $w=\sum_{j=1}^{\infty}w_j e_j\in E\setminus\{0\}$ and every $j_0\in\N$ we have that
$$
w_{j_0}\neq 0 \implies \langle J(w), e_{j_{0}}\rangle\neq 0,
$$
where $J(w)$ denotes $D\|\cdot\|(w)$, and $\langle J(w), u\rangle :=J(w)(u)$.
\end{fact}
\begin{proof}
If $w_{j_0}\neq 0$ then, by assumption $(2)$ of the statement of Theorem \ref{Main theorem for spaces with
unconditional bases}, we have that
$$
\|\sum_{j=1, \, j\neq j_0}^{\infty}w_j e_j\|\leq \|\sum_{j=1}^{\infty}w_j e_j\|.
$$
This means that the convex function $\theta :\R\to \R$ defined by
$$
\theta (t)= \| w+ te_{j_0}\|
$$
has a minimum at $t= -w_{j_0}$. On the other hand, if we had $\langle J(w), e_{j_0}\rangle =0$, then the same
function $\theta$ would have another minimum at the point $t=0$.  But since $\|\cdot\|$ is strictly convex the
function $\theta$ can only attain its minimum at a unique point. Therefore we must have $\langle J(w),
e_{j_0}\rangle \neq 0$.
\end{proof}

So let us pick a point $ w\in E\setminus \overline{\textrm{span}}\left\lbrace e_j:\, j\in\mathbb{P} \text{ or }
j=1, \dots, N_n \right\rbrace$ and a vector $v\in F$, and let us construct a vector $u\in E$ such that $D\varphi(w,
\sqrt{1-\|w\|^2})(u, L_{y}(u))=v$, where $y=(w, \sqrt{1-\|w\|^2})$. By assumption, there exists $j_0\in\N$, $j_0\in
\N\setminus\mathbb{P}$, such that  $j_{0}> N_n\geq N_m$ and $w_{j_0}\neq 0$. According to the fact just shown, we
have $\langle J(w), e_{j_0}\rangle \neq 0$. Now, since  $S_m$ is surjective and $\psi_m(y)\neq 0$, we may find a
sequence  $(u_{m_i})_{i\in\N}$ (indexed by the subsequence $(m_i)_{i\in\N}$ defined by $I_m$) such that
$$
\psi_{m}(y)\frac{\varepsilon (y_m)}{4\|S_{m}\|} \, S_{m}\left(\sum_{i=1}^{\infty} u_{m_i}e_{m_i}\right)=v.
$$
Note that $j_0\notin I_{m}=\{m_i : i\in\N\}$, because $j_0\in\N\setminus\mathbb{P}$ and $I_m\subset \mathbb{P}$.
Then we can set
$$
u_{j_0}:= - \frac{ \langle J(w), \, \sum_{i=1}^{\infty} u_{m_i} e_{m_i}\rangle }{ \langle J(w), e_{j_0}\rangle},
$$
so that we have
$$
\langle J(w), \, u_{j_0}e_{j_0} + \sum_{i=1}^{\infty} u_{m_i}e_{m_i}\rangle =0,
$$
which bearing in mind that
$$
L_{y}= - \frac{ \|w\|}{\sqrt{1- \|w\|^{2}}} \langle J(w), \, \cdot\rangle
$$
also implies that
$$
L_{y} \left( u_{j_0}e_{j_0} + \sum_{i=1}^{\infty}u_{m_i}e_{m_i}\right)=0.
$$
So if we set $u_j=0$ for all $j\notin I_{m}\cup\{j_0\}$ and we define $$u:=\sum_{j=1}^{\infty}u_j e_j$$ then we
have that
$$
\psi_{m}(y)\frac{\varepsilon (y_m)}{4\|S_{m}\|} \, S_{m}\left(u\right)=v, \,\,\, L_{y}(u) =0, \,\,\,
\sum_{j=1}^{N_m}\alpha_{N_{m}, j}(y)e_{j}^{*}(u)=0, \,\,\, \textrm{ and also } S_{j}(u)=0 \textrm{ for } j<m,
$$
because $j_0>N_n\geq N_m$, $I_{m}\cap \{1, 2, ..., N_m\}=\emptyset$, and the sets $I_{j}$ are pairwise disjoint. In
view of \eqref{expression for the derivative of varphi} these equalities imply that $ D\varphi(y)(u)= v$.
\end{proof}

Now, according to Lemma \ref{the critical set of varphi is small} and Theorem \ref{final extractibility theorem
rough version general form}, we can extract the set $C_{\varphi}$, since it is $C^1$ diffeomorphic (via the
projection of the graph $S^{+}$ of the function $w\mapsto \sqrt{1- \|w\|^{2}}$ onto the open unit ball of $E$) to a
subset which can be extracted. Therefore we can finish the proof of Theorem \ref{Main theorem for spaces with
unconditional bases} exactly as we did with Theorem \ref{Main theorem for reflexive spaces}.

\medskip

{\bf Case 2: Assume that $F=\R^m$.} The proof is almost identical, but with the following  important difference:
now the set $\mathbb{P}$ is by definition the set of {\em even} positive integers, and the sets $I_{n}$ are {\em
finite} subsets of $\mathbb{P}$ such that:
\begin{enumerate}
\item $\sharp I_n =m$ for each $n\in \N$;
\item $I_n\cap I_j=\emptyset$ for all $n\neq j$; and
\item $\left\lbrace 1,\dots, N_n\right\rbrace \cap I_n=\emptyset$ for all $n\in\N$.
\end{enumerate}
Here $\{N_n\}_{n\in\N}$ is the non-decreasing sequence of positive integers that appears in the construction of the
functionals $g_n$ of Lemma \ref{properties of the partition of unity on the upper sphere}.

Of course in this case we can always find linear surjections $A_n :\textrm{span}\{ e_i \, : \, i\in I_n\}\to \R^m$.
\qed

\medskip

\section{Technical versions of Theorems 1.6 and 1.7, examples, and remarks}

In this section we will give some examples, make some remarks and establish more technical variants of our results
which follow by the same method of proof. We will also prove Proposition \ref{C1 approximation is enough}.
\medskip

The proof of Theorem \ref{Main theorem for reflexive spaces} can be easily adjusted to obtain more general results
with more complicated statements. Namely, the following two results are true.

\begin{theorem}\label{result for recursively reflexively composite spaces}
Let $E$ and $F$ be Banach spaces. Assume that:
\begin{enumerate}
\item $E$ is infinite-dimensional, with a separable dual $E^{*}$.
\item There exist three sequences $\{E_{n,1}\}_{n\geq 1}$, $\{E_{n,2}\}_{n\geq 1}$, $\{E_{n,3}\}_{n\geq 2}$ of subspaces
of $E$ such that
$$E=E_{1,1}\oplus E_{1,2},\,\,\, E_{1,2}=
 (E_{2,1}\oplus E_{2,2})\oplus ... \oplus (E_{n,1}\oplus E_{n,2}\oplus E_{n,3}), \,\,\,
E_{n,3}=E_{n+1,1}\oplus E_{n+1, 2}\oplus E_{n+1,3},
$$
with either $E_{n,3}$ infinite-dimensional and reflexive and $\dim E_{n,2}\geq 1$, or else
$E_{n,2}$ infinite-dimensional and reflexive for all $n\geq 2$. Suppose also that there exists a bounded linear
operator from $E_{n,1}$ onto $F$ for every $n\in\N$.
\end{enumerate}
Then, for every continuous mapping  $f:E\to F$ and every continuous function $\varepsilon: E\to (0, \infty)$ there
exists a $C^{1}$ mapping $g:E\to F$ such that $\|f(x)-g(x)\|\leq\varepsilon(x)$ and $Dg(x):E\to F$ is a surjective
linear operator for every $x\in E$.
\end{theorem}
Observe that if $E_{n,3}$ is infinite-dimensional and reflexive, the spaces $E_{n,2}$
can be taken to be of dimension $1$ for every $n\in \N$. The proof is almost the same as that of Theorem \ref{Main
theorem for reflexive spaces}. Here,  up to finite-dimensional perturbations of the subspaces $E_{k,j}$, we can arrange that $E_{1,2}\subset Ker g^{1}_{1}$ and that $E_{n,1}\oplus E_{n,3}\subseteq \bigcap^{n}_{j=1}Ker g^{1}_{j}$, and we may set $A_x=\{y\in S^+:\,E_{n,3}\subset Ker L_y\}$.

\begin{theorem}\label{result for reflexive complemented spaces}
Let $E$, $X$, and $F$ be Banach spaces. Assume either that $E$ is infinite-dimensional, separable, and reflexive,
and $F$ is finite-dimensional, or that:
\begin{enumerate}
\item $E$ is infinite-dimensional, with a separable dual $E^{*}$.
\item There exists a decomposition of $E$,
$$E=G\oplus E_{1}\oplus X,$$
such that $G$ is infinite-dimensional and reflexive, and $E_{1}$ is isomorphic to $E$.
\item There exists a
bounded linear operator from $G\oplus X$ onto $F$ (equivalently, $F$ is a quotient of $G\oplus X$).
\end{enumerate}
Then, for every continuous mapping  $f:E\to F$ and every continuous function $\varepsilon: E\to (0, \infty)$ there
exists a $C^{1}$ mapping $g:E\to F$ such that $\|f(x)-g(x)\|\leq\varepsilon(x)$ and $Dg(x):E\to F$ is a surjective
linear operator for every $x\in E$.

\medskip
If, additionally, $X$ is isomorphic to $E$, then $F$ can be taken as a quotient of $E$.
\end{theorem}

Observe that each of these two results imply Theorem \ref{Main theorem for reflexive spaces}, with Theorem
\ref{result for recursively reflexively composite spaces} being the most general one.

For instance, Theorem \ref{result for reflexive complemented spaces} can be applied to the James space  $J$ and to
its dual $J^*$. Indeed, both spaces have separable dual. It is known that $J$ has many reflexive 
infinite-dimensional complemented subspaces $G$ \cite{CasLinLoh}. Since $J$ is prime \cite{Cas}, for
each such $G$, we can write $J=G\oplus J$ (for instance we have $J=l_2\oplus J$). Now, recalling the
fact that $J$ has a separable dual, apply Theorem \ref{result for reflexive complemented spaces} to $E=E_1=J$ and
$X=\{0\}$ to see that every continuous function $f:J=G\oplus J\to F$, where $F$ is a quotient of
$G$, can be uniformly approximated by $C^1$ smooth mappings without critical points.  Similar
arguments work for the dual of the James space $J^*$.
\medskip

It also follows that the conclusion of Theorem \ref{Main theorem for reflexive spaces} is true for {\em composite}
spaces of the form $c_0\oplus\ell_p$ or $c_0\oplus L^{p}$, $1<p<\infty$, with $k$ being the order of smoothness of
$\ell_p$ or $L^{p}$. More generally, if $E$ is any finite direct sum of the classical Banach spaces $c_0$, $\ell_p$
or $L^{p}$, $1<p<\infty$, and there is a bounded linear operator from $E$ onto $F$ then the conclusion of Theorem
\ref{Main result for classical Banach} is true with $k$ being the minimum of the orders of smoothness of the spaces
appearing in this decomposition of $E$.

\medskip

\begin{remark}\label{In the Hilbert case we directly get Cinfinity}
{\em Notice that that in the case that $E$ is a separable Hilbert space, we have that the function $w\mapsto
\|w\|^2$ is of class $C^{\infty}$, hence all the mappings appearing in the proof of Theorem \ref{Main theorem for
reflexive spaces} are of class $C^{\infty}$, and we directly obtain an approximating function $g$ of class
$C^{\infty}$ with no critical points.

In fact, in the Hilbertian case we do not need to use a partition of unity in the upper sphere $S^{+}$. We can
directly construct a partition of unity $\{\psi_n\}_{n\in\N}$ in $E$ subordinated to an open covering by open balls
with linearly independent centers $\{y_j\}$, as in \cite{AzagraCepedello}. Then, choosing an orthonormal basis
$\{e_j\}$ for which $span\{y_1,\dots,y_n\}=span\{e_1,\dots,e_n\}$ for every $n\in\N$, we define operators $T_n:E\to
F$ as in the proof of Theorem \ref{Main theorem for spaces with unconditional bases}, where $\mathbb{P}$ can be any
infinite subset of $\N$ such that $\N\setminus\mathbb{P}$ is also infinite. Then one can easily check that the
function
$$
\varphi(y)=\sum_{n=1}^{\infty}\left( f(y_n)+T_n(y-y_n)\right) \psi_{n}(y)
$$
approximates $f$ and the set $C_{\varphi}$ of its critical points is locally contained in a subspace of infinite
codimension in $E$, specifically in subspaces of the form $\overline{span}\{e_j:\,j\in \mathbb{P}
\;\text{or}\;j=1,\dots,n\}$. Then one can extract $C_{\varphi}$ by means of a $C^{\infty}$ diffeomorphism $h:E\to
E\setminus C_{\varphi}$ which is sufficiently close to the identity, and conclude that the function
$g:=\varphi\circ h$ approximates $f$ and has no critical points. }
\end{remark}

\medskip

The same proof as that of Theorem \ref{Main theorem for spaces with unconditional bases}, with some adjustments,
allows us to obtain a more general (and also more technical) result as follows.

\begin{thm}\label{Refinement of the theorem for spaces with unconditional bases}
Let $E$ be an infinite-dimensional Banach space, and $F$ be a Banach space such that:
\begin{enumerate}
\item $E$ has an equivalent norm $\|\cdot\|$ which is $C^1$ and locally uniformly
convex.
\item $E$ has a (normalized) Schauder basis $\{e_{n}\}_{n\in\N}$ which is shrinking.
\item There exists an infinite subset $\mathbb{I}$ of $\N$ such that
the subspace $\overline{\textrm{span}}\{e_j : j\in \N\setminus\mathbb{I}\}$ is complemented in $E$, and for every
$x=\sum_{j=}^{\infty}x_j e_j$ and every $j_0\in\mathbb{I}$ we have that
$$
\left\|\sum_{j\in\N, \, j\neq j_0} x_j e_j\right\|\leq \left\|\sum_{j\in\N}x_j e_j\right\|.
$$
\item In the case that $F$ is infinite-dimensional, there exists an infinite subset $\mathbb{P}$ of $\N$ such that $\mathbb{I}\setminus \mathbb{P}$ is infinite and for every infinite subset $J$ of $\mathbb{P}$ the subspace $E'=\overline{\textrm{span}}\{e_j : j\in J\cup (\N\setminus \mathbb{I})\}$ is complemented in
$E$, and there exists a linear bounded operator from $E'$  onto $F$.
\end{enumerate}
Then, for every continuous mapping  $f:E\to F$ and for every continuous function $\varepsilon: E\to (0, \infty)$
there exists a $C^{1}$ mapping $g:E\to F$ such that $\|f(x)-g(x)\|\leq\varepsilon(x)$ and $Dg(x):E\to F$ is a
surjective linear operator for every $x\in E$.
\end{thm}
\begin{proof}
The only important difference with the proof of Theorem \ref{Main theorem for spaces with unconditional bases} is
that now we have to use an analogue of Fact \ref{fact about suppression basis} which is true if we just pick
$j_0\in\mathbb{I}$. Therefore, if we take $w=\sum^{\infty}_{j=1}w_je_j\notin
\overline{span}\{e_j:\,j\in\mathbb{P} \cup (\mathbb{N}\setminus\mathbb{I} )\;\text{or}\;j=1,\dots,
N_n\}$, since $\mathbb{I}\setminus \mathbb{P}$ is infinite, there will exist $j_0\in\mathbb{I}$, $j_0>N_n$ such
that $w_{j_0}\neq 0$ and thus $\langle J(w),e_{j_0}\rangle\neq 0$. The operators $S_n$ have supports in
complemented subspaces of the form $\overline{span}\{e_j:j\in I_n\cup(\mathbb{N}\setminus\mathbb{I})\}$, where the
sets $I_n\subset \mathbb{P}$ are defined as in the proof of Theorem \ref{Main theorem for spaces with unconditional
bases}.
\end{proof}

\begin{remark}\label{the reason why c0 is OK}
{\em It is clear that the spaces $\ell_p$ and $L^{p}$, $1<p<\infty$ satisfy the assumptions of Theorem \ref{Main
theorem for reflexive spaces}. It may not be so obvious why the space $c_0$ satisfy the assumptions of Theorem
\ref{Main theorem for spaces with unconditional bases}; let us clarify this point. If we repeat the proof of
\cite[Theorem V.1.5]{DGZ} in the particular case that $\Gamma=\N$, since all the operations that are made in this
proof are coordinate-wise monotone, we see that the $C^1$ and LUR renorming $\|\cdot\|$ that we obtain for $c_0$
has the property that
$$
\left\|\sum_{j\in\N, \, j\neq j_0} x_j e_j\right\|\leq \left\|\sum_{j\in\N}x_j e_j\right\|
$$
for every $j_0\in \N$ and every $x=(x_1, x_2, x_3, ...)\in c_0$, where $\{e_n\}$ is the canonical basis of $c_0$.
This shows that this norm $\|\cdot\|$ satisfies assumptions $(1)$ and $(2)$ of Theorem \ref{Main theorem for spaces
with unconditional bases}. On the other hand, for every infinite subset $J$ of $\N$ we have that
$\overline{\textrm{span}}\{e_j : j\in J\}$ is isomorphic to $c_0$, so it is clear that assumption $(3)$ is
satisfied as well, provided that there exists a continuous linear operator from $c_0$ onto $F$. Therefore $E=c_0$
satisfies the conclusion of Theorem \ref{Main result for classical Banach}. }
\end{remark}

The latter fact can be generalized to Banach spaces with a shrinking basis which contain copies of $c_0$.
\begin{theorem}\label{version for direct sums with c0}
Let $E$ be a Banach space that contains the space $c_0$ and admits a shrinking Schauder basis. Let $F$ be a
quotient of $E$.

Then, for every continuous mapping $f:E\to F$ and for every continuous function $\varepsilon: E\to (0, \infty)$
there exists a $C^{1}$ mapping $g:E\to F$ such that $\|f(x)-g(x)\|\leq\varepsilon(x)$ and $Dg(x):E\to F$ is a
surjective linear operator for every $x\in E$.
\end{theorem}
\begin{proof}
We will show that  $E$ satisfies the assumptions of Theorem \ref{Refinement of the theorem for spaces with
unconditional bases}.

First, by Sobczyk's Theorem \cite{Sobczyk}, $c_0$ is complemented in $E$, that is, $E$ is isomorphic to
$G\oplus c_0$, for a certain Banach space $G$. Since $c_0$ is isomorphic to $c_0\oplus c_0$, $G$ may be taken as
$E$. So, we can and will assume that
$$
E= c_0\oplus E.
$$
Let $\{e_j\}_{j\in N}$ be the canonical Schauder basis in $c_0$. Equip $c_0$ with the $C^1$ and LUR norm
$||\cdot||$ which was described in Remark \ref{the reason why c0 is OK}. That is, for every $j_0\in \N$, we have
$$||\sum^{\infty}_{j=1,j\neq j_0}\alpha_je_j||\leq ||\sum^{\infty}_{j=1}\alpha_je_j||,$$
for every $x=\sum^{\infty}_{j=1}\alpha_je_j\in c_0$.

Similarly, let  $\{d_n\}_{n\in\N}$ be a shrinking Schauder basis in $E$. Equip $E$ with a $C^1$ and LUR norm
$|\cdot|$. Define a new norm in $c_0\oplus E$, by letting
$$
|||x+y|||=\sqrt{||x||^2+|y|^2},
$$
for every $x+y\in c_0\oplus E$. This norm is $C^1$ and LUR as well. Define $\{f_k\}_{k\in\N}\subset c_0\oplus E$,
where $f_{2j-1}=e_{j}+0$ and $f_{2n}=0+d_n$ for every $j,n\in\N $. It is also easy to check that
 $\{f_k\}_{k\in\N}$ is a shrinking Schauder basis for $c_0\oplus E$.

For $x+y\in c_0\oplus E$, let $x=\sum^{\infty}_{j=1}\alpha_je_j\in c_0$ and $y=\sum^{\infty}_{n=1}\beta_nd_n\in E$
be their basis expansions. Then, writing $z_{2j-1}=\alpha_j$ and $z_{2n}=\beta_n$, we obtain the expansion of
$z=x+y=\sum^{\infty}_{k=1}z_kf_k=\sum^{\infty}_{j=1}z_{2j-1}e_j+\sum^{\infty}_{n=1}z_{2n}d_n\in c_0\oplus E$.
For every $j_0\in\N$, we have
\begin{align*}
|||\sum^{\infty}_{k=1,k\neq 2j_0}z_kf_k|||&=|||\sum^{\infty}_{j=1,j\neq
j_0}z_{2j-1}e_j+\sum^{\infty}_{n=1}z_{2n}d_n||| \leq|||\sum^{\infty}_{j=1,j\neq
j_0}\alpha_je_j+\sum^{\infty}_{n=1}\beta_nd_n|||
\leq\\
&|||\left(\sum^{\infty}_{j=1,j\neq j_0}\alpha_je_j\right)+y|||=\Big(||\sum^{\infty}_{j=1,j\neq
j_0}\alpha_je_j||^2+|y|^2\Big)^{\frac12}\leq\sqrt{||x||^2+|y|^2}=\\
& |||x+y|||=|||\sum^{\infty}_{k=1}z_kf_k|||,
\end{align*}
where in the second line we have used the fact that $||\sum^{\infty}_{j=1,j\neq j_0}\alpha_je_j||^2\leq
||\sum^{\infty}_{j=1}\alpha_je_j||^2=||x||^2$. Now, we are in a position to apply Theorem \ref{Refinement of
the theorem for spaces with unconditional bases}. Namely, let $\mathbb{I}=\{2n:\,n\in\N\}$ and
$\mathbb{P}=\{4n:\,n\in\N\}$.
\end{proof}

\begin{corollary}
Let $C(K)$ be the Banach space of continuous functions, where $K$ is a metrizable countable compactum and $F$ be a
quotient of $C(K)$.

Then, for every continuous mapping $f:C(K)\to F$ and for every continuous function $\varepsilon: C(K)\to (0,
\infty)$ there exists a $C^{\infty }$ mapping $g:E\to F$ such that $\|f(x)-g(x)\|\leq\varepsilon(x)$ and
$Dg(x):E\to F$ is a surjective linear operator for every $x\in E$.
\end{corollary}

\begin{proof} By an application of \cite[Theorem 1.4]{JRZ}, which states that a Banach space
has a shrinking basis provided its dual has a Schauder basis, we obtain that $C(K)$ has a shrinking basis (because
$C(K)^*=l_1$).  Moreover, using the fact that $c_0$ is a subspace of $C(K)$, we infer that $C(K)$ is isomorphic to
$c_0\oplus G$ for some Banach space $G$, which yields (as in the above proof) that $C(K)$ is isomorphic to
$c_0\oplus C(K)$. Hence, by Theorem \ref{version for direct sums with c0}, the $C^1$ version of our assertion
holds. The $C^\infty$ version requires the fact that $C(K)$ has an equivalent $C^\infty$ norm, which is due to
Haydon \cite{Hay}, and Proposition \ref{C1 approximation is enough}.
\end{proof}

For more information about the spaces $C(K)$ we refer the reader to \cite{Ros}. The space $C(K)$ is an
example of {\it isometric} predual of $\ell_1$ (meaning a Banach space $E$ with an equivalent norm
$\|\cdot\|$ such that the dual $(E^{*}, \|\cdot\|^{*})$ is isometric to $\ell_1$). The class of isomorphic predual
spaces for $\ell_1$ is larger that the class of isometric predual spaces (the space constructed by
Bourgain and Delbaen \cite{BouDel} is such an example), which in turn is smaller than the class of $C(K)$ spaces
for metrizable countable compactum $K$, see \cite{BenLin}.
\medskip

\begin{remark}
{\em Since every isometric predual space $E$ of $\ell_1$ contains $c_0$ (see for instance \cite[Corollary
1]{Zip}) and admits an equivalent real-analytic norm \cite[Corollary 3.3]{DFH}, the above corollary is valid for
$E$. Even more, the corollary is valid for any infinite-dimensional separable Banach space $E$ which has a
shrinking basis and which admits an equivalent polyhedral norm (equivalently, with a countable James
boundary). This follows from the facts that, being polyhedral, $E$ must contain $c_0$, and that a space with a
countable James boundary admits an equivalent real-analytic norm (see \cite{DFH} or \cite[Chapter 5, section 6]{HajJoh} for reference).}
\end{remark}

\medskip

As we noted in the introduction our main results imply that continuous functions between many Banach spaces can be
arbitrarily well approximated by smooth open mappings.

\begin{remark}\label{approximation by open mappings}
Let $(E, F$) be a pair of Banach spaces with the property that for every continuous mapping $f:E\to F$ and for
every continuous function $\varepsilon:E\to (0, \infty)$ there exists a $C^k$ mapping $g:E\to F$ with no critical
points such that $\|f(x)-g(x)\|\leq\varepsilon(x)$, $x\in E$. Then the pair $(E, F)$ also has the following
property: for every continuous mapping $f:E\to F$ and for every continuous function $\varepsilon:E\to (0, \infty)$
there exists an open mapping $g:E\to F$ of class $C^k$ such that $\|f(x)-g(x)\|\leq\varepsilon(x)$, $x\in E$.
\end{remark}
This follows trivially from \cite[Theorem XV.3.5]{Lang}. Recall that $g:E\to F$ is said to be open if for every
open subset $U$ of $E$ we have that $g(U)$ is open in $F$. Notice that the approximation of arbitrary continuous
maps by smooth (or even merely continuous) open maps is impossible for $E=\R^n$: for instance, if $E=\R^n$, $F=\R$,
$f(x)=e^{-\|x\|^2}$, $\varepsilon(x)=1/3$, every continuous function $g$ which $\varepsilon$-approximates $f$ must
attain a global maximum in $\R^n$, hence $g(\R^n)$ is not open in $\R$.

\medskip

\begin{example}\label{the main theorem cannot be improved so as to get approximations whose derivatives are isomorphisms}
{\em In view of Theorem \ref{Main result for separable Hilbert} it is perhaps natural to ask whether in the case
$E=F$ one can get $C^{\infty}$ approximations $g:E\to E$ such that $Dg(x):E\to E$ is a linear isomorphism for every
$x\in E$. This is not possible, as the following example shows.

Let $f:\ell_{2}\to\ell_{2}$ be defined by
$$
f\left( \sum_{n=1}^{\infty}x_n e_n\right)= \left( \sum_{n=1}^{\infty}|x_n| e_n\right),
$$
where $\{e_n\}_{n\in\N}$ denotes the usual basis of $\ell_2$ (that is, $e_1=(1, 0, 0, ...)$, $e_2=(0, 1, 0, ...)$,
etc). Assume that there exists $g\in C^{\infty}(E, E)$ such that $Dg(x):\ell_2\to\ell_2$ is an isomorphism and
$\|f(x)-g(x)\|\leq 1/3$ for every $x\in \ell_2$. Consider the projection $P_1 :\ell_2\to\R$ given by $P_1(x)=x_1$,
and the function $g_1=P_1\circ g$. Since $Dg(x)$ is an isomorphism for every $x$, we must have $Dg_1(x)=P_1\circ
Dg(x)=Dg(x)(e_1)\neq 0$ for every $x\in E$, and in particular, considering the curve $\gamma_{1}(t)=te_{1}$,
$t\in\R$, and the function
$$
\theta(t):=g_1(\gamma_1(t)), \,\,\, t\in\R,
$$
we must have
\begin{equation}\label{the derivative of theta does not vanish}
\theta'(t)=Dg_1(\gamma_1(t))(e_1)\neq 0
\end{equation}
for all $t\in\R$. However,
$$
| P_1(g(\gamma_1(t)))-P_1(f(\gamma_1(t)))|\leq \|g(\gamma_1(t))-f(\gamma_1(t))\|\leq 1/3,
$$
hence
$$
\theta(1)=P_{1}(g(\gamma_1(1)))\geq 1 - 1/3=2/3,
$$
and similarly
$$
\theta(-1) \geq 2/3 >1/3\geq \theta(0).
$$
Thus $\theta$ must attain a minimum at some point $t_0$ of the interval $(-1, 1)$, which implies that
$\theta'(t_0)=0$ and contradicts \eqref{the derivative of theta does not vanish}. }
\end{example}

\medskip

{\bf Proof of Theorem \ref{Main result for separable Hilbert manifolds}.} As said in the introduction, by the
results of \cite{EE, Kuiper}, it is sufficient to show Theorem \ref{Main result for separable Hilbert manifolds}
for functions $f:U\to V$, where $U\subset E$ and $V\subset F$ are open subsets of two separable Hilbert spaces $E,
F$, respectively. Observe that we can assume $V=F$. Indeed, if $f:U\to V\subset F$, $\varepsilon: U\to (0, \infty)$
are continuous functions then, by taking $\widetilde{\varepsilon}(x)=\frac{1}{2}\min\{\varepsilon(x),
\textrm{dist}(f(x),  F\setminus V)\}$, if we are able to $\widetilde{\varepsilon}$-approximate $f:U\to F$ by a
smooth function $g:U\to F$  with no critical points, then we also have that $\|g(x)-f(x)\|< \textrm{dist}(f(x),
F\setminus V)$, which implies that $g(x)\in V$ for every $x\in U$; that is, we really have $g:U\to V$. On the other
hand, showing the result for $f:U\to F$ is not more difficult than proving it in the case $U=E$ (though it does
encumber the notation). For example, it requires a version of the extractibility fact (a counterpart of Theorem
\ref{final extractibility theorem rough version}) where the whole space $E$, its closed subset $X$, and an open
cover $\mathcal{G}$ of $E$ must be replaced with an open subset $U$ (of $E$), a closed subset of $U$, and an open
cover of $U$, respectively. Such a fact can be proved by mimicking the technique of the proof of Theorem \ref{final
extractibility theorem rough version}); one just has to make some easy adjustments in the appropriate places. We
leave the details to the interested reader.

\medskip
Throughout the paper the ``limiting'' function $\varepsilon(x)$ is assumed to be positive. The following
remark explains what can be said if we merely require that $\varepsilon(x)\ge0$.

\begin{remark}
{\em Let $H$ be a separable, infinite-dimensional Hilbert space and $f:H\to H$ be a continuous mapping. Then, for every
continuous function $\varepsilon:H\to[0,\infty)$, there exists a continuous mapping $g:H\to H$ such that the
restriction $g_{|_{H\setminus \varepsilon^{-1}(0)}}$ is $C^\infty$ smooth and has no critical points, and
$\|f(x)-g(x)\|\le\varepsilon(x)$ for every $x\in H$ (hence, $f(x)=g(x)$ provided $\varepsilon(x)=0$). This a
consequence of Theorem \ref{Main result for separable Hilbert manifolds} applied to $U=H\setminus
\varepsilon^{-1}(0)$ and $\varepsilon_{|_U}$.}
\end{remark}

Let us conclude this paper with the proof of Proposition \ref{C1 approximation is enough}.

\medskip

{\bf Proof of Proposition \ref{C1 approximation is enough}.} Let $f:E\to F$ and $\varepsilon:E\to (0, \infty)$ be
continuous. By assumption $(1)$ there exists a $C^1$ function $\varphi:E\to F$ without critical points so that
$$
\|f(x)-\varphi(x)\|\leq\varepsilon(x)/2.
$$
It is well known that the set of continuous linear surjections from a Banach space $E$ onto a Banach space $F$ is
open; see \cite[Theorem XV.3.4]{Lang} for instance. Therefore, for each $x\in E$ there exists $r_x>0$ such that if
$S:E\to F$ is a bounded linear operator then
\begin{equation}\label{If S is close to Dvarphiy then S is surjective}
\|S-D\varphi(x)\|<2r_x \implies \, S \textrm{ is surjective}.
\end{equation}
By continuity of $D\varphi$, for every $x$ we may find a number $s_x\in (0, r_x)$ such that if $y\in B(x, s_x)$
then
$$
\|D\varphi(y)-D\varphi(x)\|<r_x.
$$
Since $E$ is separable, we can extract a countable subcovering
$$
E=\bigcup_{n=1}^{\infty}B(x_n, s_n),
$$
where $s_n :=s_{x_{n}}$. Let us also denote $r_n:= r_{x_{n}}$, and define $\eta:E\to (0,1)$ by
$$
\eta(y)=\min\left\{\frac{\varepsilon (y)}{2}, \, \sum_{n=1}^{\infty}\frac{s_n}{2}\psi_{n}(y)\right\},
$$
where $\{\psi_n\}$ is a partition of unity such that the open support of $\psi_n$ is contained in $B(x_n, s_n)$.
Now we may apply assumption $(2)$ to find a $C^k$ function $g:E\to F$  such that
$$\|\varphi(y)-g(y)\| \leq \eta(y), \textrm{ and } \|D\varphi(y)-Dg(y)\|\leq\eta(y)
$$
for all $y\in E$. Then for every $y\in E$ there exists $n=n_y\in\N$ such that $y\in B(x_n, s_n)$ and $\eta(y)\leq
s_n/2$. It follows that $\|Dg(y)-D\varphi(y)\|\leq s_n/2<r_n$ and $\|D\varphi(y)-D\varphi(x_n)\|<r_n$, hence
$\|Dg(y)-D\varphi(x_n)\|<2r_n$, and according to \eqref{If S is close to Dvarphiy then S is surjective} this
implies that $Dg(y)$ is surjective. This shows that $g$ has no critical points. On the other hand, since
$\eta\leq\varepsilon/2$, it is clear that
$$
\|f(y)-g(y)\|\leq \|f(y)-\varphi(y)\|+\|\varphi(y)-g(y)\|\leq\varepsilon(y)/2+\varepsilon(y)/2=\varepsilon(y),
$$
so $g$ also approximates $f$ as required. \qed

\medskip


\end{document}